\documentclass[12pt, oneside]{amsart}
\usepackage[margin=1in]{geometry}

\usepackage{amssymb}
\usepackage{tikz}
\usepackage{graphicx}

\usepackage{soul}
\usepackage{hyperref}
\usepackage{tikz-cd}

\newtheorem{theorem}{Theorem}[section]
\newtheorem{lemma}[theorem]{Lemma}

\theoremstyle{definition}
\newtheorem{definition}[theorem]{Definition}

\theoremstyle{remark}

\numberwithin{equation}{section}

\newcommand{\GL}{\text{GL}}
\newcommand{\GLp}{\text{GL}^{+}}
\newcommand{\SL}{\text{SL}}
\newcommand{\Aff}{\mathrm{Aff}}
\newcommand{\Affp}{\mathrm{Aff}^{+}}

\newcommand{\MCG}{\text{MCG}^{+}}

\newcommand{\R}{\mathbb{R}}
\newcommand{\Q}{\mathbb{Q}}
\newcommand{\C}{\mathbb{C}}
\newcommand{\Z}{\mathbb{Z}}
\newcommand{\A}{\mathbb{A}}

\newcommand{\id}{\text{id}}

\newcommand{\lra}{\longrightarrow}
\newcommand{\del}{\partial}
\newcommand{\CDT}{\mathcal{CD}_{\Affp} T^{2}}
\newcommand{\MCQ}{\mathcal{MCQ}_{\Affp} T^{2}}
\newcommand{\fC}{\mathcal{C}}
\newcommand{\fD}{\mathcal{D}}
\newcommand{\fF}{\mathcal{F}}
\newcommand{\fDnn}{\fD_{\omega_{k}}^{\geq}}
\newcommand{\fDnp}{\fD_{\omega_{k}}^{\leq}}
\newcommand{\fDnnK}{\fD_{\omega_{K}}^{\geq}}

\newcommand{\fDnno}{\fD_{\omega_{T}}^{\geq}}
\newcommand{\fDnpo}{\fD_{\omega_{T}}^{\leq}}
\newcommand{\fDnnoK}{\fD_{\omega_{T_{K}}}^{\geq}}
\newcommand{\fDnpoK}{\fD_{\omega_{T_{K}}}^{\leq}}
\newcommand{\fDpos}{\fD^{+}}
\newcommand{\fDneg}{\fD^{-}}
\newcommand{\fDneu}{\fD^{0}}

\definecolor{oorange}{RGB}{230,159,0}
\definecolor{oskyblue}{RGB}{86,180,233}
\definecolor{oblugreen}{RGB}{0,158,115}
\definecolor{oyellow}{RGB}{240,228,66}
\definecolor{oblue}{RGB}{0,114,178}
\definecolor{overm}{RGB}{213,94,0}
\definecolor{oredpur}{RGB}{204,121,167}
\definecolor{oblack}{RGB}{0,0,0}

\begin{document}

\title[Symplectic Tiling Billiards on Affine Tori]{Symplectic Tiling Billiards on Complete Affine Tori}

\author{Charles Daly}
\address{Max-Planck-Institut f\"{u}r Mathematik in den Naturwissenschaften, Leipzig, Germany}
\curraddr{}
\email{charles.daly@mis.mpg.de}
\thanks{Both authors express their gratitude to the Max-Planck-Institut f\"{u}r Mathematik in den Naturwissenschaften for facilitating this research and Richard Schwartz for introducing the authors to the subject along with his helpful insights about the existence of convex quadrilateral tilings.  The first author would like to thank both Max Riestenberg and Arielle Leitner for their insightful discussions on affine structures on the torus.  The first author would also like to thank Magali Jay for her thorough explanation on how tiling billiards relate to interval exchange transformations and other fundamental results of the field.}

\author{Fabian Lander}
\address{Max-Planck-Institut f\"{u}r Mathematik in den Naturwissenschaften, Leipzig, Germany}
\curraddr{}
\email{fabian.lander@mis.mpg.de}
\thanks{}


\date{\today}

\dedicatory{%
  {Dedicated in honor of Richard Evan Schwartz on the occasion of his birthday.}\\[2em]
  \raggedleft
  \footnotesize \textit{``Look Shavey, it's like billiards!''}\\[0.5em]
  \footnotesize --- Rich teaching our cat math with a laser pointer
}

\begin{abstract}
In 2023, Richard Schwartz introduced a new dynamical system which is a marriage of two types of familiar billiards, tiling billiards and symplectic billiards.  In this paper we investigate this dynamical system played on tilings of the plane which arise from non-Euclidean geometries on the torus.  We review the affine analogue of the flat conformal structures on the torus through the work of Oliver Baues and William Goldman, and define an open subset of this deformation space corresponding to markings of complete affine tilings of the plane.  We make this definition precise, and provide algebraic conditions on the symmetries of the tiling to define it.  We then analyze the dynamics of symplectic tiling billiards played on these types of tilings and investigate the long-term dynamics of the system to prove a stability result concerning divergent trajectories.  The divergence is defined in terms of geometric invariants arising from the tiling symmetry group.  We argue that in some sense this divergence is a consequence of the tiles of a non-Euclidean tiling becoming `thin' as one moves far away in the tiling.  To do so we introduce a notion of thinness that is well adapted to the non-Euclidean affine tilings.  
\end{abstract}

\maketitle

\section{Introduction}\label{sec:intro}
Symplectic Tiling Billiards is a recently developed dynamical system introduced by Richard Schwartz in 2023 in his work \emph{Symplectic Tiling Billiards, Planar Linkages, and Hyperbolic Geometry} \cite{Schwartz2025Symplectic}.  In this work, he introduces the system inspired by a union of two other dynamical systems: tiling billiards and symplectic billiards.  Tiling billiards studies the dynamics of typical billiards played on a convex tiling of the plane $T$ where the trajectories go from tile to tile via the rule that the angle of incidence is equal to the angle of refraction.  Specifically, pick a pair of points $(x,y) \in (\del T)^{2}$, on the boundary of the tiling $\del T$, and draw the line segment through $x$ to $y$.  Draw another line starting at $y$ where the angle of incidence is equal to the angle of refraction as in Figure \ref{fig:tilingbil}.  This new line will hit another point in the boundary of the tiling which we call $z \in \del T$, and this produces a new pair of points $(y,z) \in (\del T)^{2}$.  We iterate this process indefinitely unless we hit a vertex of the tiling.  This system of billiards was originally introduced by Davis, DiPietro, Rustad, and St. Laurent in 2015, and since then has enjoyed consistent interest and research where the principal results address the existence of periodic trajectories \cite{Davis2018Negative}.  In the original work where the system was introduced, Davis et al. showed the dynamical system played on certain tilings admits an abundance of bounded and unbounded periodic trajectories.  In Corollary 4.2 and Theorem 4.10 of this work they showed every tiling which arises from the subdivision of the plane by a finite number of non-parallel lines where the resulting tiling consists of congruent triangles with 6 triangles meeting at each vertex admits a periodic trajectory, and, almost all admit a period $10$-trajectory around $2$-vertices.  In a follow up work by Davis and Hooper, they show that for tilings from the subdivision process just described where the resulting polygons are regular hexagons and triangles, called trihexagonal tilings, that every bounded trajectory defined for all time is periodic [Theorem 1.8(a)] \cite{DavisHooper2018}.  
\\
\\
Since the original work by Davis et al., the dynamics in the triangular tiling case described above seem to be well understood.  The combined works of Hubert and Paris-Romaskevich and Paris-Romaskevich show every trajectory falls into a dichotomy of either periodic or linearly escaping, and the possible periods of closed trajectories must be of the form $(4n+2)$.  In the latter work, Paris-Romaskevich proves the `Tree-Conjecture' yielding a geometric association between closed periodic orbits and trees in the triangular tiling, see [Theorem 1] and [Theorem 3] of the respective works \cite{HubertParisRomaskevich2022} \cite{ParisRomaskevich2019}.  More recently, Jay showed that the dynamical system appropriately adapted to the context of polygons inscribed in a half circle admits open sets of trajectories which escape to infinity with unbounded distance from an asymptotic direction of divergence [Theorem A] \cite{Jay2025}.  
\\
\begin{figure}
\includegraphics[scale=0.5]{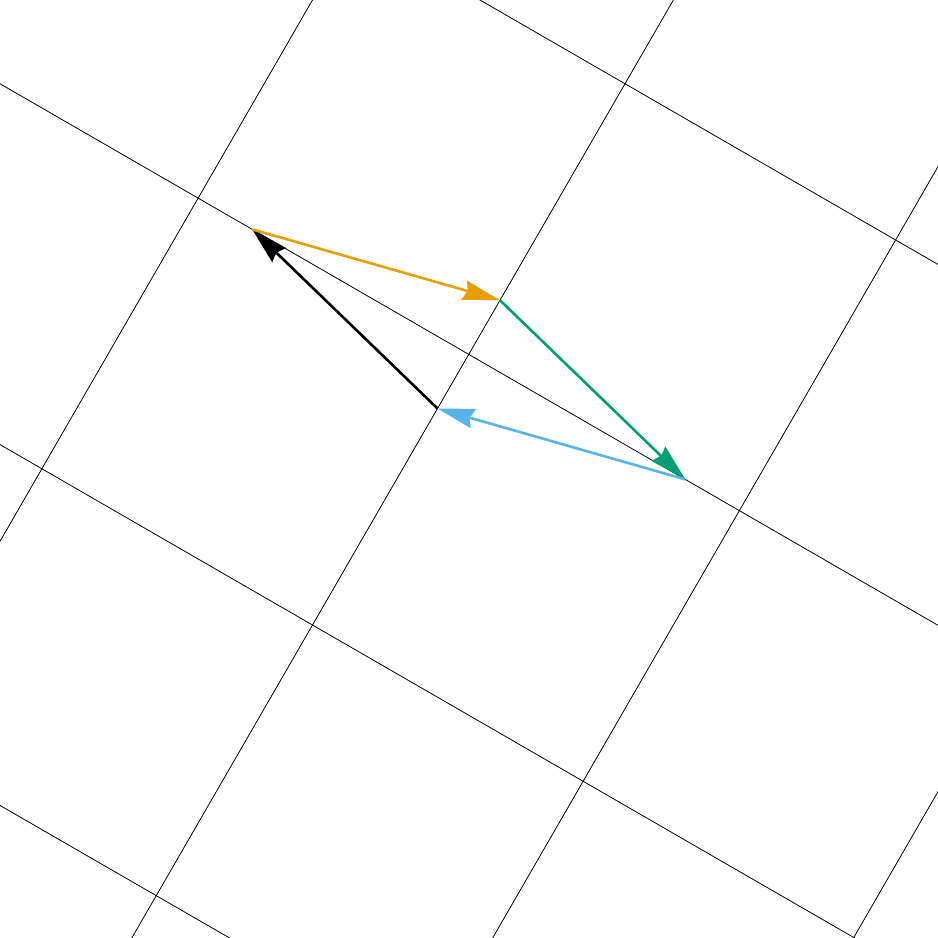}
\caption{A periodic trajectory in the tiling billiards rule.  The points $x,y$ which are the initial and terminal points of the black arrow respectively determine the point $z$ which is the terminal point of the orange arrow.  Then the points $y,z$ determine the terminal point $w$ of the green arrow, and the points $z,w$ determine the point $x$ once again.  Periodicity and drift periodicity are the only two possible long term behaviors for this type of tiling.  
}\label{fig:tilingbil}
\end{figure}

Symplectic billiards on the other hand is a billiard game defined on the boundary of a compact convex body $C$ of the plane where the rule defining how to get from one pair of points on the boundary to the next is defined in terms of \emph{parallelism}.  Specifically, pick a pair of points $(x, y) \in (\del C)^{2}$ and let $L_{y}$ be the tangent to $\del C$ at $y$.  The line through $x$ parallel to $L_{y}$ meets $\del C$ in one further point $z$, i.e. the chord $xz$ is parallel to the tangent at $y$.  The trajectory continues from $y$ to $z$, giving the new pair $(y,z) \in \del C^{2}$ which we iterate through unless we reach an instance where the translated tangent line fails to intersect $\del C$ at another point.  Figure \ref{fig:symptil} below illustrates this rule.  This system was originally introduced by Albers and Tabachnikov in 2017, wherein they show the corresponding map on phase space is a monotone area-preserving twist map on the phase space of the system, and, calculate the area of phase space in terms of the billiard table [Theorem 1] \cite{ALBERS2018822}.  In the same work, the authors also showed two possible extremes of the existence of invariant curves in the phase space.  They showed if the boundary $\del C$ has a point where the curvature vanishes, then the system admits no such invariant curves in the phase space, whereas if $\del C$ has everywhere positive curvature, for example an ellipse, then there are an abundance of invariant curves in phase space, [Theorem 2].\\
\\
In more recent work by Albers, Lander, and Westermann  \cite{albers2026symplecticbilliardspairspolygons}, the authors adapted this game to a pair of polygons $P,Q$.  The rule determining how to go from one pair of points in $\del P \times \del Q$ is adapted from the symplectic case above but where the trajectory directions through a point in the boundary are determined by the tangent line through the point on the \emph{other} boundary.  This is the non-tiling context of symplectic tiling billiards which we define in Section \ref{sec:stb}.  In this context, the aforementioned authors were able to show that every orbit is periodic under the hypothesis that the set of critical points of this system is finite [Theorem 1.1].  In the same work, they were also able to produce pairs of convex polygons which admit no periodic orbits at all [Theorem 1.4].  
\\
\begin{figure}
\includegraphics[scale=0.5]{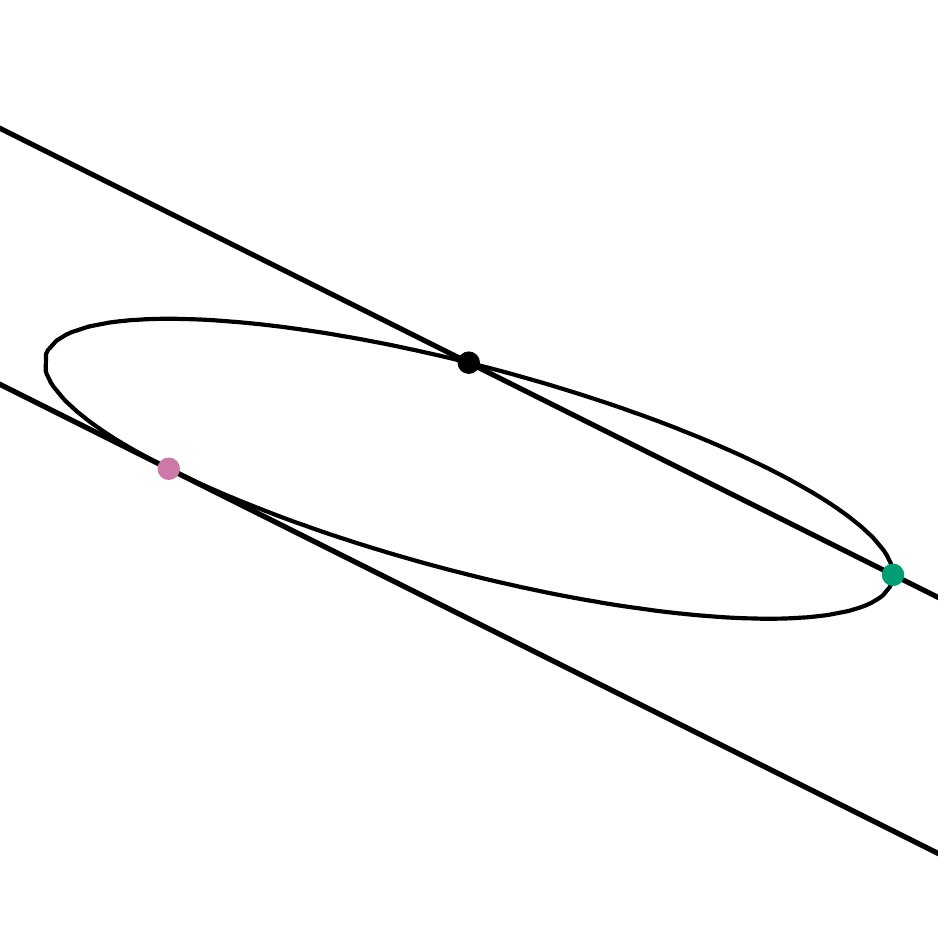}
\caption{A single iteration of the symplectic tiling billiard rule on the convex unit ball $C$ defined by the quadratic form $\frac{1}{4}x^{2} + xy + 2y^{2}$.  The points $x,y \in \del C$ are drawn in black and red respectively.  The tangent line $L_{y}$ through $y$ is moved to $x$ which will generically hit another unique point $z \in \del C$.  We illustrate this point here in green $z$.  
}\label{fig:symptil}
\end{figure}

In 2023, Schwartz introduced a combination of tiling billiards and symplectic billiards which he coined \emph{symplectic tiling billiards}.  To play this game one needs to start with two tilings of the plane, call them $(T_{A},T_{B})$.  For our purposes, these tilings are always polygonal which we will later assume to be tiled by quadrilaterals.  The actual rule to define the game is somewhat technical, but loosely speaking the game takes a pair of points $(p_{1},q_{1}) \in \del T_{A} \times \del T_{B}$, where each point is equipped with a direction, and produces another pair of points $(p_{2},q_{2}) \in \del T_{A} \times \del T_{B}$ and directions in such a way that the trajectories in a single tiling are dependent on both tilings, and, the rule producing a new pair of points is defined entirely in terms of affine invariants such as lines and parallelism as in the symplectic billiards case.  A formal definition can be found in Section \ref{sec:stb}.  One benefit of having the rules defined in terms of affine geometry is that the dynamics are then affinely invariant.  A bit more precisely, that is to say if we understand the dynamics for a fixed pair of tilings $(T_{A},T_{B})$, then we understand the dynamics for any pair of tilings $(T_{A}',T_{B}')$ where $T_{A}'$ and $T_{B}'$ are the tilings obtained from $T_{A}$ and $T_{B}$ after a fixed affine transformation.  Thus affine geometry is the natural context to better understand this dynamical system.  This is in contrast with tiling or Birkhoff billiards where the natural context is Euclidean or similarity geometry.\\
\\
Because symplectic tiling billiards requires tilings of the plane to play the game, it is natural to ask which tilings to consider.  Because of our observation that the dynamics of the game are affinely invariant, a natural context to consider is tilings that arise from \emph{complete affine tori}.  These are affine structures on the torus, as defined in Section \ref{ssec:afgeo}.  Theorem \ref{thm:goodtiles} shows that these complete affine tori give rise to `good' pairs of tilings adapted to this dynamical system, and, we classify the collection of such \emph{marked} complete affine tori which admit these tilings in Theorem \ref{thm:marked_tiles} in terms of algebraic inequalities.  Figure \ref{fig:typicalmarked} illustrates the generic sort of tiling we have in mind.  After establishing that there is a rich space of tilings to play the game on, we then analyze the long-term dynamics of this system in Section \ref{sec:divorbs}.  To do so we introduce a notion of thinness relative to another quadrilateral as in Definition \ref{def:pthin} and analyze the shapes of tiles as one goes `far out' in the tiling as in Lemma \ref{lem:getthin}.  Using these ideas and applying the context to when our tilings admit a translational symmetry, we produce pairs of divergent trajectories in Theorem \ref{thm:eucdiv} and Theorem \ref{thm:noneucdiv} where divergence is measured in terms of a geometric invariant of the tiling.  Experimentations are illustrated to both inform and support our claims as in Figure \ref{fig:symtil7}.  
\\
\begin{figure}
\includegraphics[scale=0.5]{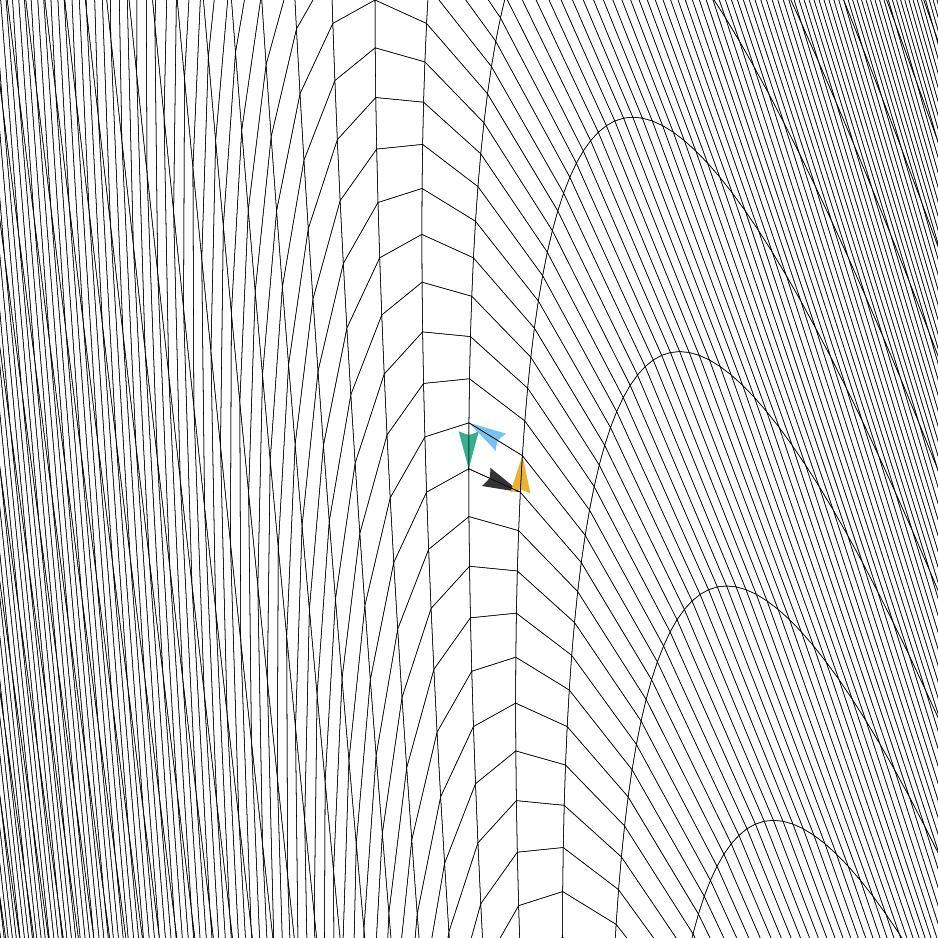}
\caption{A tiling of the plane by convex quadrilaterals where the tiling admits $\Z^{2}$-symmetry by affine transformations.  The arrows of this tiling correspond to a \emph{marking} as in Definition \ref{def:mark_tile}.  This tiling is the developing image of a complete rational affine torus with period parameters $k = (1,2/10)$ as defined in Equation \ref{eq:per_pair}.    
}\label{fig:typicalmarked}
\end{figure}

This paper is organized into four sections.  Section \ref{sec:intro} is background and motivation.  Section \ref{sec:stb} addresses the dynamical system, establishes some terminology we will regularly use throughout the work, and provides an overview of some results by Schwartz's original work \cite{Schwartz2025Symplectic}.  Section \ref{sec:cat} reviews some basics of affine planar geometry and covers the classification of complete affine structures on the torus as carried out by Oliver Baues originally in 1999 and again by Oliver Baues and William Goldman in 2005 \cite{Baues1999Gluing} \cite{BauesGoldman2005}.  In this section we adapt their theory and define a subspace of the space of marked complete affine tori which we call the \emph{marked convex complete tilings} denoted by $\MCQ$.  This space can be thought of as the space of counter-clockwise oriented quadrilaterals $Q$ with a preferred oriented edge, up to affine equivalence, that tile the full plane by commuting affine symmetries taking opposite edges to edges.  We describe this space algebraically in terms of Baues and Goldman's parametrization of marked complete affine tori.  We then prove that for any pair of points in this space $\MCQ$ we can find tilings which represent them such that the tilings are `strongly' transverse in some sense as in Theorem \ref{thm:goodtiles} and thus well suited to play the dynamical system on.  In Section \ref{sec:divorbs} we prove dynamical results about symplectic tiling billiards played on pairs of tilings as constructed in Section \ref{sec:cat}.  We define notions of divergence and thinness that are well adapted to our context in Definition \ref{def:divpath} and Definition \ref{def:pthin}.  We prove that this notion of thinness is preserved along certain trajectories of the game under the assumption our tilings admit a translational symmetry, and find tiles that are sufficiently `far away', so that these tiles admit divergent trajectories in Theorem \ref{thm:eucdiv} and Theorem \ref{thm:noneucdiv}.  Moreover we show this condition is stable for such pairs of tiles.


\section{Symplectic Tiling Billiards}\label{sec:stb}
It is the purpose of this section to provide a formal definition of symplectic tiling billiards and address some of the known results of this dynamical system.  We conclude the section with Lemma \ref{lem:gen_good} which gives us an ample phase space to play the game on for our types of tilings.  Let us fix some notation.  By a tiling of the plane $T$, we mean an infinite collection of finite area non-overlapping polygons in $\R^{2}$ whose union is all of the plane, where every edge of the tiling meets exactly two polygons, and, every vertex meets only finitely many polygons.  Because our tiling consists of a union of polygons, notions of vertices, edges, and interiors of polygons all make sense.  The boundary of the tiling denoted by $\del T$, is the union of all the edges.  For a fixed tile $P$, we denote by $\del P$ the union of its edges and we say a point $p \in \del P$ is an \emph{interior point of} $\del T$ if it is not a vertex.  By definition, every interior point of an edge is contained in exactly two polygons.  As the edges of the tiling will play a considerable role in this work we introduce the circle of directions and the \emph{edge directions} of a tiling $T$.

\begin{definition}\label{def:cded}
The \emph{circle of directions}, $\fD$, is the space of all vectors in $\R^{2}\setminus 0$ up to positive scaling.  That is, two vectors $0 \neq v,w$ determine the same direction if and only if $v = \lambda w$ for some $\lambda > 0$.  We denote the direction determined by a vector $0 \neq v \in \R^{2}$ by $[v]$.  
\\
\\
The set of \emph{edge directions} of a tiling $T$, denoted by $\fD_{T}$, is defined to be the collection of all positive and negative directions of vectors representing edges of tiling $T$.  That is to say $\fD_{T}$ is the subset of all $[v] \in \fD$ where $v$ is parallel to an edge of a tiling $T$.  
\end{definition}

Having established these preliminaries, we now describe the rules of the original formulation of the dynamical system.  Begin with a pair of tilings of the plane $T_{A}$ and $T_{B}$.  Following Schwartz, define a \emph{particle} as a pair $(p,v) \in \del T_{A} \times (\R^{2}\setminus 0)$ where $p$ is an interior point of an edge of $\del T_{A}$ and $v$ is a vector transverse to the edge of $\del T_{A}$ containing $p$ which we call the direction of $p$ [Section 2.1] \cite{Schwartz2025Symplectic}.  We need to consider particles instead of simply points because otherwise there is ambiguity as to which tile to move into.  Abusively we denote the particle by $p$, but keep in mind it comes with a direction.  
\\
\\
We pick a pair of particles $(p_{1},q_{1})$ and describe a rule to produce another pair of particles.  First draw the line through $p_{1}$ parallel to $L_{q_{1}}$ where $L_{q_{1}}$ is the line through $q_{1}$ parallel to the edge of the tiling of $T_{B}$ containing it.  This line has two possible directions to go along from $p_{1}$, and we choose the one the particle is pointed towards.  This will generically hit another point in $\del T_{A}$ which we label by $p_{2}$ and let the new direction of $p_{2}$ equal the direction of the ray just travelled along.  \\ 
\\
Next draw the line through $q_{1}$ parallel to $L_{p_{2}}$ where $L_{p_{2}}$ is the line through $p_{2}$ parallel to the edge of the tiling of $T_{A}$ containing it.  Again, we choose the direction of $L_{p_{2}}$ that $q_{1}$ is pointed towards.  Generically this will hit another point in $\del T_{B}$ which we label by $q_{2}$, and define the new direction of $q_{2}$ to be the direction of the ray we just travelled along.  
\\
\\
This produces our new pair of particles $(p_{2},q_{2})$ and we iterate this process.  Figure \ref{fig:stbex} illustrates this dynamical system.  There are some technical difficulties that one could run into.  For example, we are implicitly assuming that the edge containing $p_{n}$ is not parallel to the edge containing $q_{n}$.  Moreover it is possible to run into a vertex.  If either of these possibilities occur, the dynamical system breaks.  It is for this reason we generally wish to consider tilings which are \emph{transverse} in the sense that any pair of edges from $T_{A}$ and $T_{B}$ span all of the plane.  In this case one then only has to worry about whether the trajectories run through vertices.  
\\
\begin{figure}
	\begin{minipage}{0.45\textwidth}
		\includegraphics[width=\textwidth]{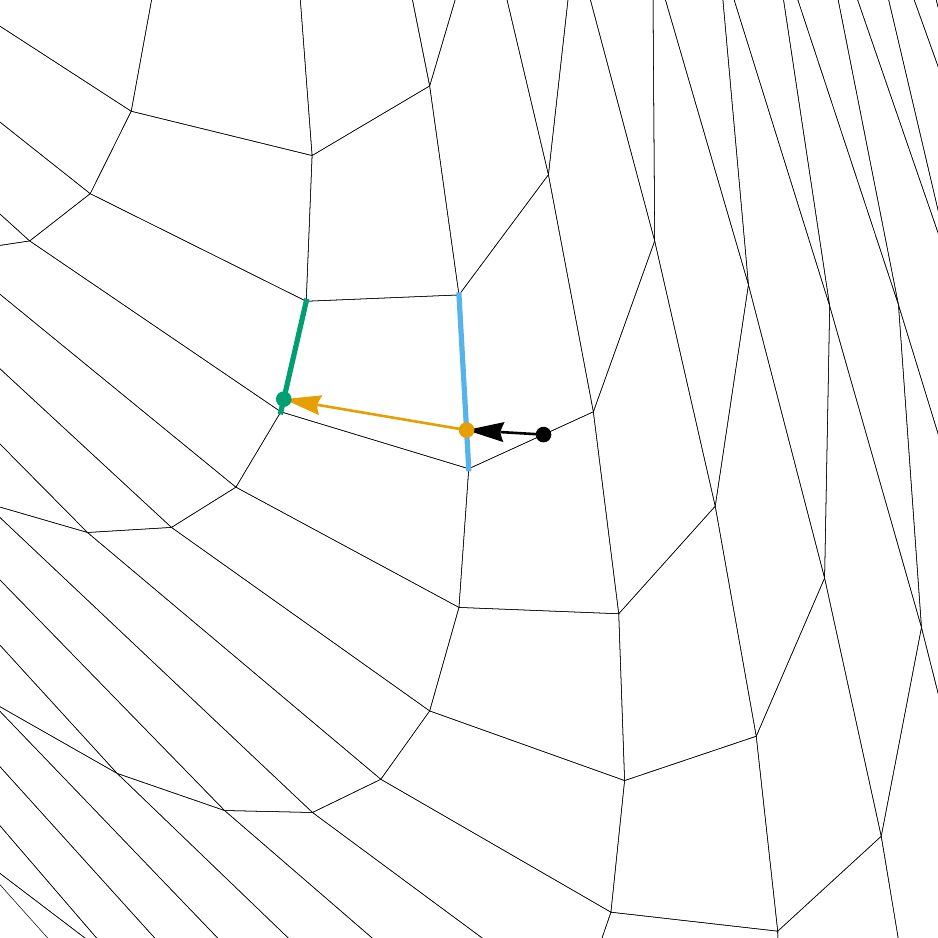}
	\end{minipage}
	\hfill
	\begin{minipage}{0.45\textwidth}
		\includegraphics[width=\textwidth]{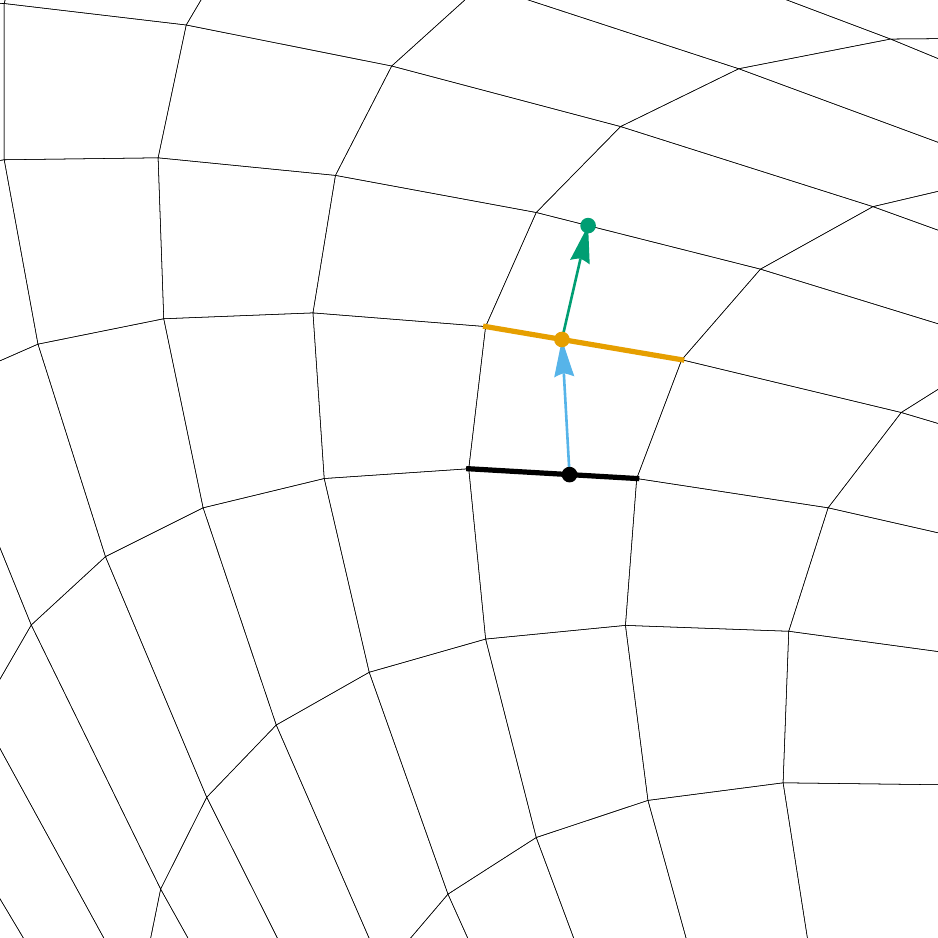}
	\end{minipage}
	\caption{An example of the rule of symplectic tiling billiards.  The particles $(p_{1},q_{1})$ are illustrated in black.  Parallel edges and trajectories are drawn in the same color.  This configuration produces the new pair of particles $(p_{2},q_{2})$ in orange.  Another iteration produces the new pair of particles $(p_{3},q_{3})$ in green.  The initial directions of $(p_{1},q_{1})$ are suppressed but inferable from the picture as they determine the directions of the first arrows.}\label{fig:stbex} 
\end{figure}

Schwartz originally considered the system in two contexts.  In the first instance, he considers the pair of tilings where $T_{A}$ is the typical Euclidean square tiling by horizontal and vertical translations and where $T_{B}$ is obtained by rotating $T_{A}$ via multiplication by the unit complex number $z_{t} = (1-t^{2},2t)/(1+t^{2})$ for rational $t$.  Through these families he was able to produce a range of examples where trajectories exhibit vastly different dynamical properties.  For example if $t = 1/3$ it appears most orbits are aperiodic, bounded, and converge to some limiting cycle, whereas if $t = 7/11$ these tilings admit divergent orbits [Figure 2.2] [Figure 2.3] \cite{Schwartz2025Symplectic}.\\
\\
In the case where $t = 7/11$ he observed a behavior in which the orbits tend to spiral around a vertex.  This sort of local picture is the second context Schwartz studied using his construction of \emph{sunbursts} which are, loosely speaking, a finite collection of half-rays emanating from the origin subdividing the plane into finitely many unbounded sectors and whose convex hull is the whole plane.  By analyzing the dynamics of symplectic tiling billiards played locally on two pairs of sunbursts $(S_{A}, S_{B})$ which satisfy a certain compatibility property, he was able to prove the existence of periodic orbits on $(S_{A}, R_{\theta}S_{B})$ for some $\theta \in S^{1}$ [Theorem 1.1].  Through the existence of these periodic orbits he produced another bijection witnessing the celebrated correspondence between similarity classes of strictly convex equilateral $N$-gons and similarity classes of strictly convex equiangular $N$-gons.  He shows his correspondence is algebraic in the appropriate sense of the term [Lemma 4.4].  This gave rise to a novel relation between this local version of symplectic tiling billiards and hyperbolic structures as constructed in Thurston's \emph{Shapes of Polyhedra} \cite{Thurston1998Shapes}.  \\
\\
We are interested in analyzing this system when the tilings are non-Euclidean, but still admit $\Z^{2}$-symmetry.  That is to say, the tilings we consider are not by a translation of a fixed parallelogram, but instead a family of convex quadrilaterals which are all affinely equivalent.  Section \ref{sec:cat} is dedicated to analyzing the space of such tilings to justify them as a natural context to play the game.  Because we will exclusively work with tilings of $\R^{2}$ by convex quadrilaterals that exhibit $\Z^{2}$-symmetry, we may unambiguously specify a tile by an integer pair $(n,m) \in \Z^{2}$ once we pick an oriented-edge of the tiling and insist the tiles are oriented counterclockwise.  The tile corresponding to $(0,0)$ is the tile whose interior lies to the left of the selected oriented-edge.  Figure \ref{fig:moving} illustrates this convention.  It is for this reason we call such a choice of an oriented edge a \emph{marking} of the tiling.  The formal Definition \ref{def:mark_tile} is carried out in more detail later, but for now, we only need to use the fact that a marking allows us to unambiguously specify every tile of $T$ by an integer pair $(n,m)$.  \\
\\
In what follows we specialize to the context where we play the game on two \emph{marked} tilings $(T_{A},T_{B})$ of $\R^{2}$ by convex quadrilaterals with $\Z^{2}$-symmetry, and instead of using pairs of particles $(p,v),(q,w)$, we may instead use pairs of point-tiles $(p,(n_{1},m_{1})) \in \del T_{A} \times \Z^{2}$ and $(q,(n_{2},m_{2})) \in \del T_{B} \times \Z^{2}$.  When drawing the line $L_{q}$ through $p$, the choice of tile $(n_{1},m_{1})$ picks out a choice of direction parallel to $L_{q}$, namely the one pointed to the interior of the tile $(n_{1},m_{1})$.  We may therefore identify configurations of the game with pairs of point-tiles $(p,(n_{1},m_{1}))$ and $(q,(n_{2},m_{2}))$.  We topologize the first coordinates $p,q$ by pairs of elements in the unit square boundary taking edges to edges and vertices to vertices.  In so doing the space of pairs of point-tiles is readily seen to be identifiable with $(S^{1}\times \Z^{2})\times (S^{1}\times \Z^{2}) \simeq T^{2}\times \Z^{4}$ which we denote by $\fC$.  This simplification is technical but facilitates a more simply stated proof of Lemma \ref{lem:gen_good} below.  We note that this proof is modeled directly from Definition 2.11 and Remark 2.12 of Albers, Lander, and Westermann \cite{albers2026symplecticbilliardspairspolygons}.  The main idea is to show that the set of configurations in $\fC$ which reach a vertex after finitely many iterations is measure zero.  One may do so by starting at a vertex and choosing directions from the edge directions of the other tiling, $\fD_{T_{A}}$ or $\fD_{T_{B}}$, to head out along.  This will define a countable collection of intersection points which we wish to remove.  Applying this argument inductively proves the result.  
%
\begin{lemma}\label{lem:gen_good}
Let $T_{A}$ and $T_{B}$ be transverse convex quadrilateral tilings of the plane which admit $\Z^{2}$-symmetry.  The subset of configurations which terminate in finitely many iterations is a measure zero subset of $\fC$.  Thus the set of configurations which extend for all iterations has full measure in $\fC$.  
\end{lemma}

\begin{proof}
For the sake of simplicity, let us assume that $T_{A}$ is the unit square tiling and $T_{B}$ is a parallelogram tiling transverse to $T_{A}$ so $\fD_{T_{A}}$ and $\fD_{T_{B}}$ are disjoint four point sets of the circle of directions $\fD$.  The arguments provided in this case carry out the same way for the more general tilings, but for the sake of readability and illustration, we restrict our attention to this case. \\
\\
Preliminarily we must remove from $\fC$ all configurations of the form $\left((p,(n_{1},m_{1})),(q,(n_{2},m_{2}))\right)$ $\in \fC$ where $p$ is a vertex.  For each fixed choice of integer pairs $((n_{1},m_{1}), (n_{2},m_{2}))$ the collection of all such $p$ is readily seen to be measure zero as it determines four circles in $T^{2}$.  Figure \ref{fig:torusvert} illustrates this.  The same is true for when $q$ is a vertex.  As the set of configurations which begin at a vertex for the fixed element of $\Z^{4}$ has measure zero, the union of all such ones over $\Z^{4}$ is also measure zero as the countable union of measure zero sets is measure zero.
\\
\begin{figure}
\begin{center}
\begin{tikzpicture}[scale=1.5]

	\def\side{4}

  	\fill[white!20] (0,0) rectangle (\side,\side);

  	\draw[thick] (0,0) rectangle (\side,\side);
	
	\draw[line width=1pt, -] (\side/4,0) -- (\side/4,\side);
	\draw[line width=1pt, -] (2*\side/4,0) -- (2*\side/4,\side);
	\draw[line width=1pt, -] (3*\side/4,0) -- (3*\side/4,\side);
	\draw[line width=1pt, -] (\side,0) -- (\side,\side);
	\draw[line width=1pt, -] (0,\side) -- (0,0);
	
	\draw[line width=1pt,overm, -] (0,\side/4) -- (\side,\side/4);
	\draw[line width=1pt,overm, -] (0,2*\side/4) -- (\side,2*\side/4);
	\draw[line width=1pt,overm, -] (0,3*\side/4) -- (\side,3*\side/4);
	\draw[line width=1pt,overm, -] (0,0) -- (\side,0);
	\draw[line width=1pt,overm, -] (\side,\side) -- (0,\side);
	
	\node[xshift=-20pt, yshift=-20pt] at (0,0) {$((n_{1},m_{1}), (n_{2},m_{2}))$};
	\node[xshift=20pt, yshift=-20pt] at (\side,0) {$S^{1}$};
	\node[xshift=-20pt, yshift=+20pt] at (0,\side) {$S^{1}$};

\end{tikzpicture}
\caption{A drawing of the subset of $\fC$ at the fixed index pairs $((n_{1},m_{1}), (n_{2},m_{2}))$.  The axes are circles corresponding to the boundary of the tiles determined by $((n_{1},m_{1}), (n_{2},m_{2}))$ in $T_{A}$ and $T_{B}$ respectively.  Here we choose parameterizations where the vertices lie along the quarters of the circles.  The black vertical circles correspond to when $p$ in the point pair $(p,q) \in T^{2}$ is a vertex whereas the red horizontal circles correspond to when $q$ is a vertex.  We must remove both from $\fC$ for every pair of integer pairs $((n_{1},m_{1}), (n_{2},m_{2}))$.  For a single pair of integer pairs, the union of these circles is visibly a measure zero set.  
}\label{fig:torusvert}
\end{center}
\end{figure}
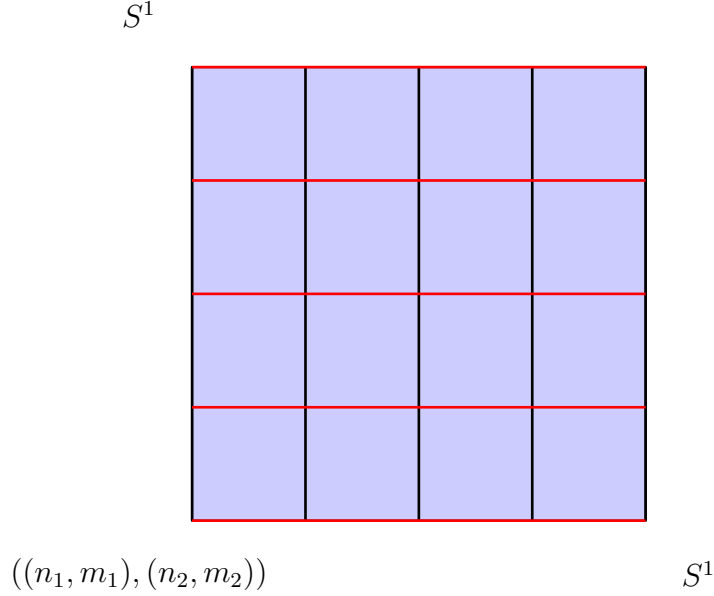

One can think of these configurations which begin at vertices as the `time zero bad set.'  We now wish to remove the `time one bad set.'  We focus our attention on $P$, the preferred unit square whose interior is to the left of the fixed oriented edge, and consider the vertex $p \in \del P$ determined by the initial point of this edge. We aim to show that all configurations of $\fC$ which reach $p$ after a single iteration form a measure zero subset.  To this end, at the point $p$ draw the four lines through $p$ parallel to directions in $\fD_{T_{B}}$.  By convexity and transversality, these line segments will intersect the $4$-tiles through $p$ at $4$ different points which we label $p_{i}$.  These points of intersection $p_{i}$ determine at most $4$ new tiles in $T_{A}$.  We now remove all configurations which begin at $p_{i}$ in one of these new tiles.   
Fix the integer pair in $\Z^{2}$ determining this tile in $T_{A}$, and consider any other tile of $T_{B}$.  By removing all configurations which begin at $p_{i}$ at this tile pair, we remove yet another circle through a torus which is a measure zero set.  This is illustrated in Figure \ref{fig:removepi} below.  We remark removing the entire circle is done in excess, and we only need to remove the quarter of the circle corresponding to the edge of the $T_{B}$ tile which sends $p_{i}$ to a vertex.  We repeat this process for each of the at most four new tiles.  
\\
\begin{figure}
\begin{center}
\begin{tikzpicture}[scale=1.5]

	\def\side{4}

  	\fill[white!20] (0,0) rectangle (\side,\side);

  	\draw[thick] (0,0) rectangle (\side,\side);
	
	\draw[line width=1pt, -] (\side/4,0) -- (\side/4,\side);
	\draw[line width=1pt, -] (2*\side/4,0) -- (2*\side/4,\side);
	\draw[line width=1pt, -] (3*\side/4,0) -- (3*\side/4,\side);
	\draw[line width=1pt, -] (\side,0) -- (\side,\side);
	\draw[line width=1pt, -] (0,\side) -- (0,0);
	
	\draw[line width=1pt,overm, -] (0,\side/4) -- (\side,\side/4);
	\draw[line width=1pt,overm, -] (0,2*\side/4) -- (\side,2*\side/4);
	\draw[line width=1pt,overm, -] (0,3*\side/4) -- (\side,3*\side/4);
	\draw[line width=1pt,overm, -] (0,0) -- (\side,0);
	\draw[line width=1pt,overm, -] (\side,\side) -- (0,\side);
	
	\draw[line width=1pt, oblugreen, -] (3*\side/5,0) -- (3*\side/5,\side);
	
	\node[xshift=-20pt, yshift=-20pt] at (0,0) {$((n_{1},m_{1}), (n_{2},m_{2}))$};
	\node[xshift=20pt, yshift=-20pt] at (\side,0) {$S^{1}$};
	\node[xshift=-20pt, yshift=+20pt] at (0,\side) {$S^{1}$};
	
	\node[xshift=0pt, yshift=-20pt] at (3*\side/5,0)  {$p_{i}$};

\end{tikzpicture}
\caption{For every fixed pair of integer pairs, we remove the line corresponding to $p_{i}$ which we illustrate in green.  
}\label{fig:removepi}
\end{center}
\end{figure}
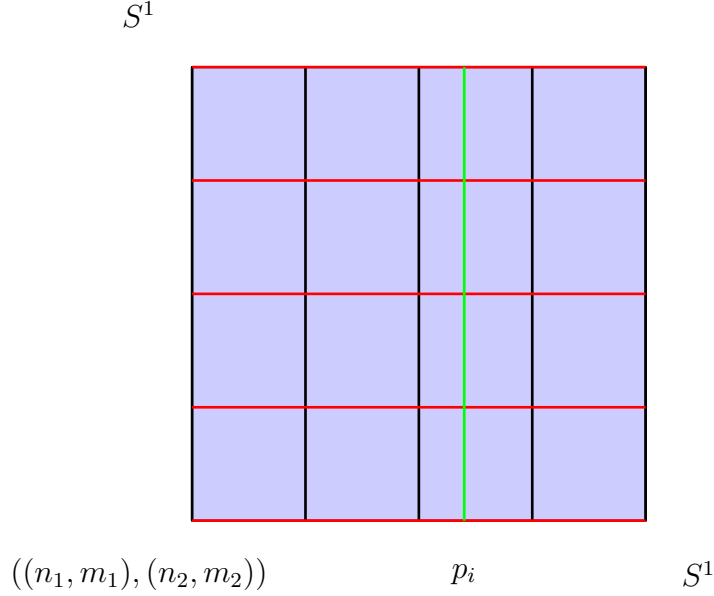

In general the collection of such $p_{i}$ is countable, as the edge set of our tilings is countable.  Nevertheless, a countable collection of circles in $T^{2}$ is still a measure zero set, and taking the union over all $\Z^{4}$ is still measure zero in $\fC$.  Thus the collection of all points which go through $p$ after a single iteration is measure zero in $\fC$.  Inductively proceed in this fashion to remove all configurations which reach $p$ in finitely many iterations, which is a countable union of measure zero sets, and hence measure zero.  As the vertex $p$ was arbitrary, this is true for any vertex of $T_{A}$, and thus the set of configurations which reach a vertex of $T_{A}$ in finitely many iterations has measure zero.  The same holds for $T_{B}$.  Hence the collection of configurations of $\fC$ which are defined for all time has full measure.  
\end{proof}


\section{Affine Geometry and Affine Tilings of the Plane}\label{sec:cat}
Having established some basics of the dynamical system, the next natural question becomes \emph{which} tilings of the plane should one consider?  Ideally one would want to play on transverse pairs of tilings to avoid issues where the edges of pairs of particles are parallel.  It is the purpose of this section to provide an abundance of such tilings which arise from complete affine structures on the torus.  These are tilings of the plane which admit $\Z^{2}$-tiling symmetries by affine automorphisms and are naturally adapted to the affine invariance of the dynamical system.     
\\
\\
The first half of this section is dedicated to reviewing notions of affine structures, completeness, and affine markings of the torus.  We give this account through the work of Baues and Goldman's classification of complete affine structures on the torus via their \emph{period parametrization} with particular attention given to the geometric invariants as in Lemma \ref{lem:caninvs}.  We define what we mean by a \emph{marked complete affine tiling} in Definition \ref{def:mark_tile} which is the context we carry out the dynamical system as in Section \ref{sec:divorbs}.  We show that the space of marked complete affine tilings forms a subspace $\MCQ$ of the space of marked complete affine structures on the torus $\CDT$, and show it is cut out by algebraic inequalities in the period parameters.  We conclude this section with Theorem \ref{thm:goodtiles} which proves we can realize any pair of elements of $\MCQ$ as tilings which are in some sense strongly transverse, and thus well suited to symplectic tiling billiards.  
\subsection{Complete Affine Tori}\label{ssec:afgeo}
We begin by recalling some basics of planar affine geometry.  Let $\A^{2}$ be the collection of all $2$-tuples of real numbers equipped with the simply transitive action of $\R^{2}$ on $\A^{2}$ by translations.  Specifically, a vector $v \in \R^{2}$ acts on a point $p \in \A^{2}$ by $vp:= p + v$ under the usual addition of tuples.  We simply write $p+v$ instead of $vp$.  The philosophy of affine geometry is to free up the origin and let every point be indistinguishable from the rest.  The typical mantra is an affine plane is a vector space where we forget the origin, or in other words, the affine plane is an $\R^{2}$-\emph{torsor}.  While we cannot naturally add points in $\A^{2}$, we can canonically subtract them to get a unique vector in $\R^{2}$.  Denote by $q-p \in \R^{2}$ the unique translation vector $v \in \R^{2}$ for which $p + v = q$.  \\
\\
We define an \emph{affine automorphism} $f: \A^{2} \lra \A^{2}$ to be a bijection which respects the $\R^{2}$-action on $\A^{2}$.  Precisely, an affine automorphism is a bijection between the affine plane for which there is a linear automorphism $L(f): \R^{2} \lra \R^{2}$ so that $f(p+v) = f(p) + L(f)(v)$ for all $p \in \A^{2}$ and $v \in \R^{2}$.  We call $L(f)$ the \emph{linear part} of $f$.  \\
\\
By picking an origin $O \in \A^{2}$, we may identify $\A^{2}$ with $\R^{2}$ by associating every point $p \in \A^{2}$ the vector $v := p-O \in \R^{2}$.  We then define the \emph{translational part} of $f$ by $T(f): = f(O)-O$.  The translational part, unlike the linear part, is dependent on the choice of origin $O$.  The linear and translational part of $f$ then determine the affine transformation, as we have for every $v \in \R^{2}$, $f(O + v) - O= L(f)(v) + T(f)$.  With the choice of an origin, we may identify the group of all affine automorphisms of $\A^{2}$ with the subgroup of $\GL(3,\R)$ given below.
\begin{equation}\label{eq:lin_tran}
\Aff(2,\R) = \bigg\{ 
\left(
\begin{array}{cc}
A & b \\
0 & 1
\end{array}
\right) \in \GL(3,\R) \, \bigg| \, A \in \GL(2,\R) \text{ and } b \in \R^{2} \bigg\}
\end{equation}
The matrix $A$ and the vector $b$ in the above equation are the linear and translational parts of the affine transformation respectively relative to an affine frame, i.e. a choice of origin $O$ and a basis for $T_{O}\A^{2} \simeq \R^{2}$.  As groups we have $\Aff(2,\R) = \GL(2,\R) \ltimes \R^{2}$ where $\GL(2,\R)$ acts on $\R^{2}$ by the birth-certificate representation.  Frequently we will denote an affine transformation by $(A,b)$, and its action on a vector $v \in \R^{2}$ is given by $Av + b$.  This turns $\R^{2}$ into a model space for affine geometry in the sense of $(G,X)$-structures where $G = \Aff(2,\R)$ and $X = \A^{2}$.  This geometry lies in between Euclidean and projective and is the geometry of \emph{parallelism}.  A bit more precisely, the standard torsion-free flat connection $\nabla$ is invariant under $\Aff(2,\R)$.  Consequently, parallel lines are sent to parallel lines, and a bit more generally, geodesics are sent to geodesics.  Here we emphasize in the absence of an affinely invariant Riemannian metric, we are still able to define geodesics as zero acceleration curves in $\R^{2}$ under the connection $\nabla$.  Locally all such geodesics look like $p + tv$ where $p, v \in \R^{2}$ and $t$ is the time parameter.  Note the form of such a curve is invariant under the affine group.  Other invariants such as collinearity, barycenters, convexity, polynomials, and more are also preserved under $\Aff(2,\R)$.   \\
\\
We will exclusively restrict our attention to the subgroup of orientation preserving affine transformations, $\Affp(2,\R) = \GLp(2,\R) \ltimes \R^{2}$ and we equip $\R^{2}$ with the typical counter-clockwise orientation.  The orientation class defined by positive scales of the unique parallel area form given by the standard determinant is invariant under this group.  We wish to locally model the geometry of $(\Affp(2,\R),\R^{2})$ on oriented surfaces.  With these definitions we are able to define an affine structure on an oriented surface.  Note that our definition is stronger than the typical one as we insist our affine structure is oriented, and, the orientation of the surface is compatible with the orientation of $\R^{2}$ via the defining geometric charts.  We do this largely to avoid having to include the word orientation-preserving in hypotheses and statements of theorems.   

\begin{definition}\label{def:affstr}
An \emph{affine structure} on a closed oriented surface $M$ is a maximal atlas of \emph{orientation preserving charts} $\phi_{\alpha} : U_{\alpha} \lra \R^{2}$ of $M$ where on each connected component of $C \subset U \cap V$, we can find an affine automorphism $g_{V,U} \in \Affp(2,\R)$ for which $\phi_{V} \circ \phi_{U}^{-1}|_{C} : \phi_{U}(C) \lra \phi_{V}(C)$ equals $g_{V,U}|_{C}$.  We say a diffeomorphism $f: M \lra N$ is an \emph{affine diffeomorphism} between affine surfaces if locally in charts $\phi: U \subset M \lra \R^{2}$ and $\psi: V \subset N \lra \R^{2}$, $\psi\circ f \circ \phi^{-1}$ is the restriction of an element of $\Affp(2,\R)$.  
\end{definition}

Having such a structure on an oriented surface is topologically restrictive.  By the work of Benz\'{e}cri, it necessitates $\chi(M) = 0$ so $M$ is homeomorphic to the torus $T^{2}$ \cite{Benzecri1959}.  The affine structure on $T^{2}$ allows us to locally model the geometry of $(\Affp(2,\R),\R^{2})$ on $T^{2}$.  To an observer living on the torus, they would be unable to distinguish small neighborhoods of the torus and the affine plane.  Most importantly, the torus comes equipped with a torsion-free flat connection $\nabla$ given by pulling back the standard connection $\nabla$ on $\R^{2}$ through the atlas charts $\phi_{U}: U \lra \R^{2}$ defining the affine structure.  The fact that the coordinate changes are locally given by affine automorphisms means that the pulled-back connection is locally well-defined and thus extends to all of the torus.  As a consequence, this means we can define geodesics on the affine torus.  For our purposes we will largely be interested in the case where the torus is \emph{complete}, i.e. the geodesics can be extended indefinitely both forwards and backwards in time.  \\
\\
Given two affine structures on $T^{2}$, it is of natural interest to ask when the affine structures are the `same.'  To make this notion precise, we introduce some terminology which is standard in the literature, but included for the sake of self-containment.  Define an \emph{affine marking of} $T^{2}$ as an orientation-preserving diffeomorphism $f: T^{2} \lra M$ where $M$ is an affine torus.  Two affine markings of $T^{2}$, $(f,M)$ and $(g,N)$, are said to be \emph{equivalent} if there exists an affine diffeomorphism $\phi: M \lra N$ which is isotopic to $g\circ f^{-1} : M \lra N$.  In other words, the diagram below commutes up to isotopy.  
\begin{equation*}\label{eq:markedeq}
\begin{tikzcd}
    T^{2} \arrow[dr, "g", swap] \arrow[r, "f"', swap] & M \arrow[d, "\phi"', swap ]  \\
     & N
\end{tikzcd}
\end{equation*}
The space of all equivalence classes of affine markings of the torus can be given a topology, which should be thought of as the affine analogue of equivalence classes of marked conformally flat structures on the torus [\S 6 Classification] \cite{Goldman2022}.  We warn the reader that in the affine context, this space is non-Hausdorff as seen by the works of Nagano-Yagi and Arrowsmith-Furness \cite{NaganoYagi1974}, \cite{ArrowsmithFurness1975}, \cite{ArrowsmithFurness1976}.  If however, we restrict our attention to the space of \emph{complete} affine markings on the torus, which we denote by $\CDT$, then this space is homeomorphic to $\R^{2}$ as was shown in Baues and Baues-Goldman \cite{Baues1999Gluing} \cite{BauesGoldman2005}.  Throughout this paper, we refer to $\CDT$ as \emph{deformation space}.  More recently, another parametrization of $\CDT$ was given through the Veronese embedding that shows this affine deformation space admits a singular smooth structure where the singular point of this parametrization corresponds to the equivalence class of marked Euclidean structures on the torus \cite{Goldman2025AffineTwistedCubic}.   \\
\\
We will largely work with a parametrization given by the `period map' defined later in Equation \ref{eq:per_pair}.  The typical means of parametrizing equivalent marked conformally flat structures on the torus yield coordinates on the upper-half plane and are defined by periods of a holomorphic one-form relative to a fixed basis of the fundamental group of the torus.  In our case, the periods of a canonical \emph{parallel} one-form define coordinates on $\CDT$.  We note in $\CDT$, the entire space of marked conformally flat structures collapses to a single point, as all parallelogram tilings are affinely equivalent.  We refer to this point as the \emph{Euclidean marking} and, under the period parametrization, it corresponds to the origin in $\R^{2}$ under the period map.  We refer to all other points in $\CDT$ as \emph{non-Euclidean}.  
\\
\\
To describe the parametrization in detail, we provide a brief account of the correspondence between marked complete affine structures on the torus and certain representations of its fundamental group into $\Affp(2,\R)$ up to conjugation which act properly discontinuously on $\R^{2}$.  We refer the reader to Sections 5.2 of Goldman for details \cite{Goldman2022}.  Fix a base point $x_{0} \in T^{2}$.  To every complete affine marking $f: T^{2} \lra M$, we can pull back the affine charts of $M$ to obtain a complete affine structure on $T^{2}$.  We then obtain a faithful representation $\rho: \pi_{1}(T^{2},x_{0}) \lra \Affp(2,\R)$ so that the image acts properly discontinuously on $\R^{2}$ by affine transformations.  This representation is the \emph{holonomy} representation.  Its definition depends on a choice of an initial affine chart about $x_{0} \in T^{2}$.  A different choice of an affine chart about $x_{0}$ changes the representation by conjugation, and thus we instead consider the holonomy representation up to conjugation.   \\
\\
If we have two equivalent complete affine markings $f: T^{2} \lra M$ and $g: T^{2} \lra N$, then by definition $g\circ f^{-1}$ is isotopic to an affine automorphism $\phi: M \lra N$.  Let us denote by $T^{2}_{M}$ and $T^{2}_{N}$ the affine structures on $T^{2}$ pulled back by $f$ and $g$ respectively.  Then by definition the composition $g^{-1}\circ \phi \circ f: T^{2}_{M} \lra T^{2}_{N}$ is an affine automorphism isotopic to the identity.  By the development theorem, this means the holonomy representations $\rho_{M}, \rho_{N}$ from $\pi_{1}(T^{2},x_{0})$ to $\Affp(2,\R)$ are equal up to conjugation.  Hence the map associating an equivalence class of complete affine markings to its holonomy representation up to conjugation is well-defined.  This map is a homeomorphism onto a component of the space of representations of $\Z^{2}$ into $\Affp(2,\R)$ up to conjugation, and, the homeomorphism type is $\R^{2}$ [Proposition 2.1 \& Theorem 5.2] \cite{BauesGoldman2005}.  
\\
\\
It therefore follows that a complete affine marking of the torus, up to equivalence, is determined by a single parameter in $\R^{2}$.  The construction Baues-Goldman used to prove this correspondence is explicit and begins with the fact that if $\R^{2}$ acts (non-Euclidean) affinely and properly on $\A^{2}$, then up to conjugation, it is a subgroup of $G_{1} \subset \Affp(2,\R)$ as defined below.
\begin{equation}\label{def:eq:unipot}
G_{1} := \bigg\{
\left(
	\begin{array}{ccc}
		1 & t & s + \frac{1}{2}t^{2} \\
		0 & 1 & t \\
		0 & 0 & 1
	\end{array}
\right) \, \bigg| \, s,t, \in \R
\bigg\} = \exp \bigg\{
\left(
	\begin{array}{ccc}
		0 & t & s  \\
		0 & 0 & t \\
		0 & 0 & 0
	\end{array}
\right) \, \bigg| \, s,t, \in \R
\bigg\} 
\end{equation}
Historically, this appears to have been known as early as the mid-1900's by Kuiper, though its proof has been generalized by other authors \cite{Kuiper1953}, \cite{ArrowsmithFurness1976}, \cite{NaganoYagi1974}, \cite{FriedGoldman1983}.  The work of Mal'cev shows that every proper affine action of $\Z^{2}$ on $\A^{2}$ must extend to a simply-transitive one on all of $\R^{2}$, and so we may assume the image of the holonomy representation of a complete affine torus lies somewhere inside of the group in Equation \ref{def:eq:unipot} \cite{Malcev1951}.  A quick inspection of $G_{1}$ yields the invariant parallel vector field generated by $\del_{x}$, its complementary invariant parallel one-form $dy$, and, the invariant parallel area form $dx\wedge dy$.  These invariants descend to the complete affine torus in question, and conventions can be imposed so that all these invariants are canonical.  This is summarized in Baues-Goldman which we cite below without proof [Lemma 4.1] \cite{BauesGoldman2005}.
\begin{lemma}\label{lem:canform}
Let $M$ be a complete non-Euclidean affine torus with connection $\nabla$.  There exists a parallel area form $d\Omega_{M}$, a parallel vector field $X_{M}$, a parallel one-form $\omega_{M}$, and a polynomial vector field $Y_{M}$ satisfying the conditions below.
\begin{equation*}
\int_{M} d\Omega_{M} = 1, \phantom{=} X_{M} = \nabla_{Y_{M}} Y_{M}, \phantom{=} d\Omega_{M}(X_{M},Y_{M}) = 1, \phantom{=} \omega_{M} = \iota_{X_{M}}(d\Omega_{M})
\end{equation*}
Here $\iota$ denotes interior multiplication.  The area form, vector field, and one-form are all unique, whereas the polynomial vector field is unique up to addition of a constant multiple of $X_{M}$.  
\end{lemma}
The above lemma gives us canonical geometric objects on our complete affine torus which we will use repeatedly throughout this paper.  However, we emphasize that these objects are defined on the complete affine torus and \emph{not} the marking.  In fact, for a fixed complete affine torus $M$, we use the parallel one-form $\omega_{M}$ to define the \emph{period parametrization} of a marking $f: T^{2} \lra M$.  Following Baues-Goldman, let $f: T^{2} \lra M$ be a marked non-Euclidean complete affine torus and $\omega_{M}$ the canonical $1$-form.  We let $e_{1},e_{2} \in \Z^{2} \simeq \pi_{1}(T^{2},x_{0})$ denote the standard basis for the fundamental group of $T^{2}$, and define the \emph{period pair} as below.
\begin{equation}\label{eq:per_pair}
p(f,M) := \left(\int_{e_{1}} f^{*}\omega_{M}, \int_{e_{2}} f^{*}\omega_{M} \right)
\end{equation}
For the Euclidean torus, we define $p(f,M) = 0$.  One of the main results from Baues-Goldman is that this map $p: \CDT \lra \R^{2}$ is a homeomorphism [Theorem 5.2] \cite{BauesGoldman2005}.  Moreover, this homeomorphism $p$ shows the mapping class group action of $\MCG(T^{2}) \simeq \SL(2,\Z)$ on $\CDT$ to be equivalent to the $\SL(2,\Z)$ action on $\R^{2}$ which is ergodic [Example 2.2.9] \cite{Zimmer1984}.  This is in stark contrast to the classical case where the mapping class group action on marked conformally flat structures on the torus is proper.  With the definition of the period parametrization, we provide some quick insights into their proof that it is a homeomorphism.  
\\
\\
Fix an identification of the universal cover of $T^{2}$ with $\R^{2}$ and identify $\pi_{1}(T^{2},0)$ with the integer lattice $\Z^{2}$ sitting inside $\R^{2}$ and denote elements of $\pi_{1}(T^{2},0) \simeq \Z^{2}$ by $\gamma = (n,m) \in \Z^{2}$.  For each $k = (k_{1},k_{2}) \in \R^{2}$, define $k^{\perp} = (-k_{2},k_{1})$ so that $(k,k^{\perp})$ is an oriented orthogonal basis of $\R^{2}$ relative to the standard orientation and denote by $x\cdot y$ the standard dot product.  To each $k \in \R^{2}$, we assign the representation $\rho_{k}: \Z^{2} \lra \Affp(2,\R)$, in terms of linear and translational parts as in Equation \ref{eq:lin_tran}, which acts properly discontinuously on $\R^{2}$ in the equation below.
\begin{equation}\label{eq:hol_rep}
L_{k}(\gamma)(v) := v - (k\cdot \gamma)(k \cdot v)k^{\perp} \text{ and } T_{k}(\gamma) := \gamma - \frac{1}{2}(k \cdot \gamma)^{2}k^{\perp}
\end{equation}
If we fix the standard oriented frame $(0,(e_{1},e_{2}))$ of $\R^{2}$, then the representation $\rho_{k} : \Z^{2} \lra \Affp(2,\R)$ is given below.
\begin{equation}\label{eq:coords_hol_rep}
\left (
  \begin {array} {ccc}
                    1+k_ {1} k_ {2} (k\cdot \gamma)  & k_ {2}^2 (k\cdot \gamma) & n+\frac {1} {2} k_ {2} (k\cdot \gamma)^2  \\
                  - k_ {1}^2 (k\cdot \gamma)& 1 -  k_ {1} k_ {2} (k\cdot \gamma) & m - \frac {1} {2} k_ {1}(k\cdot \gamma)^2  \\
                  0 & 0 & 1 
    \end {array}
   \right)
\end{equation}
We remark that this is the holonomy representation of the \emph{developing map} given by $\Psi_{k}: \R^{2} \lra \R^{2}$ via $\Psi_{k}(p) = p - \frac{1}{2}(k\cdot p)^{2}k^{\perp}$.  The developing map determines our marking by fixing a universal cover $\R^{2}$ of $T^{2}$ with the usual $\Z^{2}$-action by unit translations on it, and, the map satisfies the equivariance condition that $\Psi_{k}(p+\gamma) = \rho_{k}(\gamma)\Psi_{k}(p)$ for all $p \in \R^{2}$ and $\gamma \in \Z^{2}$.  This induces a marking with our fixed topological surface $T^{2} = \R^{2}/\Z^{2}$ with $T_{k}^{2} := \R^{2}/\rho_{k}(\Z^{2})$ via the developing map which takes $0 \in T^{2}$ to $0 \in T_{k}^{2}$.      
\\
\\ 
By Lemma \ref{lem:canform} for each $k \neq 0$, we obtain geometric invariants on $T_{k}^{2} := \R^{2}/\rho_{k}(\Z^{2})$.  One can readily check that the developing map $\Psi_{k}$ is both area and orientation preserving.  If we let $d\Omega_{k}$ be the unique invariant parallel area form on $T_{k}^{2}$, then the area of the affine torus is equal to $1$.  Because the developing map $\Psi_{k}$ is area preserving, we have the equalities below.
\begin{equation*}
1 = \int_{T^{2}} dx \wedge dy = \int_{T^{2}} \Psi_{k}^{*}d\Omega_{k} = \int_{T_{k}^{2}} d\Omega_{k} 
\end{equation*}
Thus, the parallel area form on $T_{k}^{2}$ is induced by $dx \wedge dy$.  The polynomial vector field is given below where we interpret $p \in \R^{2}$ as a vector.   
\begin{equation}\label{eq:polyvec}
Y_{k}|_{p} = \frac{1}{|k|^{2}}\left(k_{1}\del_{x} + k_{2}\del_{y}\right)|_{p} + (k\cdot p)\left(k_{2}\del_{x} - k_{1}\del_{y}\right)|_{p}
\end{equation}
With this definition, it follows that the parallel vector field is given by 
\begin{equation}\label{eq:parvec}
X_{k}|_{p} = \nabla_{Y_{k}} Y_{k}|_{p}= k_{2}\del_{x}|_{p} - k_{1}\del_{y}|_{p}
\end{equation}
Finally, the parallel one-form is given by 
\begin{equation}\label{eq:parone}
\omega_{k}|_{p} = \iota_{X_{k}}(d\Omega_{k})|_{p} = k_{1}dx|_{p} + k_{2}dy|_{p}
\end{equation}
Routine calculations verify these invariants satisfy the conditions of Lemma \ref{lem:canform} and thus induce the claimed differential structures on the affine tori $T_{k}^{2}$.  For $k \neq 0$, the vector field $X_{k}$ in Equation \ref{eq:parvec} integrates to a complete parallel flow $F_{t}: T_{k}^{2} \lra T_{k}^{2}$.  The flow is easily defined on the model space universal cover $\tilde{F}_{t}: \R^{2} \lra \R^{2}$ and is given by translation by the vector $tX_{k}$ for $t \in \R$.  Thus the flow lines yield a parallel oriented foliation $\fF_{k}$ of $\R^{2}$ invariant under the holonomy which descends to a parallel oriented foliation of $T_{k}^{2}$.  We choose leaves to be oriented positively along the direction of $X_{k}$ and negatively oriented along $-X_{k}$.  See Figure \ref{fig:foliate} below for an example.  
\\
\begin{figure}
\includegraphics[scale=0.5]{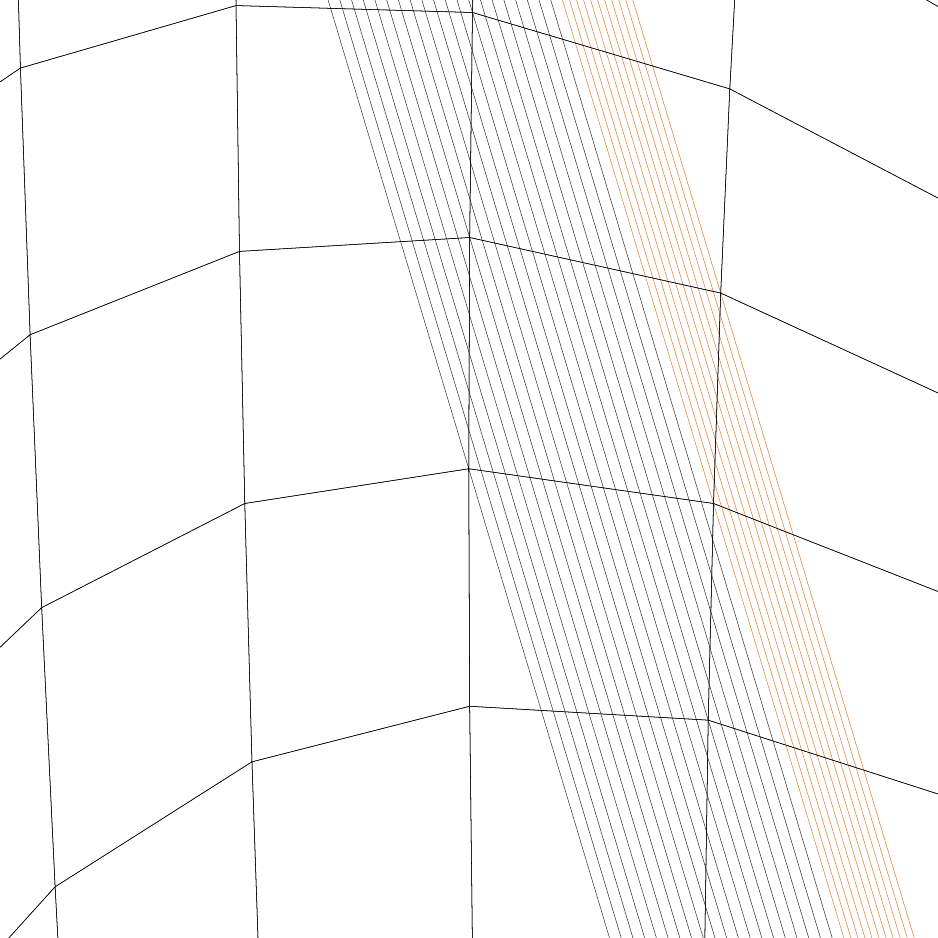}
\caption{A tiling of the marked torus whose period parameters are $k = (2/3,1/5)$.  The flow lines are directed in the direction $[(3,-10)] \in \fD$.  The parameter $k_{1}$ may be thought of as the weight of this oriented foliation across the edge through the black flow lines, and the parameter $k_{2}$ may be thought of analogously for the orange flow lines.  Recall the tiles are oriented counter-clockwise, and one can see that the vectors representing the edges of a tile pair positively with this foliation.
}\label{fig:foliate}
\end{figure}

We remark that for each $k \neq 0$, while $T_{k}^{2}$ admits a preferred parallel vector field $X_{k}$, it does not admit a natural complementary parallel vector field.  Its preferred complement $Y_{k}$ is non-parallel, and the absence of this canonical choice is a consequence of the linear holonomy of $\rho_{k}$, denoted $L_{k}$, taking its image inside a unipotent subgroup of $\SL(2,\R)$.  We now investigate the dynamics of this action on the circle of directions.  
\\
\\
To this end, choose a $k \neq 0$ and let us restrict our attention to the representation $L_{k} : \Z^{2} \lra \SL(2,\R)$ given by the linear part of the holonomy as in Equation \ref{eq:hol_rep}.  The vector fields $\pm X_{k}$ determine two invariant directions in $\fD$.  These fixed points $[\pm X_{k}]$ are unstable equilibria of the $\Z^{2}$-action on $\fD$, and the parallel one-form $\omega_{k}$ in Equation \ref{eq:parone} partitions $\fD$ into two types of directions which we define below.  Figure \ref{fig:posnegdirections} illustrates this definition.
\begin{definition}\label{def:posneg_dir}
Let $k \neq 0$, and let $\omega_{k} : \R^{2} \lra \R$ be the $\rho_{k}$-invariant one-form as defined in Equation \ref{eq:hol_rep}.  We say a direction $[v] \in \fD$ is $\omega_{k}$-non-negative if $\omega_{k}(v) > 0$, or, $[v] = [X_{k}]$.  Similarly, we define a $\omega_{k}$-non-positive direction $[v] \in \fD$ to be one where $\omega_{k}(v) < 0$, or, $[v] = [-X_{k}]$.  We denote this decomposition by $\fD = \fDnn \sqcup \fDnp$.  We say a direction $[v]$ is positive, negative, or neutral if $\omega_{k}(v)$ is positive, negative, or zero respectively.  We denote these sets as $\fD = \fDpos_{\omega_{k}} \sqcup \fDneg_{\omega_{k}} \sqcup \fDneu_{\omega_{k}}$ respectively.  
\\
\\
We define a loop $0 \neq \gamma \in \Z^{2}$ to be non-positive or non-negative if its corresponding direction $[T_{k}(\gamma)] \in \fD$ is non-positive or non-negative where $T_{k}(\gamma)$ is the translational part of $\rho_{k}(\gamma)$ as defined in Equation \ref{eq:hol_rep}.  Similarly, we say loops are positive, negative, or neutral if $[T_{k}(\gamma)] \in \fD$ is positive, negative, or neutral.  For convenience we define $0 \notin \fD$ to be a neutral loop.  
\end{definition}  

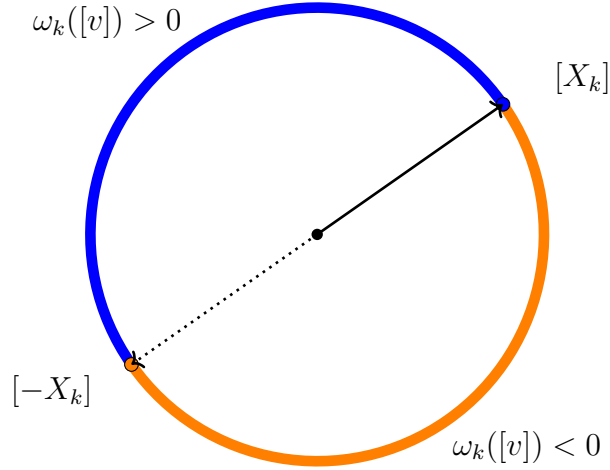
\begin{figure}
\begin{center}
\begin{tikzpicture}[scale=3]

	\def\ang{35}

  	\draw[thick] (0,0) circle (1);
	
  	\draw[line width=4pt, oskyblue]
    		({cos(\ang)},{sin(\ang)})
    		arc[
      			start angle=\ang,
      			end angle=180+\ang,
      			radius=1
    		];

  	\draw[line width=4pt, oorange]
    		({cos(180+\ang)},{sin(180+\ang)})
    		arc[
      			start angle=180+\ang,
      			end angle=360+\ang,
      			radius=1
    		];
	
	\node[xshift=-30pt, yshift=+10pt] at ({cos(90+\ang)},{sin(90+\ang)}) {$\omega_{k}([v]) > 0$};
	\node[xshift=30pt, yshift=-10pt] at ({cos(270+\ang)},{sin(270+\ang)}) {$\omega_{k}([v]) < 0$};
		
  	\fill (0,0) circle (0.025);
	
	\fill[oskyblue, draw=black] ({cos(\ang)},{sin(\ang)}) circle (0.03);
	\fill[oorange, draw=black] ({cos(180+\ang)},{sin(180+\ang)}) circle (0.03);
	
  	\draw[line width=1pt, ->]
    		(0,0) -- ({cos(\ang)},{sin(\ang)});
		\node[xshift=30pt, yshift=10pt] at ({cos(\ang)},{sin(\ang)}) {$[X_{k}]$};
		
  	\draw[dotted, line width=1pt, ->]
    		(0,0) -- ({cos(180+\ang)},{sin(180+\ang)});
    		\node[xshift=-30pt, yshift=-10pt] at ({cos(180+\ang)},{sin(180+\ang)}) {$[-X_{k}]$};

\end{tikzpicture}

\caption{A figure of the non-negative and non-positive directions of $\fD$ as defined by the one-form $\omega_{k}$.  Here the invariant directions of the linear holonomy $L_{k}$ are depicted by the thick black and dotted black rays $[X_{k}], [-X_{k}]$ respectively.  The blue half open arc with $[X_{k}]$ determines the non-negative directions and the orange half open arc with $[-X_{k}]$ determine the non-positive directions.  
}\label{fig:posnegdirections}
\end{center}
\end{figure}


Note that by definition a direction $[v] \in \fD$ is positive if and only if $\omega_{k}(v) := d\Omega_{k}(X_{k},v) > 0$.  That is to say, $(X_{k},v)$ forms an \emph{oriented basis} for $\R^{2}$.  Because the canonical one-form is holonomy invariant, the $\Z^{2}$-action on the circle of directions $\fD$ preserves the decomposition $\fD = \fDnn\sqcup \fDnp$ and in particular preserves both directions $[\pm X_{k}]$.  An inspection of the linear holonomy as in Equation \ref{eq:hol_rep} shows that positive loops move non-neutral directions clockwise, with the positive loops tending towards $X_{k}$ and the negative loops tending towards $-X_{k}$ whereas negative loops move non-neutral directions counter-clockwise.  Neutral loops act trivially on $\fD$, and all three types of loops preserve the two invariant directions $[\pm X_{k}]$.  Topologically, this action is the double cover of a unipotent subgroup of $\SL(2,\R)$ acting on the boundary of the hyperbolic plane.  \\
\\
Later in Section \ref{sec:divorbs}, we will relate the dynamics of the symplectic tiling billiards to measuring $\omega_{k}$ along billiard trajectories.  In fact, we will use $\omega_{k}$ to measure divergence as in Definition \ref{def:divpath}.  Because $\R^{2}$ is simply-connected, the pairing of $\omega_{k}$ against a path $\alpha: [0,1] \lra \R^{2}$ depends only on the end points of $\alpha$.  In fact, one can readily see that $\int_{\alpha}\omega_{k} = d\Omega_{k}(X_{k},\alpha(1)-\alpha(0))$ and is thus measuring how many (oriented) leaves of the invariant foliation of $\R^{2}$ determined by $X_{k}$ we pass through.  Having $\int_{\alpha} \omega_{k} = 0$ means up to end point fixed homotopy, $\alpha$ is parallel to $\pm X_{k}$.  Thus the parallel one-form $\omega_{k}$ in Equation \ref{eq:parone} is complementary to the flow $F_{t}: T_{k}^{2} \lra T_{k}^{2}$ as the flow lines integrate to $0$ on $\omega_{k}$.   \\
\\
There is one particularly nice case where the flow lines of $F_{t}$ close up in uniform time, and in this instance we have an induced $S^{1}$-action on $T_{k}^{2}$.  If we consider the linear holonomy $L_{k}(\gamma)$ as defined in Equation \ref{eq:hol_rep}, we see if $k\cdot \gamma = 0$ for some $0 \neq \gamma \in \Z^{2}$, then $L_{k}(\gamma) = \id$, so $\rho_{k}(\gamma)$ acts by the translation $T_{k}(\gamma) = \gamma$ and thus is parallel to the vector $X_{k}$.  Consequently, the flow lines all close up in uniform time.  Algebraically, that is to say the linear holonomy map $L_{k}: \Z^{2} \lra \SL(2,\R)$ admits a non-trivial kernel.  This kernel is generated by the smallest relatively prime pair of integers $(n,m) = \gamma$ satisfying $\omega_{k}(\gamma) = d\Omega_{k}(X_{k},T_{k}(n,m)) = 0$.  This occurs precisely when the period parameters $0 \neq k = (k_{1},k_{2})$ as defined in Equation \ref{eq:per_pair} are \emph{rationally related}.  That is to say, there exists non-trivial integer solutions to the equation $nk_{1} + mk_{2} = 0$ where $n,m \in \Z$.  This motivates the concluding definition of this section below.
\begin{definition}\label{def:rattor}
We say a complete non-Euclidean affine torus $M$ is \emph{rational} if it admits a marking $(f,M)$ for which the period parameters $k = (k_{1},k_{2})$ as defined in Equation \ref{eq:per_pair} admit a non-trivial integer solution to the equation $nk_{1} + mk_{2} = 0$ for $n,m \in \Z$ not both equal to zero.  We take the Euclidean torus to be rational and say a non-Euclidean complete affine torus is \emph{irrational} if no such solution exists.  
\end{definition}

\subsection{Marked Complete Affine Tilings}\label{ssec:mcat}
We dedicate this section to analyzing the corresponding tilings that arise from the markings of complete affine tori as in Section \ref{ssec:afgeo}.  We wish to have tilings of the plane by convex quadrilaterals whose tiling symmetries are the holonomy representations of markings of complete affine tori.  We insist that these tilings be done by \emph{convex quadrilaterals} with the intention of playing symplectic tiling billiards on them.  To this end, we make precise what we mean by a \emph{marked complete affine tiling}.   
\begin{definition}\label{def:mark_tile}
Let $T$ be a tiling of the plane by convex quadrilaterals.  We say that $T$ is a \emph{complete affine tiling} if one of its tiles is the fundamental domain of a properly discontinuous $\Z^{2}$-action on the plane by symmetries in $\Affp(2,\R)$ where each such symmetry preserves the tiling $T$.  We call a choice of an oriented edge of $\del T$ a \emph{marking} of the tiling and the pair a \emph{marked complete affine tiling}.  We represent the edge by the vector $v_{1} \in \R^{2}$ determined by the difference of the oriented edge's terminal and initial point.  
\end{definition}
We provide the following observations about our definition above.  By the hypothesis that the $\Z^{2}$-action preserves the tiling and the marked tile is a fundamental domain for the action, the $\Z^{2}$-orbit of the tile must see every other tile, and because the action is properly discontinuous, the group must glue opposite pairs of edges of $Q$ to each other.  Consequently, the quotient is a complete affine torus.  A marking of a complete affine tiling determines a representation $\rho_{T}: \Z^{2} \lra \Affp(2,\R)$ in the following manner.  The choice of an oriented edge $v_{1}$ of $\del T$ naturally determines a base-point by picking the initial point of the oriented edge.  We call this base-point the \emph{anchor} of the tile.  Moreover, the marking determines a unique base-tile.  Of the two tiles through $v_{1}$, choose the one whose interior is to the left of $v_{1}$, so that the tile $Q$ is labeled by $4$ oriented edges $v_{1},\hdots, v_{4}$ traversing the tile in a counter-clockwise manner.  See Figure \ref{fig:marked_tile} below.
\\
\begin{figure}
\includegraphics[scale=0.5]{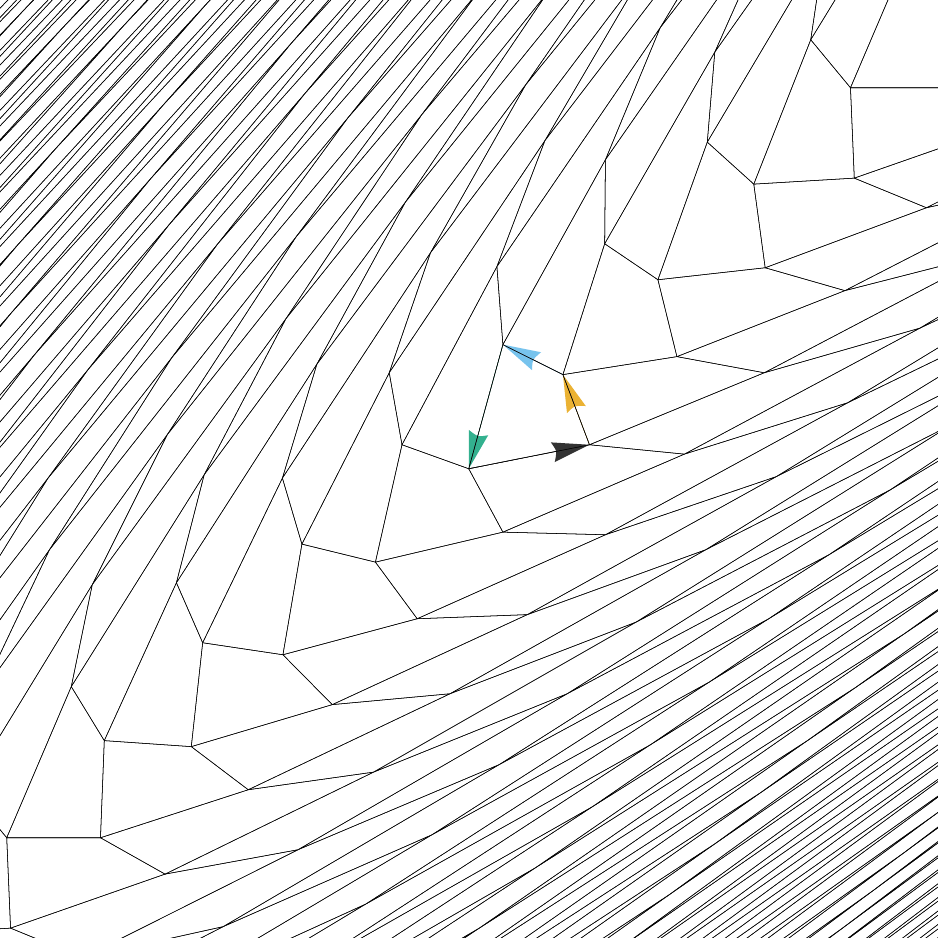}
\caption{The marked edge of this tiling is given by the black arrow, and the marked tile $Q$ is the one decorated with arrows.  The edges $v_{1},v_{2},v_{3},v_{4}$ are labeled by the black, orange, blue, and green vectors.  The anchor is the initial point of $v_{1}$.  This tiling corresponds to the torus with period parameters $k = (-8/10,9/10)$.  
}\label{fig:marked_tile}
\end{figure}

We say that $Q$ is the \emph{marked} tile of the tiling $T$ and call the labeling $v_{1}, v_{2}, v_{3}, v_{4}$ the \emph{standard labeling} of the marked tile $Q$.  The standard labeling unambiguously determines the bottom, right, top, and left oriented edges of $Q$ respectively by $v_{1},v_{2},v_{3},v_{4}$.  Such a choice of a marked edge of a single tile then induces a standard labeling on every tile of $T$ which we call the \emph{induced labeling}.  Figure \ref{fig:induced_orient} illustrates the induced labeling and we will always assume for a marked complete affine tiling, all other tiles are labeled according to the induced labeling.  With these conventions established we can provide an associated representation from the tiling. 
\\
\begin{figure}
\includegraphics[scale=0.5]{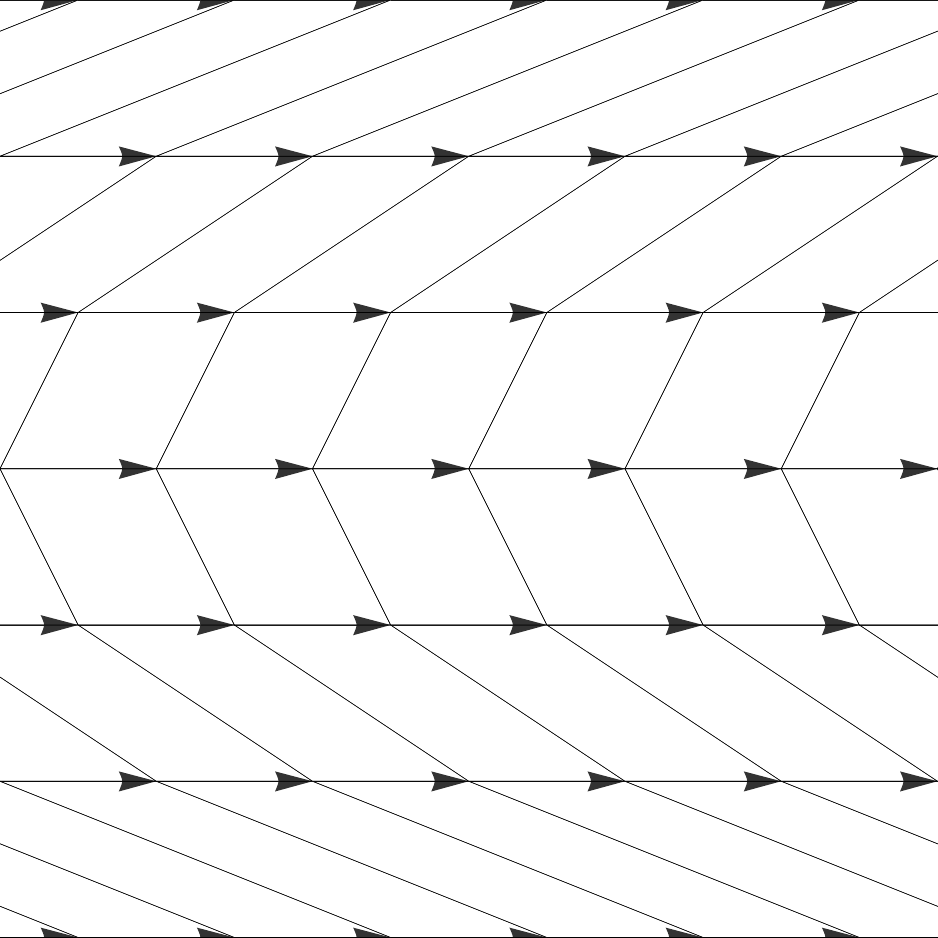}
\caption{A single marked edge of $T$ determines markings of all other edges and every tile inherits a labeling of its boundary by the standard labeling $v_{1},\hdots, v_{4}$.  Here we illustrate all possible $v_{1}$'s for each tile.  
}\label{fig:induced_orient}
\end{figure}
%
\begin{definition}\label{def:mtr}
Let $T$ be a marked complete affine tiling marked by the oriented edge $v_{1}$ of $\del T$ and whose marked tile $Q$ is labeled by the edges $v_1,v_2,v_3,v_4$ under the standard labeling.  We define the \emph{marked tiling representation} $\rho_{T}: \Z^{2} \lra \Affp(2,\R)$ to be the representation taking $e_{1} \in \Z^{2}$ to be the symmetry of the tiling taking edge $v_4$ of $Q$ to edge $-v_2$ of $Q$ and $e_{2} \in \Z^{2}$ to be the symmetry of the tiling taking edge $v_1$ of $Q$ to edge $-v_3$ of $Q$.
\end{definition}
We remark the negative signs included in $v_{2}, v_{3}$ above are necessary because of orientation considerations.  Figure \ref{fig:mtrep} illustrates the definition.  Note had we picked another oriented edge labeled $v_{1}$ from the induced labeling of $Q$, we obtain the same marked tiling representation.  By hypothesis the image of $\rho_{T}$ acts properly discontinuously on $\R^{2}$ and thus determines a point in deformation space, $\CDT$.  Observe that if instead we chose the other orientation of our edge $v_{1}$, we get another representation $\rho_{T'}$.  This representation differs from $\rho_{T}$ by the element of the mapping class group $\MCG(T^{2})$ taking $(n,m)$ to $(-n,-m)$, and generically determines a different point in deformation space, thus we take care to remember the marking of the tiling.  
\\
\begin{figure}
\includegraphics[scale=0.5]{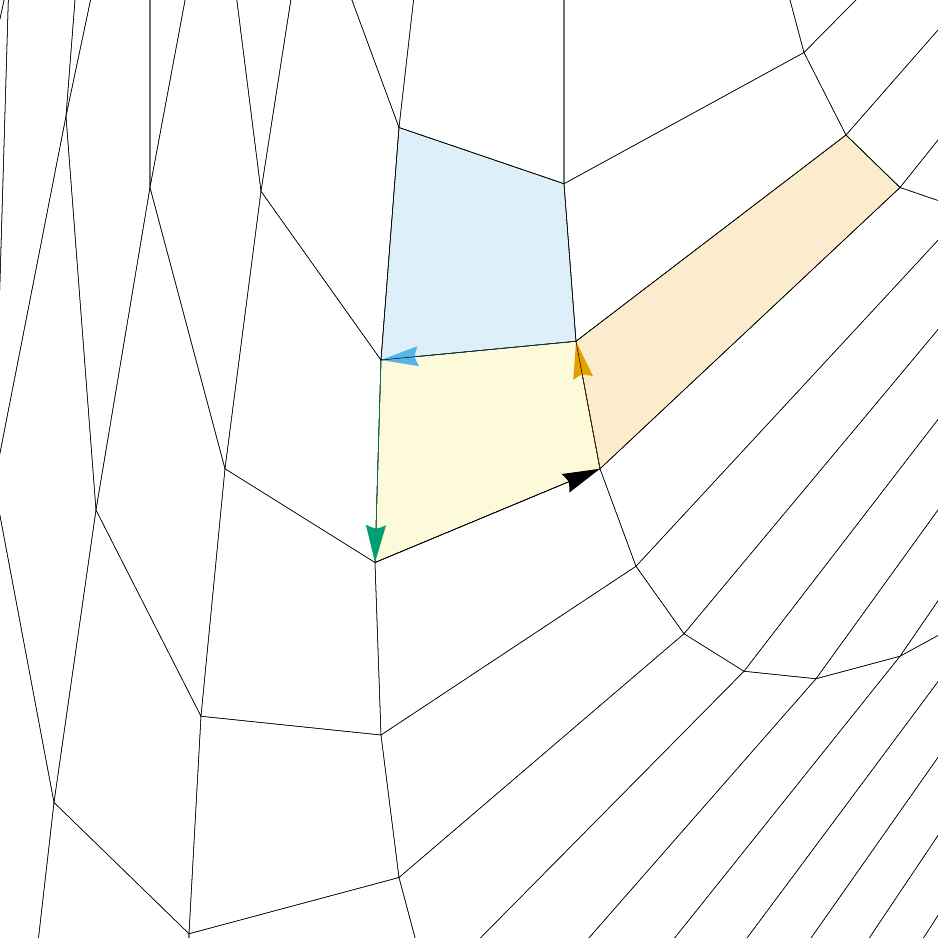}
\caption{The marked tiling representation of the marked complete affine tiling.  The period parameters are given by $k = (-1,4/10)$.  The marked tile is illustrated in $Q$ with $v_{1},\hdots, v_{4}$ drawn in black, orange, blue, and green in the standard labeling.  The action $\rho_{T}(e_{1})$ is the tiling symmetry taking the yellow tile to the orange tile, and the action $\rho_{T}(e_{2})$ is the tiling symmetry taking the yellow tile to the blue tile.  Equivalently $\rho_{T}(e_{1})$ and $\rho_{T}(e_{2})$ specify the affine symmetries of the tiling moving right and up through the tiling respectively.  
}\label{fig:mtrep}
\end{figure}

Because $\rho_{T}(e_{1})$ and $\rho_{T}(e_{2})$ take the marked tile $Q$ one tile to the right and up respectively, we have a $\Z^{2}$-action on the tiles of $T$.  We will denote by $\gamma Q$ the tile obtained from $\rho_{T}(\gamma)$ applied to $Q$ where $\gamma = (n,m) \in \Z^{2}$.  Figure \ref{fig:moving} illustrates this for the tiling in Figure \ref{fig:mtrep} with the loop $\gamma = (2,3)$.  We will make use of this action later in Section \ref{sec:divorbs} when we analyze the limiting shapes of the quadrilaterals under a sequence of divergent positive loops.  
\\
\begin{figure}
\includegraphics[scale=0.5]{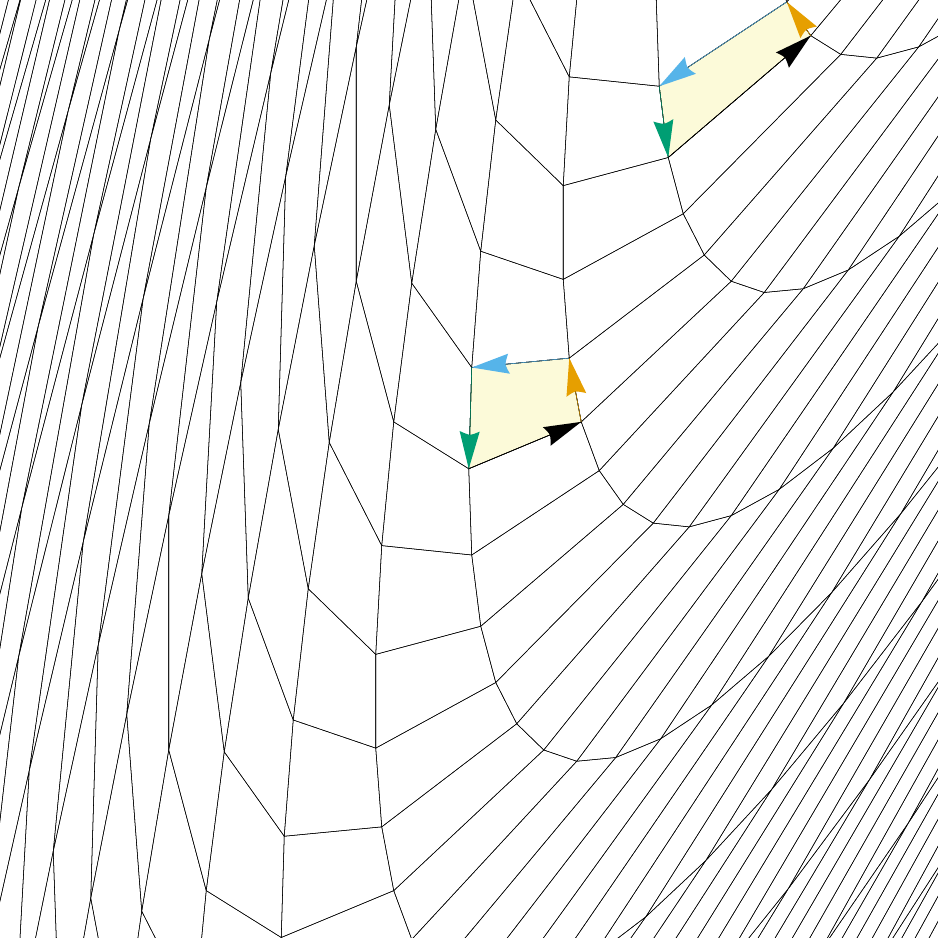}
\caption{The bottom left marked tile $Q$ is taken to the upper right marked tile $\gamma Q$ where $\gamma = (2,3) \in \Z^{2}$.}\label{fig:moving}
\end{figure}

Optimistically, one could hope that for any choice of $k \in \R^{2}$ there is a marked complete affine tiling $T$ so that the marked tiling representation of $T$ and $k$ determine the same point in deformation space $\CDT$.  Interestingly this is not the case and illustrates a subtlety that does not occur in the classical case which is the content of our first theorem below.
\begin{theorem}\label{thm:marked_tiles}
The subspace of $\CDT$ corresponding to the holonomies of marked complete affine tilings is an unbounded open subset cut out by the algebraic inequalities in the period parameters $k = (k_{1},k_{2})$ as in Equation \ref{eq:conv_cond}.  That is to say, the space of equivalence classes of markings of complete affine tori which admit a marked convex quadrilateral fundamental domain $Q$ in $\R^{2}$ where the gluings of pairs of opposite edges, $v_{4}$ to $-v_{2}$ and $v_{1}$ to $-v_{3}$, correspond to $\rho_{k}(e_{1}), \rho_{k}(e_{2}) \in \Affp(2,\R)$ respectively, is an unbounded open algebraic subset of $\CDT$.  
\end{theorem}
\begin{proof}
Let $T_{k}^{2}$ be a marked complete affine torus whose holonomy $\rho_{k} : \Z^{2} \lra \Affp(2,\R)$ admits a convex quadrilateral fundamental domain $Q$ for the action of $\rho_{k}(\Z^{2})$ on $\R^{2}$.  Because $T_{k}^{2}$ is compatibly oriented with $\R^{2}$, and $\rho_{k}(e_{1})$ and $\rho_{k}(e_{2})$ glue opposite pairs of edges, there is a unique vertex $p \in \del Q$ so that $v_{1} := \rho_{k}(e_{1})p - p$ is an edge of $Q$ whose interior lies to the left of this vector.  The image of $Q$ under $\rho_{k}(\Z^{2})$ tiles $\R^{2}$ and thus produces a marked complete affine tiling of $\R^{2}$ marked by edge $v_{1}$ of $Q$.  \\
\\
By the classification of marked complete affine tori as in Section \ref{ssec:afgeo}, there exists an affine transformation $g \in \Affp(2,\R)$ so that $\rho_{k}$ above is conjugate to the representation determined by Equation \ref{eq:hol_rep}.  In particular, this means $gQ$ is also a marked convex quadrilateral fundamental domain marked by $L(g)v_{1}$ where $L(g)$ is the linear part of $g$ and where $e_{1}$ and $e_{2}$ glue opposite pairs of edges of $gQ$.  We may without loss of generality choose $g$ so that $gQ$ has its anchor at the origin.  The standard labeling of $gQ$ is thus the images of $v_{1},\hdots, v_{4}$ under $L(g)$.  As $e_{1}$ and $e_{2}$ glue opposite pairs of edges of $gQ$, and $gQ$ has its anchor at the origin, this means the vectors $L(g)v_{1},\hdots, L(g)v_{4}$ take the following form below where $T_{k}(\gamma)$ is the translational part of the element $\gamma$ as in Equation \ref{eq:hol_rep}.
\begin{equation*}
\{L(g)v_{1},\hdots, L(g)v_{4}\} := \{T_{k}(e_{1}), T_{k}(e_{1}+e_{2}) - T_{k}(e_{1}), T_{k}(e_{2}) - T_{k}(e_{1}+e_{2}), -T_{k}(e_{2})\}
\end{equation*}
By convexity, we know the $\det(L(g)v_{i},L(g)v_{i+1}) > 0$ for all $i = 1,\hdots 4$ where the indices are interpreted cyclically.  Routine calculation shows this necessitates the above algebraic expressions in $k = (k_{1},k_{2})$ are all strictly positive.  
\begin{equation}\label{eq:conv_cond}
2-k_{1}k_{2}(k_{1}+k_{2}) \phantom{=} 2+k_{1}k_{2}(-k_{1}+k_{2}) \phantom{=} 2+k_{1}k_{2}(k_{1}+k_{2}) \phantom{=} 2+k_{1}k_{2}(k_{1}-k_{2})
\end{equation}
Thus if we start with a marked complete affine torus $T_{k}^{2}$ whose holonomy representation admits a convex quadrilateral fundamental domain $Q$ gluing opposite pairs of edges, the period parameters necessarily satisfy the algebraic conditions as in Equation \ref{eq:conv_cond}.
\\
\\
Conversely, if we start with a period parameter $k$ that satisfies the inequalities in Equation \ref{eq:conv_cond}, we can form the convex quadrilateral fundamental domain given by $Q = \{0, T_{k}(e_{1}), T_{k}(e_{1}+e_{2}),T_{k}(e_{2})\}$.  We mark this tile $Q$ by the oriented edge connecting $0$ and $T_{k}(e_{1})$, and this determines a complete marked affine torus whose holonomy representation is given by $\rho_{k} : \Z^{2} \lra \Affp(2,\R)$.  It is clear the conditions of Equation \ref{eq:conv_cond} are open on the period parameters, and, unbounded as all $k$ for which $k_{1}k_{2} = 0$ also satisfy the conditions of Equation \ref{eq:conv_cond}.  
\end{proof}
The set of such $k$ satisfying the conditions of Theorem \ref{thm:marked_tiles} is illustrated in Figure \ref{fig:convexmarked}.  As remarked before Theorem \ref{thm:marked_tiles}, this difficulty does not present itself in the classical case.  There, the deformation space of marked flat conformal structures on the torus identifies with the upper half plane.  A single complex number $z$ with positive imaginary part determines the marking $\rho_{z}: \Z^{2} \lra \text{Conf}(\C)$, and one can always construct a fundamental domain $Q$ so that $\rho_{z}(e_{1})$ and $\rho_{z}(e_{2})$ glue opposite pairs of edges of $Q$ together which is compatible with the marking.  If however, we are willing to forget the marking, then any complete affine torus $M$ always admits a convex quadrilateral fundamental domain.  That is to say, one can always find an oriented pair of generators of $\alpha, \beta \in \Z^{2}$ and a convex quadrilateral fundamental domain $Q$ so that $\rho_{k}(\alpha)$ and $\rho_{k}(\beta)$ glue opposite edges together.  This original proof is due to Baues, however we include it for the sake of self-containment \cite{Baues1999Gluing}.   

\begin{figure}
\includegraphics[scale=0.5]{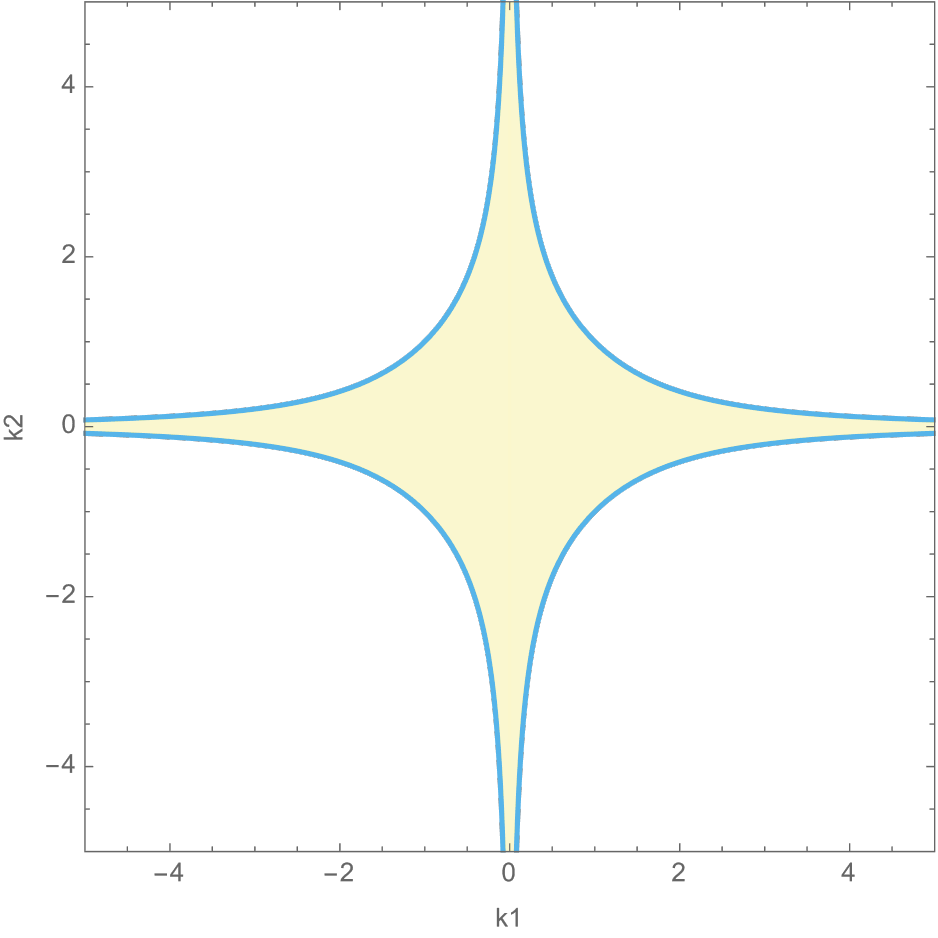}
\caption{The $x,y$ axes are $k_{1},k_{2}$ respectively.  This figure is the collection of all equivalence classes of marked complete affine tori $k \in \R^{2} \simeq \CDT$ whose holonomy representation $\rho_{k}: \Z^{2} \lra \Affp(2,\R)$ are, up to conjugation, the marked tiling representation of some marked complete affine tiling $T$ of $\R^{2}$.  We've included the boundary in thick blue line for the sake of illustration.  Points on this boundary correspond to when two edges of the marked convex quadrilateral become parallel.  
}\label{fig:convexmarked}
\end{figure}

\begin{lemma}\label{lem:convex_tilings}
Let $\Gamma_{k} \leq \Affp(2,\R)$ be the image of the holonomy map $\rho_{k}: \Z^{2} \lra \Affp(2,\R)$ of a complete affine torus $T_{k}^{2}$.  Then there exists a convex quadrilateral fundamental domain for the action of $\Gamma_{k}$ on $\R^{2}$.  
\end{lemma}
\begin{proof}
Let $\rho_{k}: \Z^{2} \lra \Affp(2,\R)$ be the holonomy map of a complete affine torus, and $\Gamma_{k}$ denotes its image in $\Affp(2,\R)$.  We begin first by observing that in the deformation space $\CDT$, the property of $\Gamma_{k}$ admitting a convex quadrilateral fundamental domain is invariant under the mapping class group, as the image of the representation stays the same.  Thus it suffices to find a mapping class group element $A \in \MCG(T^{2})$ so that $(A\rho_{k}): \Z^{2} \lra \Affp(2,\R)$ defined by $\rho_{k}\circ A^{-1}$ is an element of the open subset of $\CDT$ which admits convex quadrilateral fundamental domains as in Theorem \ref{thm:marked_tiles}.\\
\\
To this end, if $k \in \R^{2} \simeq \CDT$ it has coefficients $(k_{1},k_{2})$ with no non-trivial rational solution to $ak_{1} + bk_{2} = 0$ for $a,b \in \Q$, then the orbit of $\rho_{k}$ under $\MCG(T^{2})$ is dense and the result follows.  If however $k \in \R^{2}$ is rational, we construct a fundamental domain for the $\Gamma_{k}$-action on $\R^{2}$.  By rationality, the linear part of the map $L_{k}: \Z^{2} \lra \Gamma_{k} \lra \SL(2,\R)$ has non-trivial kernel.  Let $\alpha$ be a generator of this kernel and $\beta$ be an element of $\Z^{2}$ so that $(\alpha,\beta)$ form an oriented basis of $\Z^{2}$.  It follows then that $\rho_{k}(\alpha)$ is a translation parallel to the invariant vector field, and thus $k_{1} = 0$.  If we let $A$ be the mapping class group element satisfying $A(\alpha) = e_{1}$ and $A(\beta) = e_{2}$, which exists because $(\alpha,\beta)$ is an oriented basis of $\Z^{2}$, then the representation $A\rho_{k}$ has period parameters which satisfy $k_{1}k_{2} = 0$, as $(A\rho_{k})(e_{1})$ is a translation by construction, thus $A\rho_{k} \in \MCQ$.  By the same observation applied to the irrational case, this means $A\rho_{k}$ admits a convex quadrilateral fundamental domain.  
\end{proof}
Another way of interpreting Lemma \ref{lem:convex_tilings} is that if we have the holonomy representation of a marked complete affine torus with period parameter $k \in \CDT$, then there is an element of the mapping class group $A \in \MCG(T^{2})$ so that $A\rho_{k} := \rho_{k}\circ A^{-1}$ is a new marking of the torus, whose period parameter lands in the open unbounded subset of $\CDT$ defined in Theorem \ref{thm:marked_tiles}.  For the remainder of this work, we will work exclusively with this set, thus we dedicate a definition to it.
\begin{definition}\label{def:mcqs}
The unbounded open subset of $\CDT$ of marked representations $\rho_{k}: \Z^{2} \lra \Affp(2,\R)$, up to conjugation, which admit marked complete affine tilings with marked tile $Q$ so that $\rho_{k}(e_{1})$ and $\rho_{k}(e_{2})$ glue opposite pairs of edges, is the \emph{marked convex complete tilings subspace} of $\CDT$ and denoted by $\MCQ$.  
\end{definition}
We remark Baues was the original author to address this space, see [Section 4.6, Corollary 4.11], however the precise algebraic inequalities in the period parameters defining $\MCQ$ as in Equation \ref{eq:conv_cond} are new to the authors' knowledge \cite{Baues1999Gluing}.  The motivation for our consideration of $\MCQ$ derives from the desire to have a large parameter space of tilings to play symplectic tiling billiards on.  Ideally, we wish to be able to choose pairs of parameters in $\MCQ$ freely.  This is the content of Theorem \ref{thm:goodtiles}, but before proving this theorem we provide two technical lemmas and a definition that will aid us in our analysis.  
\begin{lemma}\label{lem:caninvs}
Let $T$ be a marked complete non-Euclidean affine tiling of $\R^{2}$ and $\Gamma_{T}$ the image of its marked tiling representation as defined in Definition \ref{def:mtr}.  There exists a unique parallel area form $d \Omega_{T}$, parallel vector field $X_{T}$, parallel one-form $\omega_{T}$, and polynomial vector field $Y_{T}$ all of which are $\Gamma_{T}$-invariant, and satisfy the conditions below.
\begin{equation*}
\int_{Q} d\Omega_{T} = 1, \phantom{=} X_{T} = \nabla_{Y_{T}} Y_{T}, \phantom{=} d\Omega_{T}(X_{T},Y_{T}) = 1, \phantom{=} \omega_{T} = \iota_{X_{T}}(d\Omega_{T})
\end{equation*}
In the above equalities, $Q$ is any tile of $T$ and the uniqueness of $Y_{T}$ is up to constant multiples of $X_{T}$.  
\end{lemma}
\begin{proof}
Lift the invariants as provided in Lemma \ref{lem:canform} from the complete affine torus $\R^{2}/\rho_{T}(\Z^{2})$ to $\R^{2}$.  The fact that $Q$ has area $1$ follows because $Q$ is a fundamental domain for the $\rho_{T}(\Z^{2})$-action on $\R^{2}$.  
\end{proof}
We frequently refer to these invariants as the canonical invariants of the marked complete non-Euclidean affine tiling $T$.  Definitions of non-negative and non-positive directions relative to the one-form $\omega_{T}$ as in Definition \ref{def:posneg_dir} follow readily.  Similarly, we say a loop $\gamma \in \Z^{2}$ is non-negative or non-positive if $[T_{T}(\gamma)] \in \fD$ is $\omega_{T}$-non-negative or non-positive where $T_{T}(\gamma)$ is the translational part of the affine transformation $\rho_{T}(\gamma)$.  We make similar definitions for $\omega_{T}$-positive, negative, and neutral loops.   
\\
\\
As mentioned before, the motivation to consider these tilings is to play symplectic tiling billiards.  The edge directions of the marked complete affine tiling $\fD_{T} \subset \fD$ as defined in Definition \ref{def:cded} thus naturally come into consideration.  Because our tilings exhibit affine symmetry, this symmetry is inherited by $\fD_{T}$.  In particular note that the edges $\{v_{1},\hdots, v_{4}\}$ of the marked tile $Q$ are sent to the edges $\{L_{T}(\gamma)(v_{1}),\hdots L_{T}(\gamma)(v_{4})\}$ of the marked tile $\gamma Q$ where $L_{T}$ denotes the linear holonomy of the marked tiling representation.  Thus the edge directions determined by $[\pm v_{i}] \in \fD_{T}$ for every marked tile are invariant under the $\Z^{2}$-action on the circle of directions.  The positive loops of $\Z^{2}$ move the positive and negative directions in a clockwise manner towards $[X_{T}]$ and $[-X_{T}]$ respectively, whereas the negative loops move them counter-clockwise as noted in the paragraph after Definition \ref{def:posneg_dir}.  The condition in which the invariant directions $[\pm X_{T}]$ are in $\fD_{T}$ can be detected by the following equivalences. 
\begin{lemma}\label{lem:k1k2zero}
Let $T$ be a marked complete affine tiling of $\R^{2}$ with marked tiling representation $\rho_{T}$.  The following conditions are equivalent.
\begin{itemize}
	\item The direction of a linear holonomy invariant vector is contained in $\fD_{T}$.
	\item The boundary of the tiling $\del T$ contains an infinite line.
	\item The period parameters of the tiling satisfy $k_{1}k_{2} = 0$.
\end{itemize}
\end{lemma}
\begin{proof}
Let $[X_{T}]$ be the direction of a linear holonomy invariant vector of the marked tiling representation contained in $\fD_{T}$.  This means there is a marked tile $Q$ in $T$ that has an oriented edge $v_{i}$ parallel to $X_{T}$.  Without loss of generality assume it's $v_{1}$.  By holonomy invariance this means for every $n \in \Z$, $L_{T}^{n}(e_{1})v_{1}$, which is the $v_{1}$-edge of the tile $(n,0)Q$, is also parallel to $X_{T}$.  Thus there is an infinite line in $\del T$.  \\
\\
If we assume there is an infinite line in $\del T$, we claim the period parameters satisfy $k_{1}k_{2} = 0$.  Let $Q$ be a marked tile of $T$.  $Q$ necessarily contains part of an infinite line segment.  This follows because by hypothesis there is another tile $Q'$ of $T$ which contains a line segment of the infinite line in $\del T$.  Because $T$ is a complete affine tiling, there is some element $\gamma \in \Z^{2}$ so that $\gamma Q' = Q$.  Because $\rho_{T}(\gamma)$ takes infinite lines to infinite lines, this means $Q$ must also have one of its boundary edges contained in an infinite line of $\del T$ because $Q'$ did.  \\
\\
Without loss of generality let us assume the marked edge $v_{1}$ of $Q$ is contained in this infinite line of $\del T$.  Then $\rho_{T}(e_{1})(v_{1})$ is parallel to $v_{1}$.  Because the linear holonomy of $L_{T}$ is parabolic, this means $v_{1}$ is parallel to the invariant vector field $X_{T}$.  By definition of the period parameter, this means $k_{1} = 0$ as claimed.  If instead the infinite line of $\del T$ contains another oriented edge of $Q$, the same argument shows either $k_{1} = 0$ or $k_{2} = 0$.  \\
\\
Finally, let us assume the period parameters satisfy $k_{1}k_{2} = 0$.  We claim there is a linear holonomy invariant direction contained in $\del T$.  As this is vacuously true for Euclidean tori, because every direction is preserved under the linear holonomy, assume our torus is non-Euclidean.  Then either $\int_{v_{1}} \omega_{T} = 0$ or $\int_{v_{2}} \omega_{T} = 0$ where $\omega_{T}$ is the canonical invariant one-form.  This means either the edge $v_{1}$ or $v_{2}$ is parallel to the flow generated by $X_{T}$, and thus, $[\pm X_{T}] \in \fD_{T}$.   
\end{proof}

Before proceeding to Theorem \ref{thm:goodtiles}, we introduce one more labeling on marked tiles that will be helpful later.  Recall as in Definition \ref{def:posneg_dir}, the canonical invariant one-form $\omega_{T}$ determines a decomposition $\fD = \fDnno \sqcup \fDnpo$ which the $\Z^{2}$-action on $\fD$ preserves.  The marking of the tile $Q$ determines the standard labeling $\{v_{1},v_{2},v_{3},v_{4}\}$ which determine the collection of points $\{[v_{1}],[v_{2}],[v_{3}],[v_{4}]\}$ in $\fD$.  Due to the invariance of the one-form, we have the following equalities below.
\begin{equation}\label{eq:edge_par}
k_{1} = \int_{v_{1}} \omega_{T} \phantom{=} k_{2} = \int_{v_{2}} \omega_{T}  \phantom{=} -k_{1} = \int_{v_{3}} \omega_{T} \phantom{=} -k_{2} = \int_{v_{4}} \omega_{T} 
\end{equation}
It therefore follows that two directions lie in $\fDnno$ and the other two lie in $\fDnpo$.  This induces an ordering on the edges $\{[v_{1}],[v_{2}],[v_{3}],[v_{4}]\}$ which is equivalent to starting from $[X_{T}]$ and traversing the circle counter-clockwise recording, in order, which directions we intersect.  It is possible that $[\pm X_{T}] = [v_{i}]$ which is determined by the condition that $k_{1}k_{2} = 0$ as in Lemma \ref{lem:k1k2zero}.  \\
\\
We label this ordering of directions by $\{[v_{1}^{+}], [v_{2}^{+}] ,[v_{1}^{-}],[v_{2}^{-}]\}$.  By definition $[v_{i}^{+}] \in \fDnno$ and $[v_{i}^{-}] \in \fDnpo$ and thus the ordering is invariant under the $\Z^{2}$-action on $\fD$.  This follows because $[v_{1}^{+}]$ comes before $[v_{2}^{+}]$ reading counter-clockwise from $[X_{T}]$, so for any $\gamma \in \Z^{2}$, we have $L_{k}(\gamma)[v_{1}^{+}]$ comes before $L_{k}(\gamma)[v_{2}^{+}]$ because the linear holonomy preserves orientation.  The same argument holds for the non-positive directions.  We call this ordering the \emph{$\omega_{T}$-ordering} of the edges and summarize it in the definition below.  Figure \ref{fig:omegatordering} illustrates the ordering.  
\begin{definition}\label{def:omegat}
Let $T$ be a non-Euclidean complete affine tiling with canonical vector field $X_{T}$ and one-form $\omega_{T}$.  Let $\{[v_{1}],[v_{2}],[v_{3}],[v_{4}]\}$ be the directions of the edges of the marked tile $Q$ under the standard labeling.  Starting from the direction $[X_{T}] \in \fD$, and traversing counter-clockwise, label these directions by $\{[v_{1}^{+}], [v_{2}^{+}] ,[v_{1}^{-}],[v_{2}^{-}]\}$.  This is the $\omega_{T}$-ordering of the edges.  We call $v_{i}^{+}$ and $v_{i}^{-}$ the \emph{non-negative} and \emph{non-positive} edges of $Q$ respectively.  
\end{definition}

\begin{figure}
\begin{center}
\begin{tikzpicture}[scale=3]

	\def\ang{35}
	\def\eps{5}

  	\draw[thick] (0,0) circle (1);
	
  	\draw[line width=4pt, oskyblue]
    		({cos(\ang)},{sin(\ang)})
    		arc[
      			start angle=\ang,
      			end angle=180+\ang,
      			radius=1
    		];

  	\draw[line width=4pt, oorange]
    		({cos(180+\ang)},{sin(180+\ang)})
    		arc[
      			start angle=180+\ang,
      			end angle=360+\ang,
      			radius=1
    		];
	
	\node[xshift=-30pt, yshift=+10pt] at ({cos(90+\ang)},{sin(90+\ang)}) {$\omega_{k}([v]) > 0$};
	\node[xshift=30pt, yshift=-10pt] at ({cos(270+\ang)},{sin(270+\ang)}) {$\omega_{k}([v]) < 0$};
		
  	\fill (0,0) circle (0.025);
	
  	\draw[dotted, line width=1pt, ->]
    		(0,0) -- ({cos(\ang)},{sin(\ang)});
		
  	\draw[dotted, line width=1pt, ->]
    		(0,0) -- ({cos(180+\ang)},{sin(180+\ang)});
		
  	\draw[line width=0.5pt, ->]
    		(0,0) -- ({cos(\ang+\eps)},{sin(\ang+\eps)});
		\node[xshift=20pt, yshift=20pt] at ({cos(\ang+\eps)},{sin(\ang+\eps)}) {$[v_{1}^{+}]$};
		
  	\draw[line width=0.5pt, ->]
    		(0,0) -- ({cos(\ang+30*\eps)},{sin(\ang+30*\eps)});
		\node[xshift=-20pt, yshift=+0pt] at ({cos(\ang+30*\eps)},{sin(\ang+30*\eps)}) {$[v_{2}^{+}]$};
		
  	\draw[line width=0.5pt, ->]
    		(0,0) -- ({cos(180+\ang+\eps)},{sin(180+\ang+\eps)});
    		\node[xshift=+0pt, yshift=-20pt] at ({cos(180+\ang+\eps)},{sin(180+\ang+\eps)}) {$[v_{1}^{-}]$};	
		
  	\draw[line width=0.5pt, ->]
    		(0,0) -- ({cos(180+\ang+25*\eps)},{sin(180+\ang+25*\eps)});
    		\node[xshift=+15pt, yshift=-10pt] at ({cos(180+\ang+25*\eps)},{sin(180+\ang+25*\eps)}) {$[v_{2}^{-}]$};

\end{tikzpicture}

\caption{A figure of the $\omega_{T}$-ordering of a quadrilateral whose oriented edges are directed towards the various directions illustrated.  The dotted lines depict the invariant directions of the linear holonomy with the blue and orange depicting the positive and negative directions respectively.  Note the labeling is done in the order $[v_{1}^{+}], [v_{2}^{+}], [v_{1}^{-}], [v_{2}^{-}]$ starting from $[X_{T}]$ and traveling $\fD$ counter-clockwise.  
}\label{fig:omegatordering}
\end{center}
\end{figure}
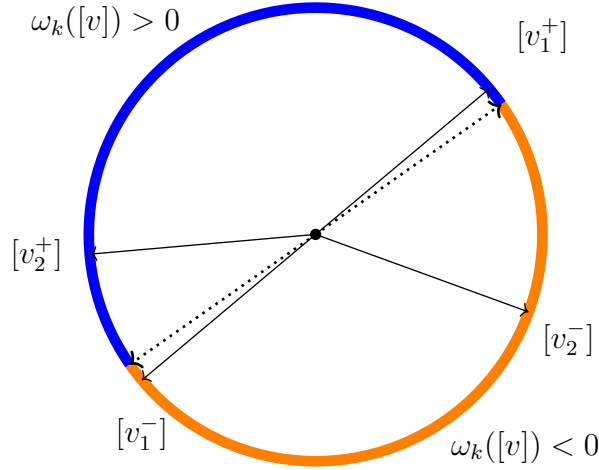

We conclude with the main theorem of this section which guarantees symplectic tiling billiards can be played on marked complete affine tilings that correspond to any pair of period parameters $k, K \in \MCQ$ in such a way that we do not have to worry about transversality issues.  Recall that two tilings $T_{A}$ and $T_{B}$ are said to be transverse if for any choice of edges $e_{A}$ and $e_{B}$, in $T_{A}$ and $T_{B}$ respectively, we have $e_{A}$ and $e_{B}$ are non-parallel.  Equivalently, we insist that $\fD_{T_{A}}, \fD_{T_{B}}$ are disjoint.  Our theorem below says that for any $k, K \in \MCQ$, we can find marked complete affine tilings which are transverse in a strong sense, and, their marked tiling representations define points in $\MCQ$ with parameters $k$ and $K$ respectively.  
\begin{theorem}\label{thm:goodtiles}
For any $k,K \in \MCQ$, there exist transverse marked complete affine tilings $T_{k}$ and $T_{K}$ of $\R^{2}$ whose marked tiling representations $\rho_{T_{k}}$ and $\rho_{T_{K}}$ determine the points $k, K \in \MCQ$ respectively.  If in addition $k = 0$ and $K \neq 0$, then the transverse tilings $T_{k}$ and $T_{K}$ can be chosen so that the invariant vector field directions $[\pm X_{T_{K}}]$ of the marked tiling representation are not contained in the edge directions $\fD_{T_{k}}$ of $T_{k}$.  If both $k, K \neq 0$, the transverse $T_{k}$ and $T_{K}$ can be chosen so that the invariant vector fields $X_{T_{k}}$ and $X_{T_{K}}$ are not parallel, and, $[\pm X_{T_{k}}] \notin \fD_{T_{K}}$ and $[\pm X_{T_{K}}] \notin \fD_{T_{k}}$.  
\end{theorem}
\begin{proof}
We proceed by cases.  If both $k = K = 0$, the claim is clear.  For $T_{k}$, take any Euclidean square tiling and pick the anchor of its marked tile $Q_{T_{k}}$.  Let $T_{K}$ be a small rotation of $T_{k}$ about this anchor.  These tilings are thus transverse and clearly Euclidean, thus the period parameters of $T_{k}$ and $T_{K}$ are both zero.
\\  
\\
If $k = 0$ and $K \neq 0$, we can choose the tiling $T_{k}$ to be tiled by the unit square marked by the oriented edge vector $e_{1}$ with anchor at the origin, and $T_{K}$ to be any marked complete affine tiling whose period parameter is $K$.  Such a choice exists by Theorem \ref{thm:marked_tiles}.  The edge directions of $T_{k}$, $\fD_{T_{k}} = \{[\pm e_{1}], [\pm e_{2}]\}$, consists of four directions whereas $\fD_{T_{K}}$ is a countably infinite subset of $\fD$ which contains the two-point set $[\pm X_{T_{K}}]$ in its closure inside $\fD$.  If $\fD_{T_{k}} \cap (\fD_{T_{K}}\cup \{[\pm X_{T_{K}}]\}) = \emptyset$, then we are done.  The tilings are transverse, and, the invariant vector field $X_{T_{K}}$ is not contained in the edge directions $\fD_{T_{k}}$.  \\
\\
If the intersection is non-empty, apply a rotation $R_{\theta}$ centered about the anchor of the marked edge of $T_{K}$ so that the tiling $R_{\theta} T_{K}$ is transverse to $T_{k}$, and, so that the vector $R_{\theta} X_{T_{K}}$ is not parallel to either $e_{1}$ or $e_{2}$.  Such a rotation exists as both direction sets are countable and the possible choices of angles is uncountable.  The new tiling $T_{K}' := R_{\theta} T_{K}$ has period parameters equal to $K$, as the new marked tiling representation $\rho_{T_{K}'}$ equals $R_{\theta}  \rho_{T_{K}}  R_{\theta}^{-1}$, and thus determines the same point in $\MCQ$.  Moreover, its invariant vector field $X_{T_{K}'} = R_{\theta} X_{T_{K}}$ is not parallel to either $e_{1}$ or $e_{2}$ so $[\pm X_{T_{K}'}] \notin \fD_{T_{k}}$.  This new pair of tilings $T_{k}$ and $T_{K}'$ satisfy the claims of the theorem. 
\\
\\
Finally, if both $k, K \neq 0$, then pick any two tilings $T_{k}$ and $T_{K}$ whose associated marked tiling representations have period parameters $k$ and $K$ respectively.  Let $X_{T_{k}}$ and $X_{T_{K}}$ be their invariant vector fields.  If $X_{T_{k}}$ and $X_{T_{K}}$ are parallel, or $[\pm X_{T_{K}}]$ is in $\fD_{T_{k}}$, apply a small rotation $R_{\theta}$ about the anchor of the marked edge of $T_{K}$ so that the new invariant vector field $X_{T_{K}'} = R_{\theta}X_{T_{K}}$ of the tiling $T_{K}' := R_{\theta}T_{K}$ is not parallel to $X_{T_{k}}$, and, not contained in $\fD_{T_{k}}$.  As seen in the previous paragraph, such a choice of $\theta$ exists as the number of possible rotations is uncountable, whereas $\fD_{T_{k}} \cup \{[\pm X_{T_{k}}]\}$ is countable.  The new tiling $T_{K}'$ still has the same period parameters as noted in the previous paragraph.  
\\
\\
By construction, $\fD_{T_{k}} \cup \{[\pm X_{T_{k}}]\}$ and $\{[\pm X_{T_{K}'}]\}$ are disjoint.  Thus it suffices to ensure that $(\fD_{T_{k}}\cup \{[\pm X_{T_{k}}]\}) \cap \fD_{T_{K}'} = \emptyset$.  We aim to change the tiling of $T_{K}'$ so that this is true.  Consider a unipotent subgroup of $\Affp(2,\R)$ which fixes the anchor of the marked edge of $T_{K}'$, and, the direction $X_{T_{K}'}$.  This subgroup of $\Affp(2,\R)$ is isomorphic to $\R$, and thus uncountable.  The sets $\fD_{T_{k}}\cup \{[\pm X_{T_{k}}]\}$ and $\fD_{T_{K}'}$ are both countable and this unipotent subgroup acts simply transitively on both the positive and negative directions of $\fD_{T_{K}'}^{+}$ and $\fD_{T_{K}'}^{-}$ determined by $\omega_{T_{K}'}$ while fixing $[\pm X_{T_{K}'}]$.  By the same uncountability argument, there exists a unipotent which preserves $X_{T_{K}'}$ and takes $\fD_{T_{K}'}$ to a subset disjoint from $\fD_{T_{k}} \cup \{[\pm X_{T_{k}}]\}$.  Applying this choice of unipotent to $T_{K}'$ then yields the desired pair of tilings thus completing our proof.
\end{proof}


\section{Divergent Orbits}\label{sec:divorbs}
The main theorem of Section \ref{sec:cat} justifies that tilings in $\MCQ$ are a natural context to play symplectic tiling billiards in.  Theorem \ref{thm:goodtiles} guarantees that we may generically ensure that tilings have transverse edge sets, so dynamically the only possible end to the trajectories occur by hitting vertices.  By Lemma \ref{lem:gen_good}, this is a measure zero subset of configurations, thus generically any configuration we pick on these strongly transverse pairs of tilings will have trajectories defined for all possible iterations.  \\
\\
In this last section we analyze some of the trajectories of symplectic tiling billiards in this context.  We begin with remarks on the shapes of marked tiles as one goes out `far' in the tiling and introduce a definition of what it means for a tile to be `thin'.  We also make precise the definition of divergence we will use in our context that is adapted to the affine geometry at hand.  We then leverage the degenerating shape of the tiles as one goes out far away in the tiling to show the dynamics of the system become predictable and must eventually remain in some positive cone as defined by the invariant one-form of the tiling.  We separate our analysis into two cases.  The first when one tiling is Euclidean and the other is non-Euclidean, and the second is when both are non-Euclidean.  The arguments will be carried out in thorough detail in the first case and then applied once more to the second with condensed arguments, as they follow the same principles as the first case.  Illustrations and experimentations will be addressed to both support, and guide, our arguments.  The software developed for these experiments can be found on the authors' webpages.\\
\\
Before proceeding, we give a brief overview of the arguments ahead to guide the reader.  Consider the case where the lefthand tiling $T_{k}$ is the standard Euclidean square tiling and the right hand tiling $T_{K}$ is non-Euclidean where the invariant directions $[\pm X_{T_{K}}]$ of $T_{K}$ are not parallel to the edges of $T_{k}$, $[\pm e_{1}]$ and $[\pm e_{2}]$.  For the ease of argument say $[e_{1}], [e_{2}]$ are positive directions relative to $\omega_{T_{K}}$.\\
\\
As one can appreciate from Figure \ref{fig:foliate}, the tiles of $T_{K}$ can become very thin.  Pick one of these thin tiles and call it $Q$.  By the fact that these tiles have edges almost parallel to $[\pm X_{T_{K}}]$, this means the edge directions of $Q$ are very close to $[\pm X_{T_{K}}]$ in the circle of directions.  So if we pick a non-negative edge $v_{i}^{+}$ of $Q$, it must be very close to $X_{T_{K}}$ in $\fD$ so that both $\det(v_{i}^{+}, e_{1})$ and $\det(v_{i}^{+}, e_{2})$ are bigger than $0$ because $e_{1}, e_{2}$ are both positive directions.  Thus if our particle starts on the edge $v_{i}^{+}$ there are only two possible ways for our trajectory to move: either right along $e_{1}$ or up along $e_{2}$.  Going along either direction, the only possible edges to hit are going to be the ones almost parallel to $[-X_{T_{K}}]$ by basic orientation considerations.  The new marked tile will label this intersected edge non-negative, and if it is still very close to $X_{T_{K}}$ in $\fD$, then again the trajectory can only move along either $e_{1}$ or $e_{2}$.  This behavior is exhibited in Figure \ref{fig:symtil1}.  Thus we end up in the positive cone defined by $e_{1}$ and $e_{2}$.  So long as we move through tiles that have their edges never leaving a small neighborhood of $[\pm X_{T_{K}}]$ the trajectories our orbits must diverge.  Controlling the shape of the tiles as we move along positive directions is the content of Lemma \ref{lem:positivemoves} and relies heavily on the rationality hypothesis of our tiling $T_{K}$ as defined in Definition \ref{def:rattor}.  We argue that by only moving in positive directions, eventually at worst, we will end up in a tile that is a translate of what we originally started with and may thus apply the same argument.  The case where both tilings are non-Euclidean follows more or less the principles.  

\subsection{Positively $P$-thin Quadrilaterals}\label{ssec:pthin}
To begin, we introduce a way to measure divergence that is compatible with the canonical invariants guaranteed by Lemma \ref{lem:caninvs}.  This notion of divergence is stated in terms of the invariant oriented foliation by flow lines parallel to $X_{T_{K}}$ as seen in Figure \ref{fig:foliate}.  
\begin{definition}\label{def:divpath}
Let $\alpha: [0,\infty) \lra T_{K}$ be an infinite continuous path in the marked complete affine tiling $T_{K}$ of $\R^{2}$ with period parameter $K$.  We say that $\alpha$ diverges \emph{positively} along the canonical parallel invariant one-form $\omega_{T_{K}}$ if $\lim_{t\to\infty} \int_{\alpha_{t}} \omega_{T_{K}} = \infty$ where $\alpha_{t}$ denotes the image of $\alpha$ on $[0,t]$.  Similarly we say $\alpha$ diverges negatively along $\omega_{T_{K}}$ if $\lim_{t\to\infty} \int_{\alpha_{t}} \omega_{T_{K}} = -\infty$
\end{definition}

The main result of this paper is to prove the existence of configurations of this game for tilings $(T_{k},T_{K})$ as in Theorem \ref{thm:goodtiles} for which the trajectories diverge positively along their respective one forms, and, show these configurations are stable.  We split the argument into two types: when $(T_{k},T_{K})$ is Euclidean and non-Euclidean, and, when both are non-Euclidean.  It is possible to combine these arguments into a single case, however, in the interest of clarity, we consider them separately.  In what follows for now, one should think of $T_{k}$ as the standard Euclidean square tiling whose edges are parallel to the $x$ and $y$ axes and $T_{K}$ as a non-Euclidean tiling whose canonical vector field $X_{T_{K}}$ is not parallel to either the $x$ or $y$ axes, and, that the tilings are transverse.  Such tilings exist by Theorem \ref{thm:goodtiles}.  
\\
\\
The main insight into the existence of divergent orbits is as follows.  For a non-Euclidean marked complete affine tiling, if $\alpha$ is a positively divergent trajectory relative to $\omega_{T_{K}}$, then the corresponding tiles that $\alpha$ goes through must become very thin.  In fact, their oriented edges limit on directions determined by either $[\pm X_{T_{K}}]$.  This can be appreciated in Figure \ref{fig:limiting} and Figure \ref{fig:almostparallel} below.  This limiting behavior is a consequence of the parabolic linear holonomy of the tiling and we make this observation precise in the following lemma.  

\begin{figure}
	\begin{minipage}{0.3\textwidth}
		\includegraphics[width=\textwidth]{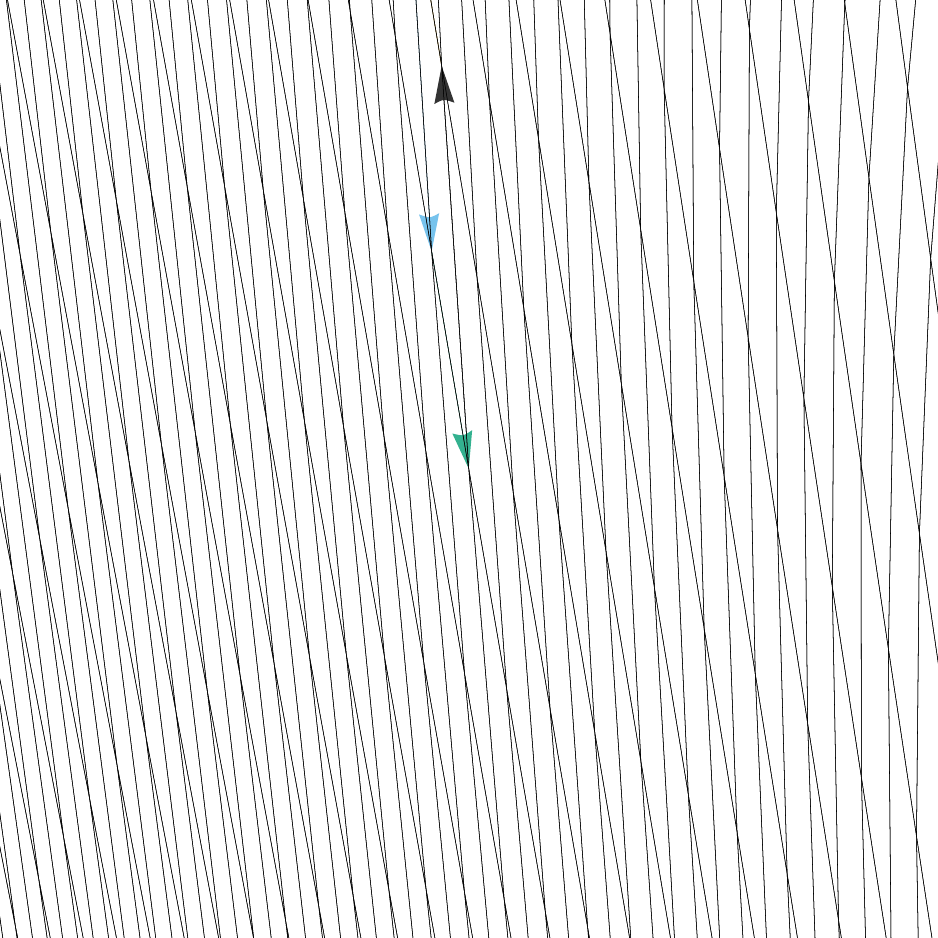}
	\end{minipage}
	\hfill
	\begin{minipage}{0.3\textwidth}
		\includegraphics[width=\textwidth]{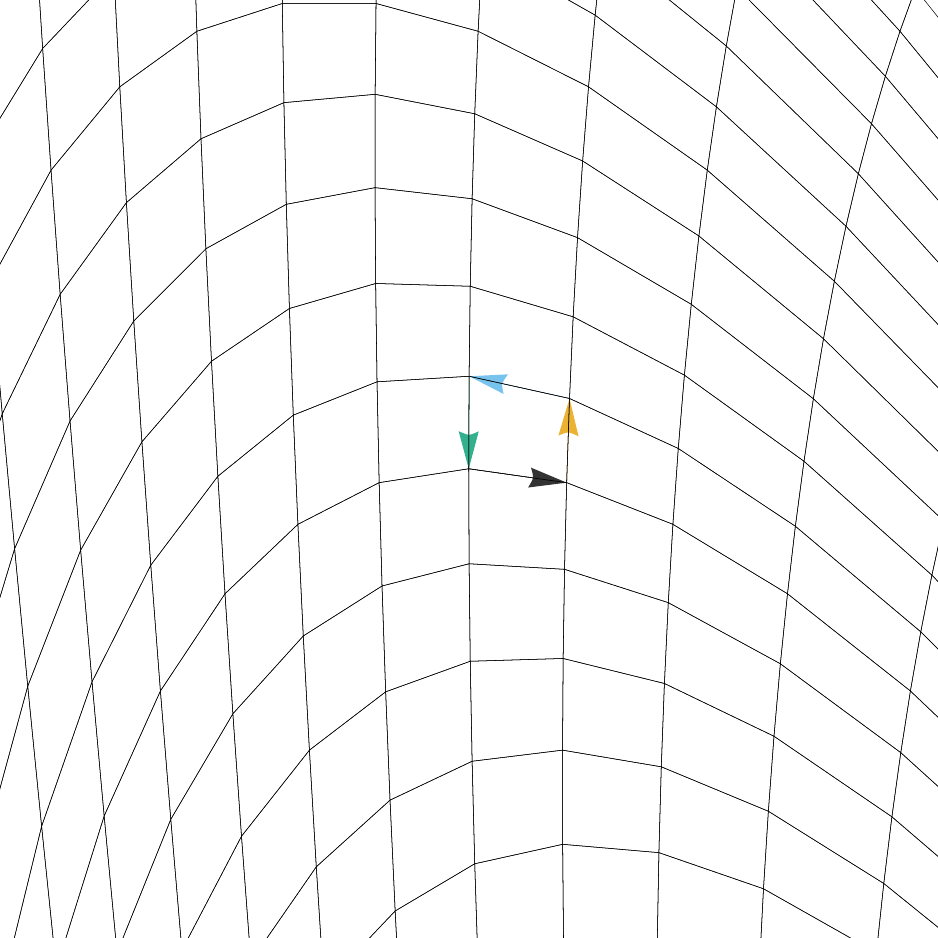}
	\end{minipage}
	\hfill
	\begin{minipage}{0.3\textwidth}
		\includegraphics[width=\textwidth]{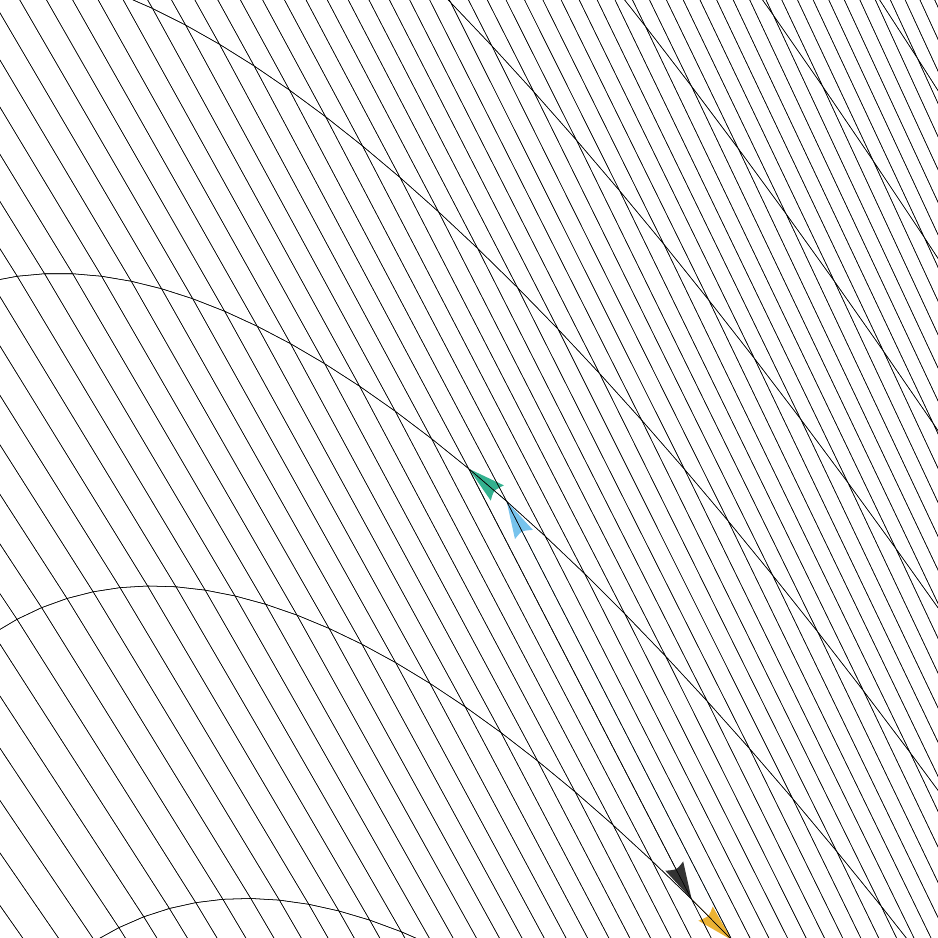}
	\end{minipage}
	\caption{Some views of a marked tiling corresponding to period parameters $K = (2/3,1/5)$.  The middle figure shows the marked tile $Q$.  The left figure shows that same tiling at the same scale moved $15$ tiles to the left where the image of $Q$ is marked.  The right figure shows the middle figure moved $15$ tiles to the right and the image of $Q$ is marked.  The tilings become appreciably thinner and the edges become more parallel limited in the direction determined by $[(1/5,-2/3)] \in \fD$.}\label{fig:limiting} 
\end{figure}

\begin{figure}
\includegraphics[scale=0.5]{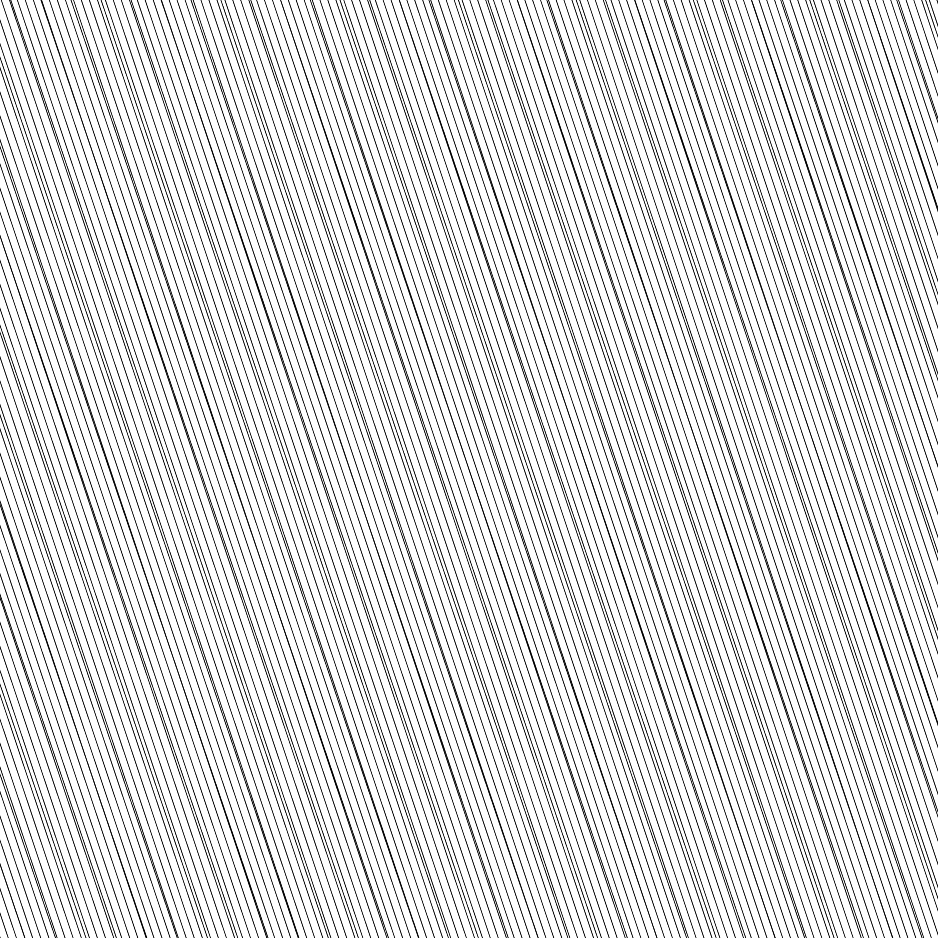}
\caption{A view of the tiling in Figure \ref{fig:limiting} with period parameter $K = (2/3,1/5)$.  These tiles are about $100$ tiles to the right of the original middle marked tiling from that figure.  The oriented edges are all approaching a direction parallel either $[\pm(1/5,-2/3)] = [\pm X_{T_{K}}]$.  Figure \ref{fig:foliate} displays the leaves of the foliation the edges are approaching.}
\label{fig:almostparallel}
\end{figure}
%
%

\begin{lemma}\label{lem:getthin}
Let $T_{K}$ be a marked complete affine tiling of $\R^{2}$ with canonical parallel invariant vector field and one-form $X_{T_{K}}$ and $\omega_{T_{K}}$.  Let $\rho_{T_{K}}: \Z^{2} \lra \Affp(2,\R)$ denote the marked tiling representation of $T_{K}$, and denote $L_{T_{K}}(\gamma)$ and $T_{T_{K}}(\gamma)$ the linear and translational parts of $\rho_{T_{K}}(\gamma)$.  \\
\\
For any sequence of $\omega_{T_{K}}$-positive loops $\gamma_{n} \in \Z^{2}$ for which $\lim_{n\to\infty} \omega_{T_{K}}(T_{T_{K}}(\gamma_{n})) = \infty$, we have that $\lim_{n\to\infty} L_{T_{K}}(\gamma_{n})([v]) = [X_{T_{K}}]$ if $[v] \in \fDnnoK$ and $\lim_{n\to\infty} L_{T_{K}}(\gamma_{n})([v])= [-X_{T_{K}}]$ if $[v] \in \fDnpoK$ where here $\fD = \fDnnoK\sqcup \fDnpoK$ is the decomposition of the circle of directions $\fD$ into non-negative and non-positive directions as determined by $\omega_{T_{K}}$.  
\end{lemma}

\begin{proof}
To simplify notation, let us simply denote the tiling by $T$ instead of $T_{K}$.  Because $T \in \MCQ$, there exists an affine transformation $g \in \Affp(2,\R)$ for which $\rho_{T}$ and $\rho_{K}$ are conjugate via $g$ where $\rho_{K}$ is the holonomy representation as defined in Equation \ref{eq:hol_rep}.  That is, for all $\gamma \in \Z^{2}$ we have the equalities $g\circ \rho_{T}(\gamma) \circ g^{-1} = \rho_{K}(\gamma)$.  We thus have the following equalities in linear and translational terms
\begin{equation}\label{eq:lintraneq}
L(g)\circ L_{T}(\gamma) \circ L(g)^{-1} = L_{K}(\gamma) \text{ and } L(g)T_{T}(\gamma) + \left(\text{id} - L_{K}(\gamma)\right)T(g) = T_{K}(\gamma)
\end{equation}
Next observe by Equation \ref{eq:hol_rep} and Equation \ref{eq:parone} that for any $\gamma_{n} \in \Z^{2}$ we have the equality below.
\begin{equation}\label{eq:dotequals}
\omega_{K}(T_{K}(\gamma_{n})) = K\cdot \gamma_{n}
\end{equation}
Thus, for any sequence $\gamma_{n} \in \Z^{2}$ of $\omega_{K}$-positive loops for which $\lim_{n\to\infty}\omega_{K}(T_{K}(\gamma_{n}))= \infty$, we can see from Equation \ref{eq:hol_rep} that $\lim_{n\to\infty}L_{K}(\gamma_{n})[v] = [-K^{\perp}]$ for any $\omega_{K}$-non-negative direction $[v]$ because the term $K\cdot \gamma_{n}$ in Equation \ref{eq:hol_rep} tends to infinity by Equation \ref{eq:dotequals}.  As seen in Equation \ref{eq:parvec}, $-K^{\perp}$ and $X_{K}$ determine the same direction in $\fD$.  Thus, if we fix a direction $[v] \in \fDnnK$, and have a sequence $\gamma_{n} \in \Z^{2}$ of $\omega_{K}$-positive loops for which $\lim_{n\to\infty}\omega_{K}(T_{K}(\gamma_{n}))= \infty$, then $\lim_{n\to\infty}L_{K}(\gamma_{n})[v] = [X_{K}]$.  The analogous statement holds for non-positive directions.  \\
\\
We now use the conjugation relations in Equation \ref{eq:lintraneq} to obtain our result.  Let $[v]$ be an $\omega_{T}$-non-negative direction and $\gamma_{n} \in \Z^{2}$ be a sequence of $\omega_{T}$-positive loops for which $\lim_{n\to\infty}\omega_{T}(T_{T}(\gamma_{n}))= \infty$.  We aim to show that $\lim_{n\to\infty}L_{T}(\gamma_{n})[v] = [X_{T}]$.  To this end, because $[v]$ is $\omega_{T}$-non-negative, it follows because $\omega_{T} = L(g)^{*}\omega_{K} = \omega_{K}\circ L(g)$, that $L(g)[v]$ is $\omega_{K}$-non-negative.  Thus if $\gamma_{n}$ is also a sequence of $\omega_{K}$-positive loops whose pairing diverges to infinity, then by the previous paragraph, we would have the following equalities.  
\begin{equation*}
\lim_{n\to\infty}(L(g)\circ L_{T}(\gamma_{n}))([v])  = \lim_{n\to\infty} (L_{K}(\gamma_{n})\circ L(g))([v]) = [X_{K}] = [L(g)X_{T}]
\end{equation*}
The outer-most equalities would then imply the desired result, $\lim_{n\to\infty} L_{T}(\gamma_{n})([v]) = [X_{T}]$.  Thus, we only need to show that $\gamma_{n}$ is also a sequence of $\omega_{K}$-positive loops diverging to infinity.  This follows from the equations below using Equation \ref{eq:lintraneq} and Equation \ref{eq:hol_rep}.  
\begin{align*}
\lim_{n\to\infty} \omega_{K}(T_{K}(\gamma_{n})) &= \lim_{n\to\infty} \omega_{K}\left(L(g)T_{T}(\gamma_{n}) + \left(\text{id} - L_{K}(\gamma_{n})\right)T(g)\right) = \lim_{n\to\infty} \omega_{K}(L(g)T_{T}(\gamma_{n}))\\
&= \lim_{n\to\infty} \omega_{T}(T_{T}(\gamma_{n}))
\end{align*}
\end{proof}

By Lemma \ref{lem:getthin} if we move far out in our marked tiling along tiles $\gamma_{n} Q$ where $\lim_{n \to\infty} \omega_{T_{K}}(T_{T_{K}}(\gamma_{n}))$ $= \infty$, this means our tiles will look very thin.  We now introduce a definition that allows us to define thinness of these quadrilaterals relative to another one.    

\begin{definition}\label{def:pthin}
Let $P$ be an unoriented quadrilateral in $\R^{2}$ and $T_{K}$ be a complete affine tiling with invariant parallel area form $d\Omega_{T_{K}}$ and one-form $\omega_{T_{K}}$.  We say a marked tile $Q$ of $T_{K}$ is \emph{$\omega_{T_{K}}$-positively $P$-thin} if there exists a choice of positive directions $[d_{1}],\hdots, [d_{4}] \in \fD_{\omega_{T_{K}}}^{+} \subset \fD$ representing the edges of $P$ for which $d\Omega_{T_{K}}(w_{j}^{+},d_{i}) > 0$ for each non-negative edge $w_{j}^{+}$ of $Q$ and all $i = 1, \hdots, 4$, and, $d\Omega_{T_{K}}(w_{j}^{-},d_{i}) < 0$ for each non-positive edge $w_{j}^{-}$ of $Q$ and all $i = 1 ,\hdots, 4$.  
\end{definition}
Note that because there is only one parallel area form compatible with the standard orientation of $\R^{2}$ up to positive scalar multiple, Definition \ref{def:pthin} is the same if we replace $d\Omega_{T_{K}}$ with the determinant which we will work with instead.  Another way of saying Definition \ref{def:pthin} is to say that we may choose $\omega_{T_{K}}$-positive representatives of the edges of $P$ so that for each oriented edge of $Q$, all the representatives lie on one side of the edge, or, the other.  Figure \ref{fig:pthin} and Figure \ref{fig:notpthin} below illustrate this concept when $P$ is a square.

\begin{figure}
\includegraphics[scale=0.35]{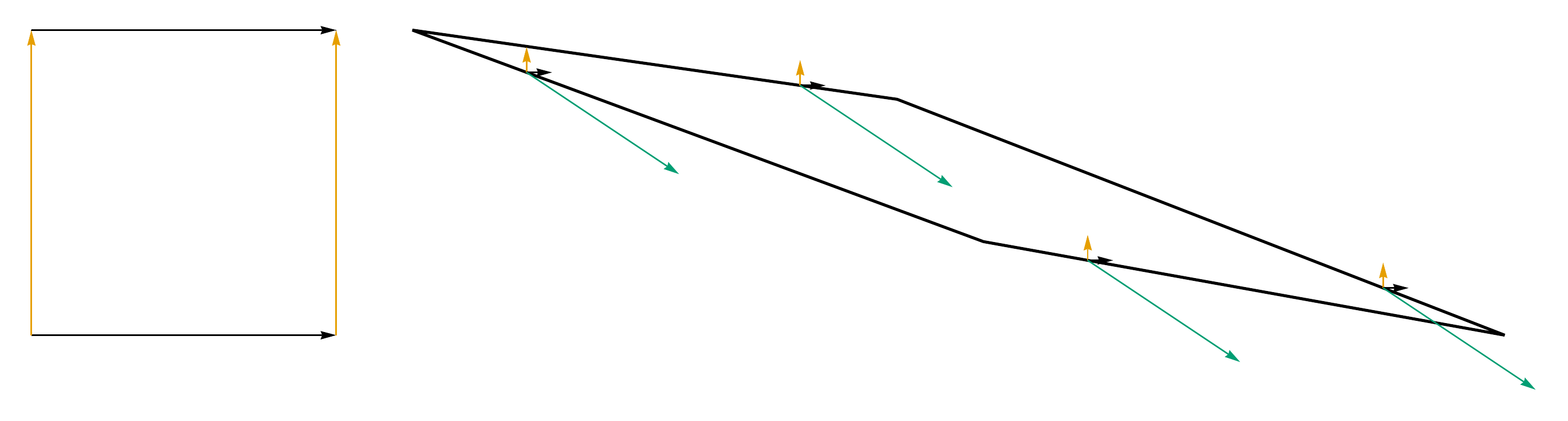}
\caption{An example of a positively $P$-thin marked quadrilateral $Q$ which is a fundamental domain for the affine torus with period parameters $K = (2/5,3/5)$.  Its invariant vector field $X_{K}$ is shown in green to illustrate our choices of directions are indeed $\omega_{K}$-positive.  We choose fixed positive representative directions of the edges of $P$, in orange and black, so that for each edge of $Q$, all these directions lie on one side of the edge, or, the other.  
}\label{fig:pthin}
\end{figure}

\begin{figure}
\includegraphics[scale=0.4]{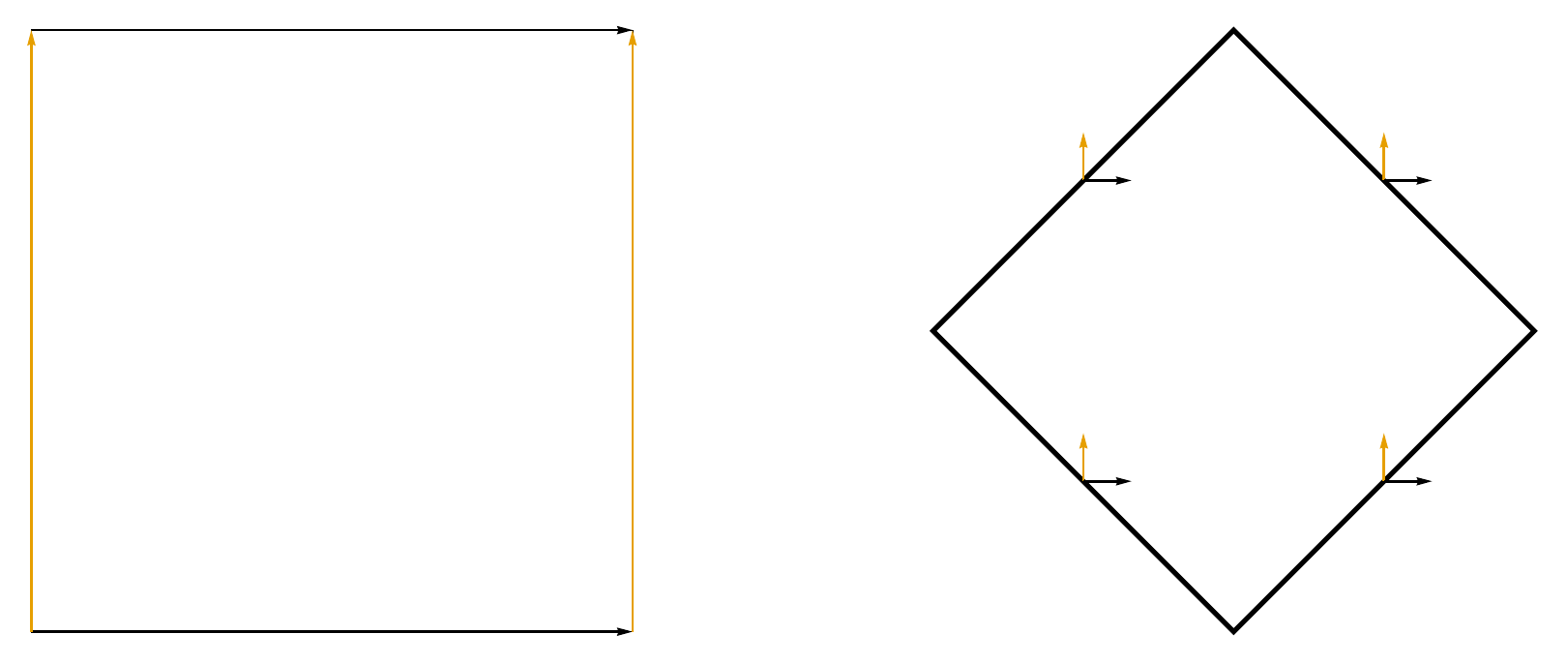}
\caption{A non-example of a positively $P$-thin marked quadrilateral $Q$.  There is no choice of positive representative directions of the edges of $P$ so that for each edge of $Q$, all these directions lie on one side of the edge, or, the other.  
}\label{fig:notpthin}
\end{figure}

\subsection{Divergence}\label{ssec:div}
With these preliminaries established, we specialize to the case where we consider a pair of Euclidean and non-Euclidean marked tilings $(T_{k},T_{K})$.  As mentioned in Section \ref{sec:intro}, the dynamics of this game are invariant under simultaneous $\Affp(2,\R)$-symmetries of the tilings.  That is, there is an orbit equivalence of the dynamics of the game played on $(T_{k},T_{K})$ and $(gT_{k},gT_{K})$ where $g$ is any affine symmetry.  Thus, without loss of generality, we may assume that $T_{k}$ is the marked affine tiling given by the unit square anchored at the origin and marked by the edge $e_{1}$.  Denote this marked square by $P$ and observe the marked tiling representations act by translation to the right one unit by $e_{1}$ and translation up one unit by $e_{2}$.  We now prove for any other non-Euclidean tiling $T_{K}$, we can find positively $P$-thin tiles where the nearby tiles also remain $P$-thin.

\begin{lemma}\label{lem:sqthinexists}
Let $T_{k}$ be the standard marked Euclidean tiling and $T_{K}$ be a marked non-Euclidean tiling whose canonical vector field $X_{T_{K}}$ is not parallel to either $e_{1},e_{2}$, and the tilings $T_{k}$ and $T_{K}$ are transverse.  Let $P$ be the square marked by $T_{k}$.  Then there exists a positively $P$-thin marked tile $Q$ of $T_{K}$.  Moreover, for any integer $N > 0$, one can choose $Q$ so that $\gamma Q$ is still $P$-thin for all $\gamma = (n,m)$ with $|n| + |m| < N$.  
\end{lemma}

\begin{proof}
By hypothesis $[X_{T_{K}}] \neq [\pm e_{1}], [\pm e_{2}]$ in the circle of directions $\fD$.  Thus we can find a choice of $[d_{1}] := [\pm e_{1}]$ and $[d_{2}] := [\pm e_{2}]$ which are both positive with respect to $\omega_{T_{K}}$, that is $[d_{1}], [d_{2}] \in \fD_{\omega_{K}}^{+}$.  By definition of $\omega_{T_{K}}$, we have that $\omega_{T_{K}}(v) :=  d\Omega_{T_{K}}(X_{T_{K}},v)$ where $d\Omega_{T_{K}}$ is the parallel invariant area-form as defined in Lemma \ref{lem:caninvs}.  As mentioned after Definition \ref{def:pthin}, $d\Omega_{T_{K}}$ equals $\det$ up to positive scale thus $v$ is $\omega_{T_{K}}$-positive if and only if $\det(X_{T_{K}},v) > 0$.  That is to say $(X_{T_{K}},v)$ is an oriented basis of $\R^{2}$.  
\\
\\
Consider the marked tile $Q$ of $T_{K}$.  By choosing a sequence $\gamma_{n} \in \Z^{2}$ so that $\omega_{T_{K}}(T_{T_{K}}(\gamma_{n}))$ tends to infinity, we may make the directions of the edges of $\gamma_{n}Q$, $[w_{1}],\hdots, [w_{4}]$, arbitrarily close to $[\pm X_{T_{K}}]$ in the circle of directions $\fD$ by Lemma \ref{lem:getthin}.  Thus, we may pick $\gamma_{n}$ so that the non-negative directions $[w_{1}^{+}], [w_{2}^{+}]$ of $\gamma_{n}Q$ lie between both $[X_{T_{K}}]$ and $[d_{1}]$, and, both $[X_{T_{K}}]$ and $[d_{2}]$ in $\fD$ in addition to having the non-positive directions $[w_{1}^{-}], [w_{2}^{-}]$ lying between both $[-X_{T_{K}}]$ and $[-d_{1}]$, and, both $[-X_{T_{K}}]$ and $[-d_{2}]$.  Figure \ref{fig:getsmall} illustrates this argument.    
\\
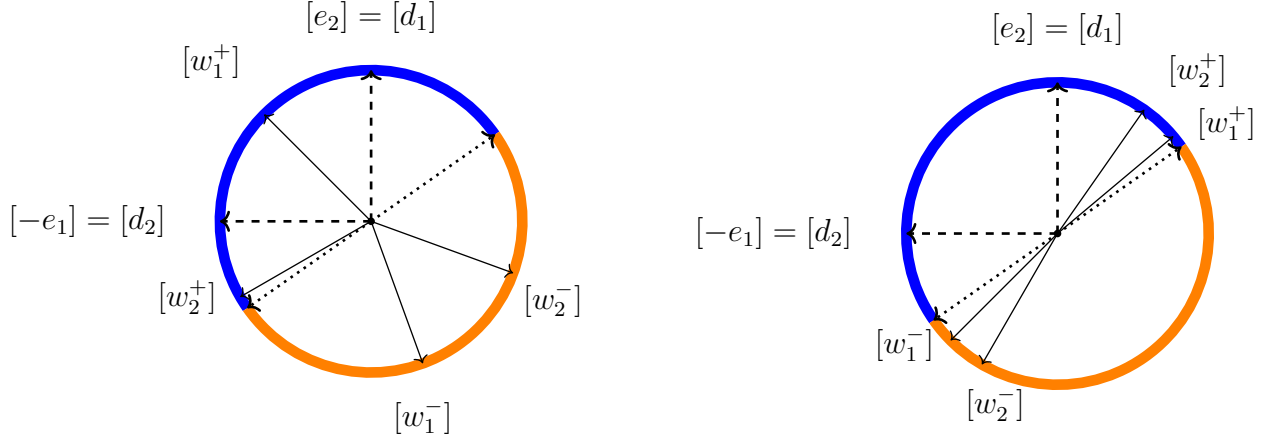
\begin{figure}
    \begin{minipage}{0.45\textwidth}
 	   \begin{center}
            \begin{tikzpicture}[scale=2]
            
            	\def\ang{35}
            	\def\eps{5}
            
            	\draw[thick] (0,0) circle (1);
            	
            	\draw[line width=4pt, oskyblue]
            		({cos(\ang)},{sin(\ang)})
            		arc[
            			start angle=\ang,
            			end angle=180+\ang,
            			radius=1
            		];
            
            	\draw[line width=4pt, oorange]
            		({cos(180+\ang)},{sin(180+\ang)})
            		arc[
            			start angle=180+\ang,
            			end angle=360+\ang,
            			radius=1
            		];
            		
            	\fill (0,0) circle (0.025);
            	
            	\draw[dotted, line width=1pt, ->]
            		(0,0) -- ({cos(\ang)},{sin(\ang)});
            		
            	\draw[dotted, line width=1pt, ->]
            		(0,0) -- ({cos(180+\ang)},{sin(180+\ang)});
		
            	\draw[dashed, line width=1pt, ->]
            		(0,0) -- ({cos(90)},{sin(90)});
			\node[xshift=0pt, yshift=20pt] at ({cos(90)},{sin(90)}) {$[e_{2}] = [d_{1}]$};
            		
            	\draw[dashed, line width=1pt, ->]
            		(0,0) -- ({cos(180)},{sin(180)});
			\node[xshift=-50pt, yshift=0pt] at ({cos(180)},{sin(180)}) {$[-e_{1}] = [d_{2}]$};
            		
            	\draw[line width=0.5pt, ->]
            		(0,0) -- ({cos(\ang+20*\eps)},{sin(\ang+20*\eps)});
            		\node[xshift=-20pt, yshift=20pt] at ({cos(\ang+20*\eps)},{sin(\ang+20*\eps)}) {$[w_{1}^{+}]$};
            		
            	\draw[line width=0.5pt, ->]
            		(0,0) -- ({cos(\ang+35*\eps)},{sin(\ang+35*\eps)});
            		\node[xshift=-20pt, yshift=+0pt] at ({cos(\ang+35*\eps)},{sin(\ang+35*\eps)}) {$[w_{2}^{+}]$};
            		
            	\draw[line width=0.5pt, ->]
            		(0,0) -- ({cos(180+\ang+15*\eps)},{sin(180+\ang+15*\eps)});
            		\node[xshift=+0pt, yshift=-20pt] at ({cos(180+\ang+15*\eps)},{sin(180+\ang+15*\eps)}) {$[w_{1}^{-}]$};	
            		
            	\draw[line width=0.5pt, ->]
            		(0,0) -- ({cos(180+\ang+25*\eps)},{sin(180+\ang+25*\eps)});
            		\node[xshift=+15pt, yshift=-10pt] at ({cos(180+\ang+25*\eps)},{sin(180+\ang+25*\eps)}) {$[w_{2}^{-}]$};
            
            \end{tikzpicture}
    \end{center}
    \end{minipage}
    \hfill
    \begin{minipage}{0.45\textwidth}
    \begin{center}
    	\begin{tikzpicture}[scale=2]
            
            	\def\ang{35}
            	\def\eps{5}
            
            	\draw[thick] (0,0) circle (1);
            	
            	\draw[line width=4pt, oskyblue]
            		({cos(\ang)},{sin(\ang)})
            		arc[
            			start angle=\ang,
            			end angle=180+\ang,
            			radius=1
            		];
            
            	\draw[line width=4pt, oorange]
            		({cos(180+\ang)},{sin(180+\ang)})
            		arc[
            			start angle=180+\ang,
            			end angle=360+\ang,
            			radius=1
            		];
            		
            	\fill (0,0) circle (0.025);
            	
            	\draw[dotted, line width=1pt, ->]
            		(0,0) -- ({cos(\ang)},{sin(\ang)});
            		
            	\draw[dotted, line width=1pt, ->]
            		(0,0) -- ({cos(180+\ang)},{sin(180+\ang)});
		
            	\draw[dashed, line width=1pt, ->]
            		(0,0) -- ({cos(90)},{sin(90)});
			\node[xshift=0pt, yshift=20pt] at ({cos(90)},{sin(90)}) {$[e_{2}] = [d_{1}]$};
            		
            	\draw[dashed, line width=1pt, ->]
            		(0,0) -- ({cos(180)},{sin(180)});
			\node[xshift=-50pt, yshift=0pt] at ({cos(180)},{sin(180)}) {$[-e_{1}] = [d_{2}]$};
            		
            	\draw[line width=0.5pt, ->]
            		(0,0) -- ({cos(\ang+1*\eps)},{sin(\ang+1*\eps)});
            		\node[xshift=+20pt, yshift=+5pt] at ({cos(\ang+1*\eps)},{sin(\ang+1*\eps)}) {$[w_{1}^{+}]$};
            		
            	\draw[line width=0.5pt, ->]
            		(0,0) -- ({cos(\ang+4*\eps)},{sin(\ang+4*\eps)});
            		\node[xshift=+20pt, yshift=15pt] at ({cos(\ang+4*\eps)},{sin(\ang+4*\eps)}) {$[w_{2}^{+}]$};
            		
            	\draw[line width=0.5pt, ->]
            		(0,0) -- ({cos(180+\ang+2*\eps)},{sin(180+\ang+2*\eps)});
            		\node[xshift=-25pt, yshift=+5pt] at ({cos(180+\ang+4*\eps)},{sin(180+\ang+4*\eps)}) {$[w_{1}^{-}]$};	
            		
            	\draw[line width=0.5pt, ->]
            		(0,0) -- ({cos(180+\ang+5*\eps)},{sin(180+\ang+5*\eps)});
            		\node[xshift=+5pt, yshift=-15pt] at ({cos(180+\ang+5*\eps)},{sin(180+\ang+5*\eps)}) {$[w_{2}^{-}]$};
            
            \end{tikzpicture}
    \end{center}
    \end{minipage}
\caption{On the left we illustrate the choice of $[d_{1}],[d_{2}]$ as dashed thick arrows which are $\omega_{T_{K}}$-positive.  The dotted thin arrows indicate the invariant directions, and the blue and orange sections indicate the positive and negative directions respectively.  The various oriented edges of the $\omega_{T_{K}}$-ordering are illustrated on the left for the marked tile $Q$.  On the right, we have the marked tile $\gamma Q$ after an application of a large positive loop which brings both non-negative edges between $[X_{T_{K}}]$ and $[d_{1}]$ and $[X_{T_{K}}]$ and $[d_{2}]$ whereas both non-positive edges get sent between $[-X_{T_{K}}]$ and $[-d_{1}]$ and $[-X_{T_{K}}]$ and $[-d_{2}]$.  
}\label{fig:getsmall}
\end{figure}

Declare this new marked tile to be $Q$.  By construction we then have the following inequalities.
\begin{equation*}
\det(w_{1}^{+}, d_{1}), \det(w_{2}^{+}, d_{2}) > 0 \text{ and } \det(w_{1}^{-}, d_{1}), \det(w_{2}^{-}, d_{2}) < 0
\end{equation*}
By Definition \ref{def:pthin}, $Q$ is a positively $P$-thin tile.  For any fixed integer $N > 0$, if we wish for $\gamma Q$ to remain $P$-thin for each $(n,m) \in \Z^{2}$ for which $|n| + |m| < N$, we simply take the sequence of positive loops $\gamma_{n}$ to be further out.  For positive loops $\gamma$, $\gamma Q$ is still $P$-thin as the non-negative directions of $Q$ remain non-negative and get closer to $X_{T_{K}}$ in $\fD$, and similarly for the non-positive directions of $Q$.  Figure \ref{fig:getsmaller} illustrates these dynamics.  

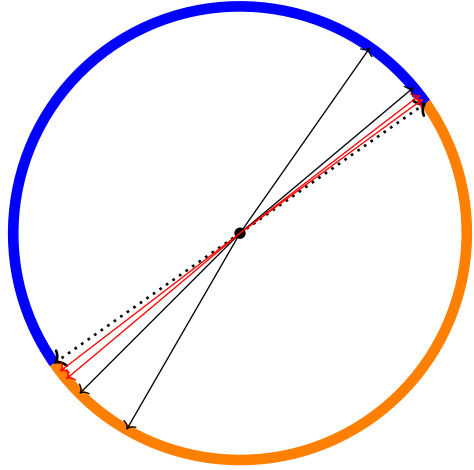
\begin{figure}
\begin{center}
    	\begin{tikzpicture}[scale=3]
            
            	\def\ang{35}
            	\def\eps{5}
            
            	\draw[thick] (0,0) circle (1);
            	
            	\draw[line width=4pt, oskyblue]
            		({cos(\ang)},{sin(\ang)})
            		arc[
            			start angle=\ang,
            			end angle=180+\ang,
            			radius=1
            		];
            
            	\draw[line width=4pt, oorange]
            		({cos(180+\ang)},{sin(180+\ang)})
            		arc[
            			start angle=180+\ang,
            			end angle=360+\ang,
            			radius=1
            		];
            		
            	\fill (0,0) circle (0.025);
            	
            	\draw[dotted, line width=1pt, ->]
            		(0,0) -- ({cos(\ang)},{sin(\ang)});
            		
            	\draw[dotted, line width=1pt, ->]
            		(0,0) -- ({cos(180+\ang)},{sin(180+\ang)});
            		
            	\draw[line width=0.5pt, ->]
            		(0,0) -- ({cos(\ang+1*\eps)},{sin(\ang+1*\eps)});
            		
            	\draw[line width=0.5pt, ->]
            		(0,0) -- ({cos(\ang+4*\eps)},{sin(\ang+4*\eps)});
		
            	\draw[line width=0.5pt, overm, ->]
            		(0,0) -- ({cos(\ang+0.5*\eps)},{sin(\ang+0.5*\eps)});
		\draw[line width=0.5pt, overm, ->]
            		(0,0) -- ({cos(\ang+0.25*\eps)},{sin(\ang+0.25*\eps)});
            		
            	\draw[line width=0.5pt, ->]
            		(0,0) -- ({cos(180+\ang+2*\eps)},{sin(180+\ang+2*\eps)});
            		
            	\draw[line width=0.5pt, ->]
            		(0,0) -- ({cos(180+\ang+5*\eps)},{sin(180+\ang+5*\eps)});
		
            	\draw[line width=0.5pt, overm, ->]
            		(0,0) -- ({cos(180+\ang+1*\eps)},{sin(180+\ang+1*\eps)});
		\draw[line width=0.5pt, overm,  ->]
            		(0,0) -- ({cos(180+\ang+0.5*\eps)},{sin(180+\ang+0.5*\eps)});
            
            \end{tikzpicture}
\end{center}
\caption{The figure above illustrates the quadrilateral edge directions of $Q$ in black of the right-hand side of Figure \ref{fig:getsmall}.  The red arrows are the new edge directions of $\gamma Q$ after the application of a very large positive loop $Q$ so that the edge directions get closer to $[\pm X_{T_{K}}]$ and still remain positively $P$-thin.  }
\end{figure}\label{fig:getsmaller}
\end{proof}

Lemma \ref{lem:sqthinexists} proves the existence of tiles which are well-adapted to our dynamical system.  In fact, the proof that there are divergent trajectories in our context is in some sense a result about the stability of a marked tile $Q$ remaining $P$-thin after the application of positive loops to $Q$.  To this end, let us recall the context.  We start with two marked tilings $(T_{k}, T_{K})$ where $T_{k}$ is the standard unit square tiling and $T_{K}$ is some non-Euclidean tiling which is transverse to $T_{k}$ and its invariant vector field is not parallel to either the $x$ or $y$-axis.  When we play symplectic tiling billiards on this system, we observe that on $T_{K}$, the trajectories will only ever go vertical or horizontal due to the edge directions of $\fD_{T_{k}}$.  For a $P$-thin tile where $P$ is the unit square as in Lemma \ref{lem:sqthinexists}, this imposes strong restrictions on possible trajectories as seen in the following elementary lemma.

\begin{lemma}\label{lem:intersectedge}
Let $P$ be an unoriented quadrilateral and $T_{K}$ a complete affine tiling with marked tile $Q$ which is positively $P$-thin.  A ray pointed in an $\omega_{T_{K}}$-positive direction $[d_{i}]$ of $P$, as in Definition \ref{def:pthin}, and starting on the interior of a non-negative edge $w_{i}^{+}$ of $Q$ must intersect a non-positive edge $w_{j}^{-}$ of $Q$.
\end{lemma}

\begin{proof}
Let $q \in \del Q$ and $[d_{i}]$ an $\omega_{T_{K}}$-positive direction representing an edge of $P$ so that $\det(w_{i}^{+},d_{i}) > 0$ where $w_{i}^{+}$ is a non-negative edge of $\del Q$ containing $q$.  Because $Q$ is positively $P$-thin, $\det(w_{j}^{+},d_{i}) > 0$ where $w_{j}$ is the other non-negative edge of $\del Q$.  If the ray $q + td_{i}$ for $t \geq 0$ intersects the other non-negative edge $w_{j}^{+}$ at the point $q'$, then the ray $q' -t d_{i}$ for $t \geq 0$ must be in the interior of $Q$.  This means that $\det(w_{j}^{+},-d_{i}) > 0$, contradicting our hypothesis that $Q$ is positively $P$-thin.  Thus the ray must intersect a non-positive edge of $Q$.  
\end{proof}

One can see this property in Figure \ref{fig:pthin} with $[d_{1}] = [e_{1}]$ and $[d_{2}] = [e_{2}]$ for the unit square $P$ and the marked quadrilateral $Q$.  We explore some elementary consequences of Lemma \ref{lem:intersectedge} in our context where $T_{K}$ is our non-Euclidean tiling and $P$ is the unit square with sides parallel to the $x$ and $y$-axes.  Using Lemma \ref{lem:sqthinexists}, we can construct tiles $Q$ in $T_{K}$ that remain positively $P$-thin as we move through positive trajectories. \\
\\
Specifically, let $Q$ be a positively $P$-thin tile in $T_{K}$.  If we start on a non-negative edge, say $w_{i}^{+}$, and move in a positive direction, from either $[d_{1}] = [\pm e_{1}]$ or $[d_{2}] = [\pm e_{2}]$ depending on $\omega_{T_{K}}$, then Lemma \ref{lem:intersectedge} shows we must intersect a non-positive edge of $Q$ say $w_{j}^{-}$.  Assuming this point is not a vertex, we may continue the system.  The new tile is $\gamma_{1}Q$ where $\gamma_{1} = \pm e_{1}, \pm e_{2}$ depending on which edge of $Q$ was hit, and, the new marked edge of $\gamma_{1}Q$ is now non-negative to ensure the interior of the new tile $\gamma_{1}Q$ lies to the left of our new marked edge.  Because $Q$ was positively $P$-thin, this means the new possible directions are still the positive directions $[d_{1}]$ or $[d_{2}]$.  Thus the new configuration on the tile $\gamma_{1}Q$ is on a non-negative edge, and our possible directions are again $[d_{1}]$ or $[d_{2}]$.  If $\gamma_{1}Q$ is still positively $P$-thin, then we may apply the same arguments to yield a configuration on the tile $\gamma_{2}\gamma_{1}Q$ on a non-negative edge, again with the same possible positive directions $[d_{1}]$ and $[d_{2}]$.  Consequently the trajectory must lie in a positive cone determined by $d_{1},d_{2}$.  This argument was informed by the behavior noticed through experimentation which is discussed towards the end of this section, however, illustrations conveying this argument can be seen for example in Figure \ref{fig:symtil1}.  
\\
\\
If we are able to show that as we move through the tiles, we may control the shapes of the tiles to ensure we consistently move through positively $P$-thin tiles, then our orbits will remain in a positive cone determined by $[d_{1}]$ and $[d_{2}]$.  This will suffice to show our orbits diverge as any divergent sequence of points $p_{n}$ in a cone determined by two positive directions is divergent in the sense that $\lim_{n \to \infty} \omega_{T_{K}}(p_{n}-p_{0}) = \infty$ because both $[d_{1}]$ and $[d_{2}]$ are positive.  Consequently we aim to show that as we move through the tiles using only the positive directions $d_{1}$ and $d_{2}$, our tiles remain $P$-thin.  This is the content of Lemma \ref{lem:positivemoves} below.  We emphasize that the proof this lemma hinges on the rationality hypothesis of our tiling, and consequently, so too does our proof on the existence of divergent orbits.  

\begin{lemma}\label{lem:positivemoves}
Let $T_{K}$ be a marked non-Euclidean complete rational affine tiling of $\R^{2}$ and let $X_{T_{K}}, \omega_{T_{K}}$ be its canonical parallel invariant vector field and one-form respectively.  Denote by $K = (K_{1},K_{2}) \in \R^{2}$ and define $N := |a| + |b|$ the minimum over all such possible values of $a,b \in \Z$ so that $aK_{1} + bK_{2} = 0$.  Let $q \in \del Q$ be an interior edge point.  Fix a positive direction $d$ and consider the ray $r(t) := q + td$ for $t \geq 0$ and the sequence of adjacent marked tiles $Q_{0} := Q, Q_{1}, \hdots, Q_{N}$ that this ray intersects where at least one point of $Q_{i}$ lies in the half-plane $H_{r}$ defined by the ray $r(t)$.  If we define $\gamma_{i} \in \Z^{2}$ to be the loop such that $\gamma_{i}Q_{0} = Q_{i}$ for each $i = 1 \hdots N$, then the collection $\{\gamma_{1},\hdots, \gamma_{N}\}$ contains a loop so that $\omega_{T_{K}}(T_{T_{K}}(\gamma_{i})) \geq 0$.  If $K_{1}K_{2} = 0$, then we may take $N = 1$ so that the adjacent tile $Q_{1} = \gamma_{1} Q_{0}$ satisfies $\omega_{T_{K}}(T_{T_{K}}(\gamma_{1})) \geq 0$.  
\end{lemma}


\begin{proof}
To make notation less cumbersome, let us denote our tiling by $T$.  For now, assume $ab \neq 0$, thus $k_{1}k_{2} \neq 0$.  Draw the ray $r(t) = q + td$ for $t \geq 0$ and $d$ some $\omega_{T}$-positive direction.  Because $k_{1}k_{2} \neq 0$, the edges are not parallel to the flow determined by $X_{T}$.  Let $Q_{i} = \gamma_{i}Q_{0}$ denote the various tiles that $r(t)$ intersects, and consider the sequence of anchors as determined by each $Q_{i}$.  We claim that there is some $\gamma_{i}$ for which $\omega_{T}(T_{T}(\gamma_{i})) \geq 0$.  Of the edge of $Q$ containing $q$, pick the vertex $p_{0}$ so that $q$ lies in the left-plane plane as determined by the flow line through $p$.  Let $p_{i}$ denote the images of this vertex given by $p_{i} = L_{T}(\gamma_{i})p + T_{T}(\gamma_{i})$ where $L_{T}$ is the linear part of $\rho_{T}$.  That is to say, $p_{i}$ is the corresponding vertex associated to $p$ in each $Q_{i} = \gamma_{i} Q_{0}$.  \\
\\
Because we have that $\omega_{T}(p_{i}-p_{0}) = \omega_{T}(T_{T}(\gamma_{i}))$ as $L_{T}(\gamma_{i})-\id$ is parallel to the vector flow as seen in Equation \ref{eq:hol_rep}, it suffices to show that one of the $\gamma_{i}$ takes the difference, $p_{i}-p_{0}$ to a neutral direction.
\\
\\
To this end consider the oriented leaf through $p$ of the foliation.  Because $r(t)$ is in a positive direction, for $t \geq 0$, $r(t)$ must lie on the left side of the oriented leaf by construction as $r(0) = q$ did.  Thus every intersection of $r(t)$ with $Q_{i}$ must be a point lying to the left of the oriented leaf, or at worst, lying on it.  Thus no $Q_{i}$ may lie entirely to the right of the this oriented leaf through $p$.  Returning to the sequence of $p_{i}-p_{0}$, if one of the $\gamma_{i}$'s takes the entire tile $Q_{i} = \gamma_{i}Q_{0}$ to the left of the leaf, then we are done.  Because no $Q_{i}$ lies entirely to the right of the leaf, we may thus assume instead that the sequence of tiles $Q_{i}$ intersects this leaf for each $i = 1, \hdots, N$.  By the rationality hypothesis of our torus, after $N = (|a|+|b|)$-iterations through these tiles, the resulting tile $Q_{N}$ will be a translate of $Q_{0}$, and thus, the difference in anchors $p_{N}-p$ is parallel to the leaf as claimed.  Figure \ref{fig:positive_lemma} illustrates this argument.  
\\
\\
In the case where $k_{1}k_{2} = 0$, we proceed as before where we let $p$ be the vertex as defined before.  When we draw the oriented leaf through $p$, by Lemma \ref{lem:k1k2zero}, the leaf is entirely contained in $\del T$.  Every adjacent tile $Q_{1}$ with at least one point to the left the of half-plane determined by this leaf has $p_{1} = L_{T}(\gamma_{1})p + T_{T}(\gamma_{1})$ that lies to the left of the leaf, or at worst, along it.  See Figure \ref{fig:induced_orient} as an example.  Thus we may take $N = 1$ as claimed so that $\omega_{T}(T_{T}(\gamma_{1})) \geq 0$.  
\end{proof}

\begin{figure}
\includegraphics[scale=0.4]{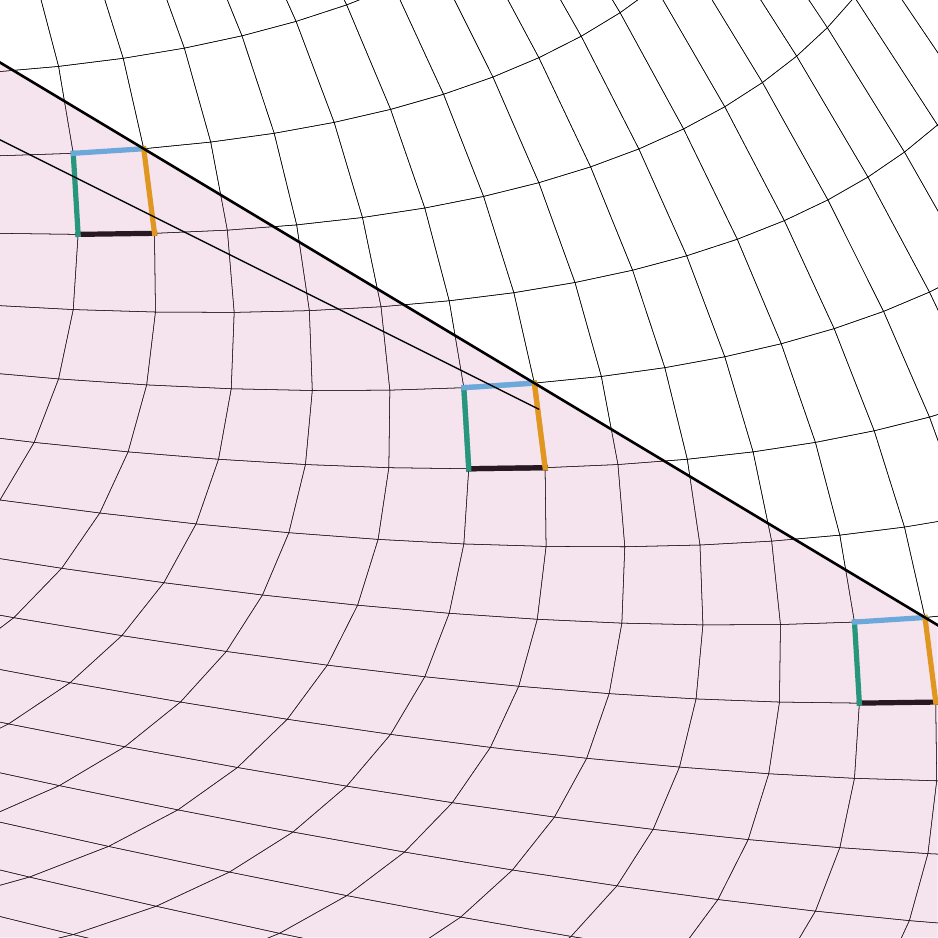}
\caption{An illustration of the argument in Lemma \ref{lem:positivemoves} where $K = (-3/10,-5/10)$.  The marked tile $Q$ is illustrated in the middle, and the half-plane as defined by the oriented leaf through $p$ is illustrated in red.  Both the leaf itself and the positive half ray $r(t)$ are drawn in thin black lines.  After at most $|a| + |b| = 8$-moves in positive directions, $\gamma_{i}Q$ will be a translate of $Q$ so $\omega_{T_{K}}(T_{T_{K}}(\gamma_{i})) \geq 0$.  
}\label{fig:positive_lemma}
\end{figure}

Optimistically, one could hope that moving in a positive direction to go from tile $Q_{0}$ to tile $\gamma_{1}Q_{0}$ always necessitates that $\gamma_{1}$ is a positive loop.  Figure \ref{fig:positive_lemma} illustrates that this is not necessarily the case.  The immediate tile above the center marked one corresponds to non-positive $\gamma_{1}$.  However, the figure illustrates how to amend this statement as in Lemma \ref{lem:positivemoves}.  While a single positive direction does not necessarily determine a loop $\gamma$ so that $\omega_{T}(T_{T}(\gamma))\geq 0$, if one only moves in positive directions on a \emph{rational torus}, then there will eventually be such a loop so that $\omega_{T}(T_{T}(\gamma)) \geq 0$.  Such a loop is guaranteed by the translational symmetry, and thus the shape of the tile is preserved after $N := |a| + |b|$ moves as in Lemma \ref{lem:positivemoves}.  This is particularly useful for us because Lemma \ref{lem:positivemoves} gives us an effective way to control the shape of the tile as we move through positive directions.  Combining the results of our previous arguments, we are finally able to establish our theorem about divergent orbits.  

\begin{theorem}\label{thm:eucdiv}
Let $(T_{k}, T_{K})$ be a pair of transverse Euclidean and non-Euclidean marked complete \emph{rational} affine tilings where the invariant vector field of $T_{K}$ is not parallel to the edge directions $\fD_{T_{k}}$ of $T_{k}$.  Then there exist marked tiles $P$ and $Q$ and oriented edges of these tiles so that for any configuration along these edges which is defined for all time diverges positively along the invariant one-form $\omega_{T_{K}}$ of $T_{K}$.  More precisely, if $\{(p_{n}, q_{n})\}_{n=0}^{\infty} \subset \del T_{k}\times \del T_{K}$ denotes the sequence of points of the symplectic tiling billiards trajectories, then we have the following asymptotic behavior.  
\begin{equation}\label{eq:eucasym}
\lim_{n \to\infty} [p_{n+1}-p_{n}] = [\pm X_{T_{K}}] \in \fD \text{ and } \lim_{n\to\infty} \int_{q_{n}-q_{0}} \omega_{T_{K}} = \infty
\end{equation}
\end{theorem}

\begin{proof}
Preliminarily observe that such pairs of tilings exist by Theorem \ref{thm:goodtiles}.  Because the dynamics are invariant under affine symmetries as discussed in the beginning of Section \ref{ssec:div}, we may assume without loss of generality $T_{k}$ is the standard marked Euclidean tiling and $T_{K}$ is a marked non-Euclidean tiling whose canonical vector field $X_{T_{K}}$ is not parallel to either $e_{1},e_{2}$, and the tilings are transverse.  Let $P$ denote the marked unit square of $T_{k}$.  \\
\\
Let $K = (K_{1},K_{2}) \in \R^{2}$ and define $N := |a| + |b|$ the minimum over all such possible values of $a,b \in \Z$ so that $aK_{1} + bK_{2} = 0$.  By Lemma \ref{lem:sqthinexists}, we may choose a $\omega_{T_{K}}$-positively $P$-thin marked tile $Q$ of $T_{K}$ so that for all $\gamma = (n,m)$ with $|n| + |m| < 2N$, $\gamma Q$ is also $P$-thin.  Pick any edge of $P$ and interior point of this edge.  Pick either \emph{non-negative} edge $w_{1}^{+}, w_{2}^{+}$ of $Q$ and an interior point of this edge.  Call these points $(p_{0},q_{0}) \in \del T_{k} \times \del T_{K}$ respectively.  Let us assume that these points are dynamically defined for all forward iterations so $\{(p_{n},q_{n})\}_{n=0}^{n=\infty}$ is well-defined.  That is to say, none of the points $p_{n} \in \del T_{k}$ or $q_{n} \in \del T_{K}$ are vertex points of their respective tilings which exist by Lemma \ref{lem:gen_good}.  We aim to prove the claims in Equation \ref{eq:eucasym}.  \\
\\
Because $Q$ is positively $P$-thin, there are choices amongst the edge directions of $P$, $[\pm e_{1}], [\pm e_{2}]$ which are $\omega_{T_{K}}$-positive directions, say $[d_{1}]$ and $[d_{2}]$ respectively which satisfy $\det(w_{i}^{+},d_{1})$, $\det(w_{i}^{+},d_{2}) > 0$ for both non-negative edges $w_{1}^{+},w_{2}^{+}$.  Let us apply the symplectic tiling billiards rule to obtain $(p_{1},q_{1}) \in \del T_{k} \times \del T_{K}$ focusing on $T_{k}$ first.  The initial point $p_{0}$ is on an edge of $P$ whose direction is one of $[\pm e_{1}], [\pm e_{2}]$.  The initial point $q_{0}$ is on a non-negative edge $Q$ whose direction is one of $[w_{i}^{+}]$.  By the positively $P$-thin property of $Q$, we have both $\det(w_{i}^{+},d_{1})$, $\det(w_{i}^{+},d_{2}) > 0$.  So if $p_{0}$ is on an edge directed in either $[d_{1}]$ or $[d_{2}]$, then $p_{0}$ will follow a trajectory directed towards $[-w_{i}^{+}]$.  Similarly if $p_{0}$ is on an edge directed in either $[-d_{1}]$ or $[-d_{2}]$, then $p_{0}$ will follow a trajectory directed towards $[w_{i}^{+}]$.  This will intersect a new point $p_{1}$ of $\del P$ and define a new marked tile $\alpha_{1}P$ containing $p_{1}$ in its boundary.  By construction, the new marked edge will be directed in either $[d_{1}]$ or $[d_{2}]$, or, either $[-d_{1}]$ or $[-d_{2}]$.  
\\
\\
On the other tiling $T_{K}$, we have $q_{0}$ lies on the non-negative edge $w_{i}^{+}$ of the marked tile $Q$.  By the positively $P$-thin property of $Q$, the trajectory must be directed in either positive direction $[d_{1}]$ or $[d_{2}]$.  By Lemma \ref{lem:intersectedge}, this trajectory must intersect a non-positive edge $w_{j}^{-}$ of $Q$ at the point $q_{1}$.  This determines a new marked tile $\beta_{1}Q$ containing $q_{1}$ in its boundary.  By the positive $P$-thinness of $Q$, since $\det(w_{j}^{-},d_{1})$, $\det(w_{j}^{-},d_{2}) < 0$, this means the directions $[d_{1}], [d_{2}]$ are pointed in the interior of $\beta_{1}Q$ based at $q_{1}$.  Moreover $q_{1}$ is contained in a non-negative edge $w_{i}^{+}$ of $\beta_{1}Q$, as $q_{1}$ was contained in a non-positive edge of $Q$.  \\
\\
This gives us our new pair $(p_{1},q_{1})$ in the boundary of the marked tiles $P_{1} = \alpha_{1}P$ and $Q_{1} := \beta_{1}Q$.  By the above arguments $q_{1}$ is again contained in a non-negative edge $w_{i}^{+}$ of the tile $Q_{1}$, and, by hypothesis $Q_{1}$ is still $P$-thin.  As $P_{1} = P$ up to translation, the same arguments applied in the case for $(p_{0},q_{0})$ follow and we may iterate this process while only moving through the positive directions $[d_{1}]$ and $[d_{2}]$ on the tiling $T_{K}$, and on $T_{k}$ only moving in directions pointed towards either $[-w_{i}^{+}]$ or $[w_{i}^{+}]$ of $\beta_{i}Q_{i}$ where the sign is determined by the initial edge of $P$.  By the hypothesis on our original tile $Q$, we remain in positively $P$-thin tiles for at least $N$-moves moving only in the positive directions $[d_{1}]$ and $[d_{2}]$ as $N < 2N$.  By the hypothesis on our original tile $Q$, we remain in positively $P$-thin tiles for at least $N$-moves in the positive directions $[d_{1}]$ and $[d_{2}]$ as $N < 2N$.  By the rationality hypothesis, Lemma \ref{lem:positivemoves} and Lemma \ref{lem:getthin} guarantee after at most $N$-moves there exists a loop $\beta_{i} \in \Z^{2}$ for which the tile $\beta_{i}Q$ must have its edge directions $[w_{1}^{+}], [w_{2}^{+}]$ and $[w_{1}^{-}], [w_{2}^{-}]$ move clockwise closer to $[X_{T_{K}}]$ and $[-X_{T_{K}}]$ in $\fD$, or, remain the same.  In either case, this new marked tile $\beta_{i}Q$ is still positively $P$-thin, and because its edge directions either move clockwise closer to $[X_{T_{K}}]$ and $[-X_{T_{K}}]$, or, remain the same, so $\beta(\beta_{i}Q)$ is also $P$-thin for all $\beta = (n,m)$ with $|n|+|m| < 2N$.  We may therefore apply the same arguments and continue in this fashion indefinitely. 
\\
\\
Thus, our sequence of points $q_{n}$ may only ever move in the positive directions $[d_{1}]$ or $[d_{2}]$ for every iteration.  Consequently, the sequence $q_{n}$ is contained in the positive cone defined by $q_{0} + (t_{1}d_{1} + t_{2}d_{2})$ for $t_{1},t_{2} \geq 0$.  Because the anchors of the tiles $\beta_{n}Q$ remain a bounded distance outside of this cone, it follows the sequence $\omega_{T_{K}}(T_{T_{K}}(\beta_{n}))$, which equals the difference of anchors between $\beta_{n}Q$ and $Q$ paired against $\omega_{T_{K}}$, tends to infinity as both rays $t_{1}d_{1}, t_{2}d_{2}$ do.  Thus $\lim_{n\to\infty} \int_{q_{n}-q_{0}} \omega_{T_{K}} = \lim_{n\to\infty}  \omega_{T_{K}}(q_{n}-q_{0}) = \infty$ as claimed.  The fact that $\lim_{n \to\infty} [p_{n+1}-p_{n}] = [\pm X_{T_{K}}] $ follows readily as the edge directions of the tiles $\beta_{i}Q$ tend to $[\pm X_{T_{K}}]$.  Since the initial points $(p_{0},q_{0})$ were generic along the edges, the asymptotic equalities of Equation \ref{eq:eucasym} are true for any points we begin with satisfying the hypotheses in the initial part of the proof.  Thus this divergence is stable.  
\end{proof}

The proof of Theorem \ref{thm:eucdiv} is better understood with accompanying images in mind.  Below are images of the game carried out on a pair of edges that satisfy the hypotheses of Theorem \ref{thm:eucdiv}.  These simulations were programmed in Java using exact rational arithmetic.  The tilings, $T_{K}$, are parametrized by the family of holonomy representations given in Equation \ref{eq:hol_rep}.  In the interest of computational speed, they only illustrate the tiles that are traversed through the trajectories.  These programs can be found on the first author's webpage.  We emphasize that while the images below only illustrate the family $T_{K}$ as parametrized by the holonomies of Equation \ref{eq:hol_rep}, Theorem \ref{thm:eucdiv} applies more broadly to families which are transverse in the sense of Theorem \ref{thm:goodtiles}.  
\\
\begin{figure}
	\begin{minipage}{0.45\textwidth}
		\includegraphics[width=\textwidth]{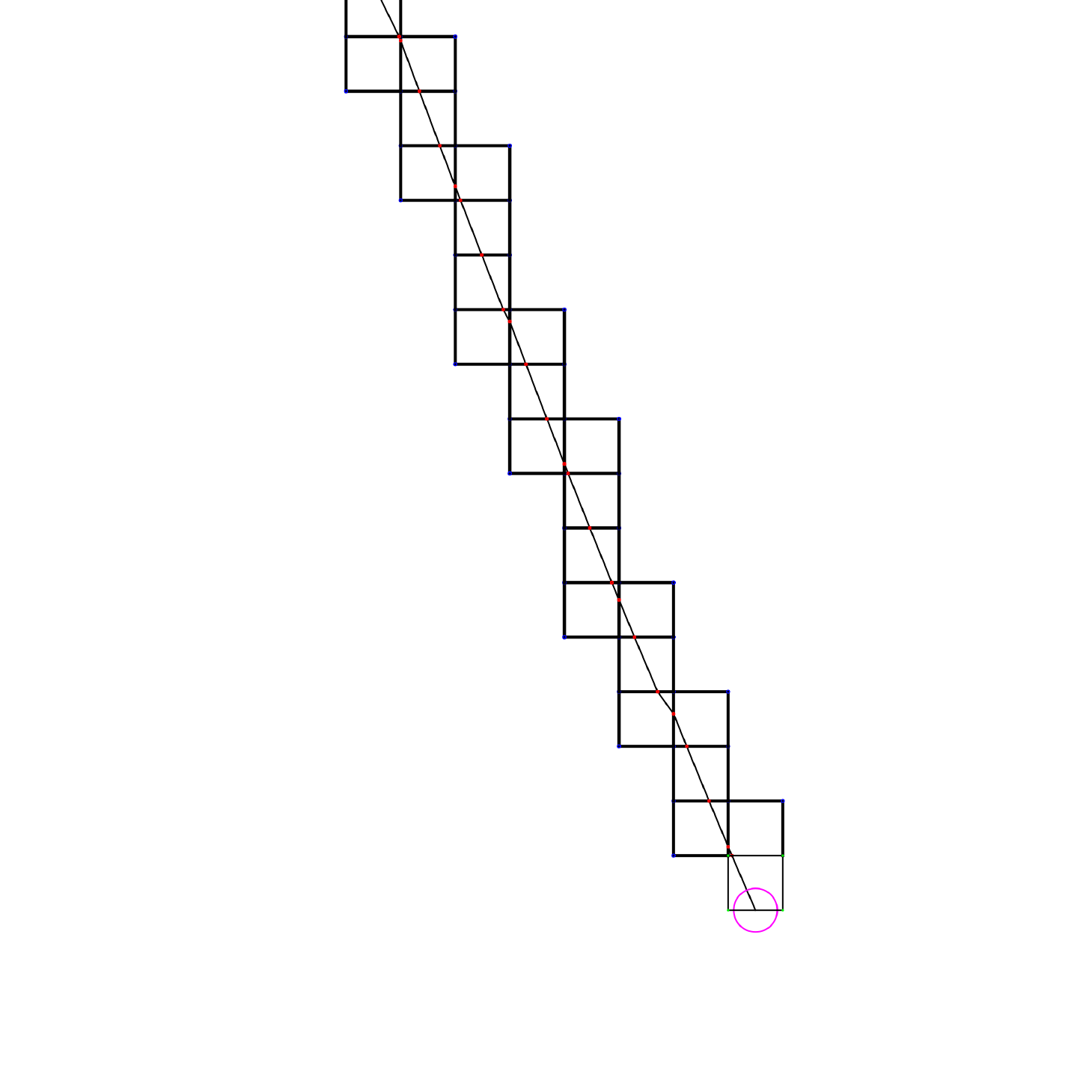}
	\end{minipage}
	\hfill
	\begin{minipage}{0.45\textwidth}
		\includegraphics[width=\textwidth]{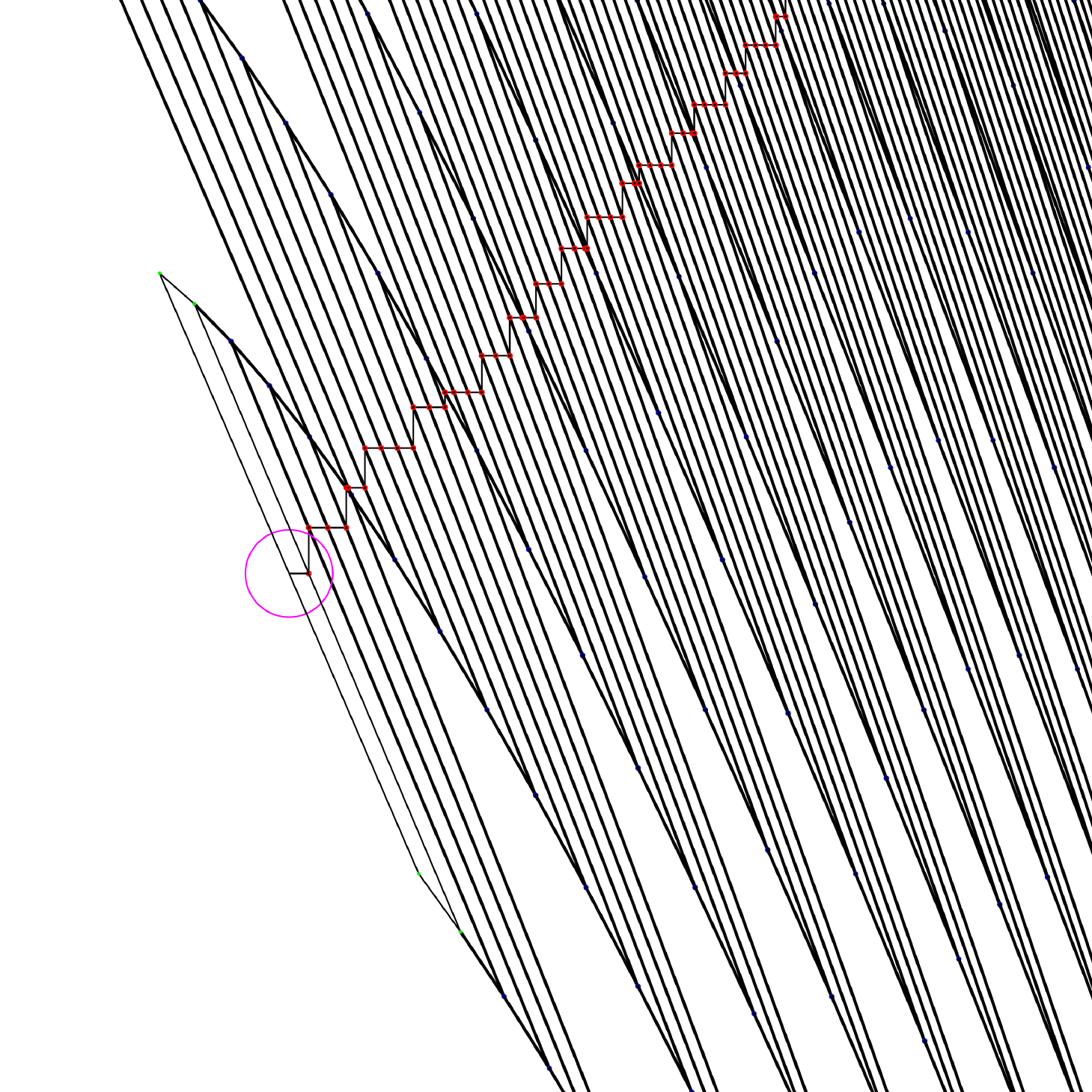}
	\end{minipage}
	\caption{Iterations of symplectic tiling billiards played on a Euclidean tiling $T_{k}$ on the left and a marked complete affine tiling $T_{K}$ on the right.  The right tiling has period parameters $K = (1,1/4)$.  The initial configurations are labeled in pink circles.  Vertices of the tiling are shown in very small blue, and the iterates $\{(p_{n},q_{n})\} \in \del T_{k}\times \del T_{K}$ are drawn in red.  The initial tile of $T_K$ is positively $P$-thin where $P$ is the unit square and the invariant vector field determines the direction $[X_{K}] = [(1/4,-1)]$.  Note the trajectories lie in the positive cone of directions as determined by $[d_{1}] = [e_{1}]$ and $[d_{2}] = [e_{2}]$ by the remarks follows Lemma \ref{lem:intersectedge}.  
	}\label{fig:symtil1} 
\end{figure}
  
In Figure \ref{fig:symtil1}, the initial configuration is $(p_{0},q_{0}) \in \del T_{k}\times \del T_{K}$ where $K =  (1,1/4)$ and is depicted in pink circles.  The point $q_{0}$ is chosen to lie on a positive edge $w_{i}^{+}$ of the initial tile $Q$ of $T_{K}$.  A choice of $[d_{1}], [d_{2}]$ as in the proof of Theorem \ref{thm:eucdiv} is given by $[d_{1}] = [e_{1}]$ and $[d_{2}] = [e_{2}]$ which are both $\omega_{K}$-positive as seen by Equation \ref{eq:parone}.  Because both $\det(w_{i}^{+},d_{1})$, $\det(w_{i}^{+},d_{2}) > 0$, this means $\det(d_{1},w_{i}^{+}) < 0$.  Because $p_{0}$ lies on an oriented edge directed towards $[d_{1}] = [e_{1}]$, the first trajectory of $T_{k}$ must go directed towards $[-w_{i}^{+}]$.  \\
\\
On the right hand tiling $T_{K}$, we only ever move through the positive directions $[d_{1}] = [e_{1}]$ and $[d_{2}] = [e_{2}]$.  By Lemma \ref{lem:positivemoves}, this means after at most $(4 + 1)$-moves through adjacent tiles, the tile must thin and have its edge directions become closer to $[\pm X_{K}]$, or at worst, be a translation.  In fact, a translation is what happens in the first $5$-moves as seen in Figure \ref{fig:symtil1}.  Because the shape of the tiles on $T_{K}$ is controlled as we move through positive trajectories, we see that the iterates $q_{n}$ lie inside the cone determined by $q_{0} + (t_{1}d_{1} + t_{2}d_{2})$ where $t_{1},t_{2} \geq 0$.  This guarantees the divergence of $\lim_{n\to\infty}\omega_{K}(q_{n}-q_{0}) = \infty$.  Because the tiles become more and more thin and have their directions tend to $[\pm X_{K}]$ as seen by Lemma \ref{lem:getthin}, the trajectory of the $p_{n}$'s becomes more and more parallel to $[\pm X_{K}]$ depending on the choice of initial edge of $P$.  In this particular instance, the trajectories become more parallel to $[-X_{K}]$.  Had we started $P$ on the opposite edge, as seen in Figure \ref{fig:symtil2}, the initial trajectories on $T_{K}$ would be the exact same, but, the trajectories in $T_{k}$ would head towards $[X_{K}]$ instead.  
\\
\begin{figure}
	\begin{minipage}{0.45\textwidth}
		\includegraphics[width=\textwidth]{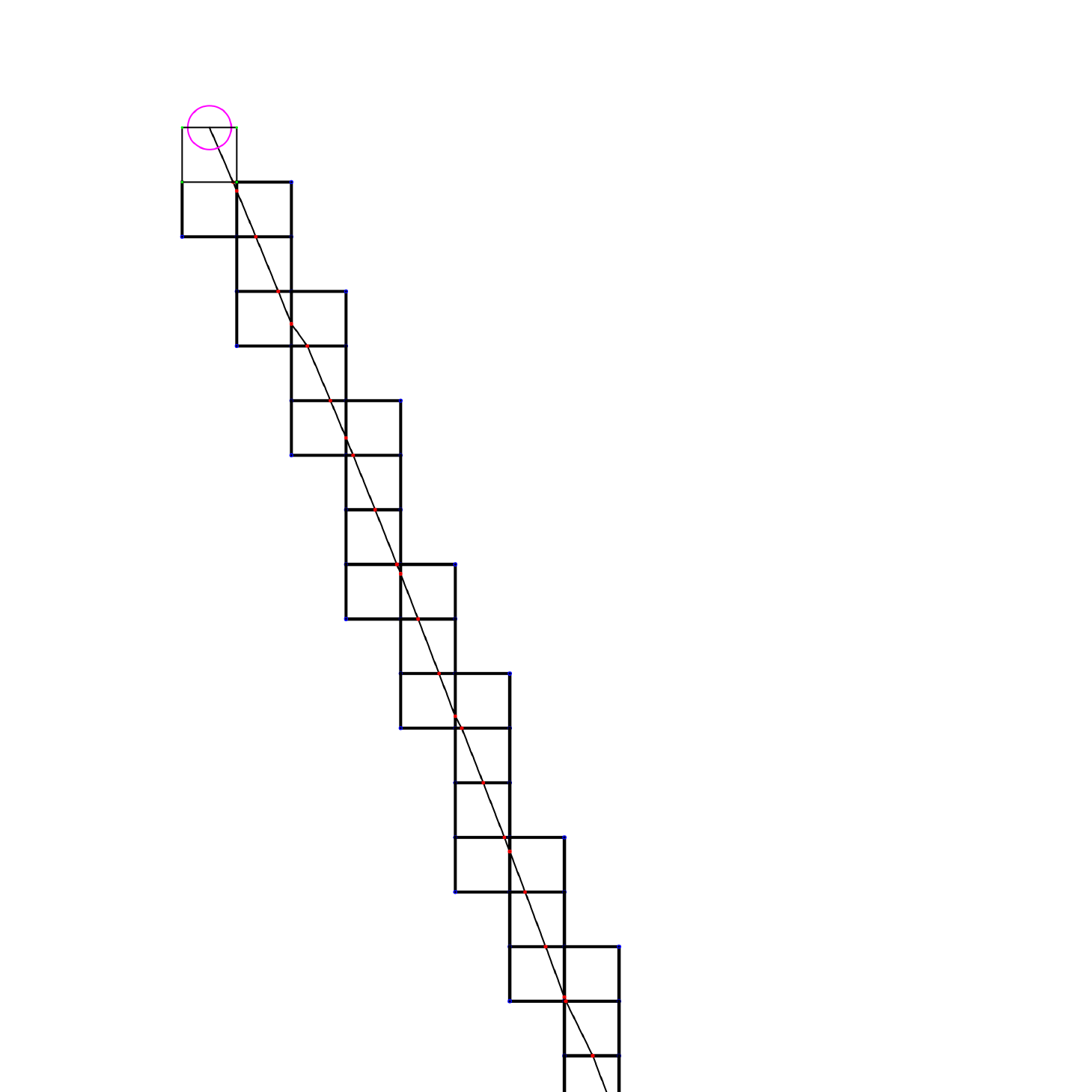}
	\end{minipage}
	\hfill
	\begin{minipage}{0.45\textwidth}
		\includegraphics[width=\textwidth]{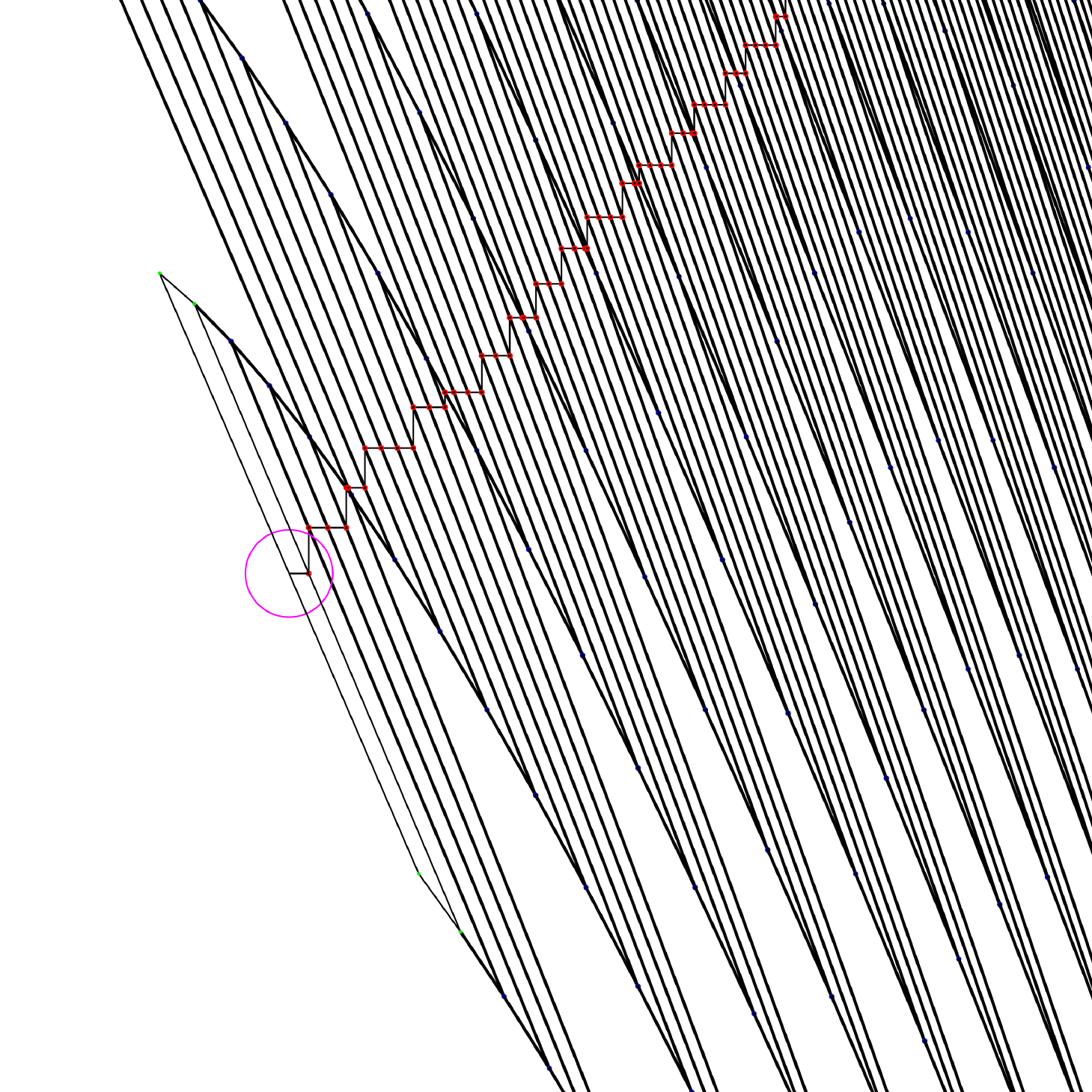}
	\end{minipage}
	\caption{Iterations of the trajectories as in Figure \ref{fig:symtil1}, but with the initial configuration of $P$ beginning on the opposite edge.  This ensures the trajectories of $T_{k}$ become more parallel to $[X_{K}] = (1/4,-1)$ instead of $[-X_{K}]$.}\label{fig:symtil2} 
\end{figure}

We remark that Theorem \ref{thm:eucdiv} shows the existence of many divergent orbits.  The proof relies on the existence of $P$-thin tiles in $T_{K}$ which are abundant in every tiling.  Interestingly enough, the authors have found this divergence to be generic.  We illustrate these findings in the following figures.  
\\
\begin{figure}
	\begin{minipage}{0.45\textwidth}
		\includegraphics[width=\textwidth]{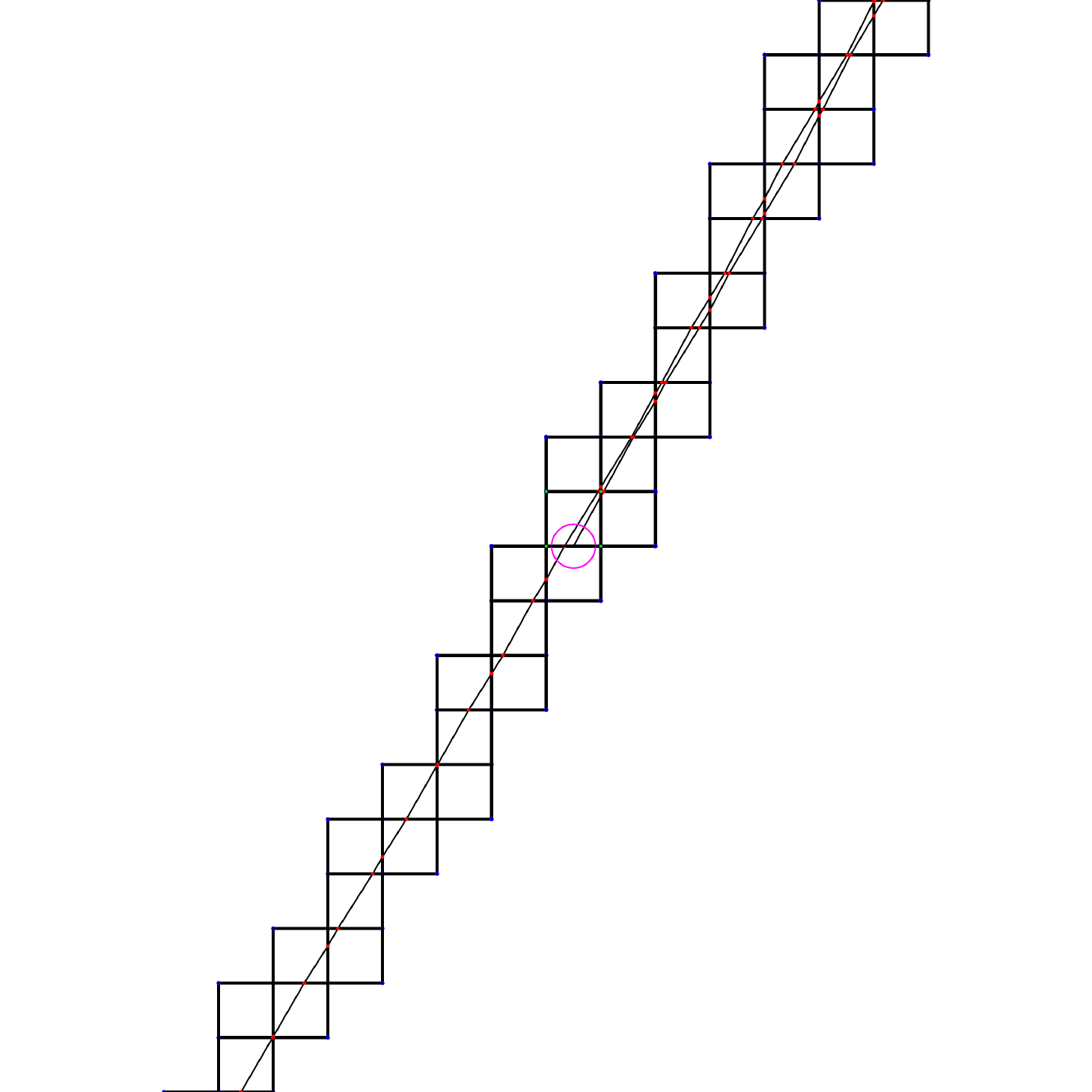}
	\end{minipage}
	\hfill
	\begin{minipage}{0.45\textwidth}
		\includegraphics[width=\textwidth]{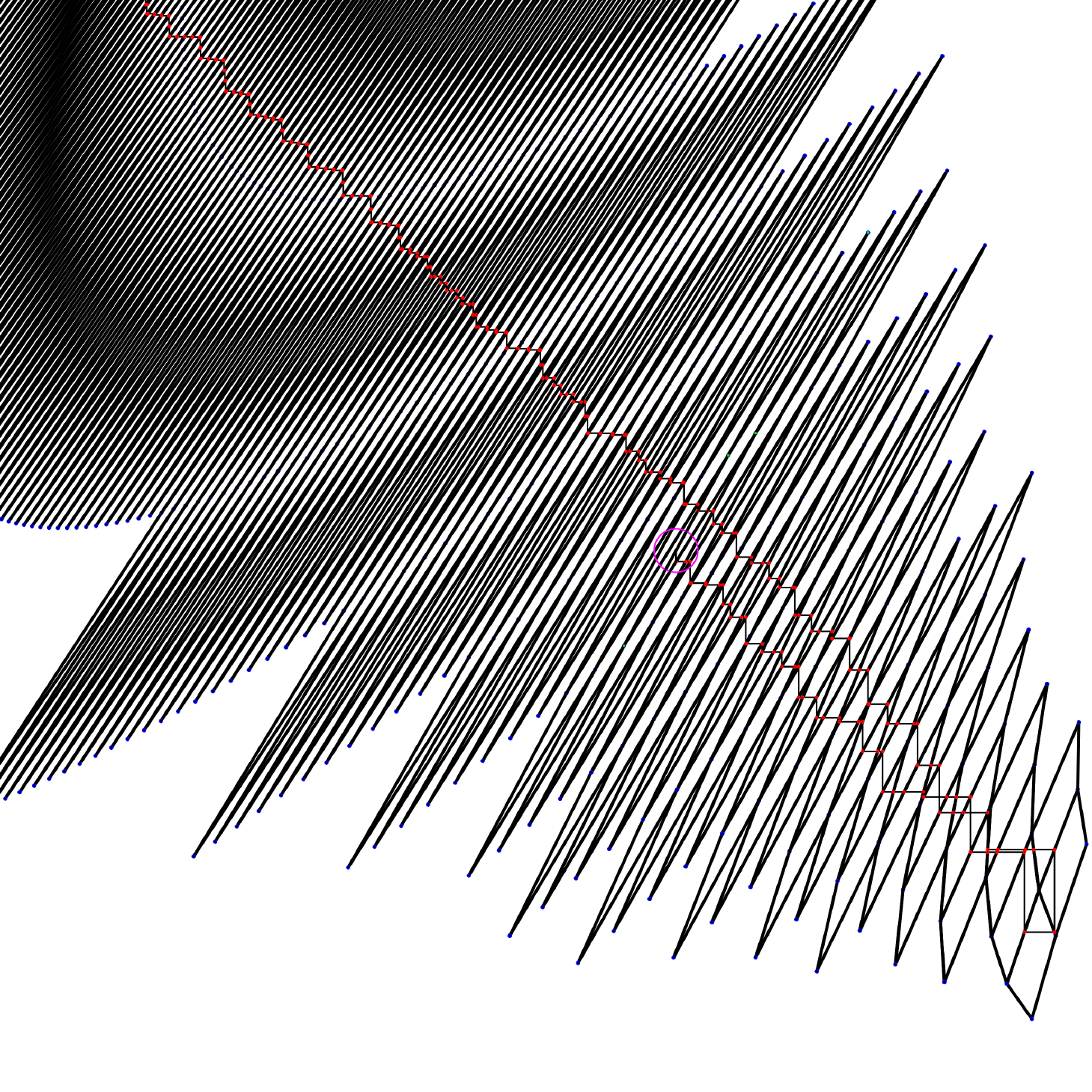}
	\end{minipage}
	\caption{Iterates of the game played on the Euclidean tiling and tiling $T_{K}$ where $K = (-6/10,1/2)$.  Here we started on a negative edge of a positively $P$-thin tile.  The initial configurations are drawn in pink.  One can see the trajectory of $T_{K}$ moves to less thin tiles, then turns around going towards more thin tiles, and ultimately ends up on a positive edge of a positively $P$-thin tile, then begins to diverge.  One can see the quantity $\omega_{T_{K}}(q_{n}-q_{0})$ grow large as $q_{n}-q_{0}$ lies in the positive half-plane determined by the one-form $\omega_{T_{K}}$ as in Equation \ref{eq:parone}.  A quick sign check is given by $K\cdot(q_{n}-q_{0}) > 0$.  
	}\label{fig:symtil3} 
\end{figure}

\begin{figure}
	\begin{minipage}{0.45\textwidth}
		\includegraphics[width=\textwidth]{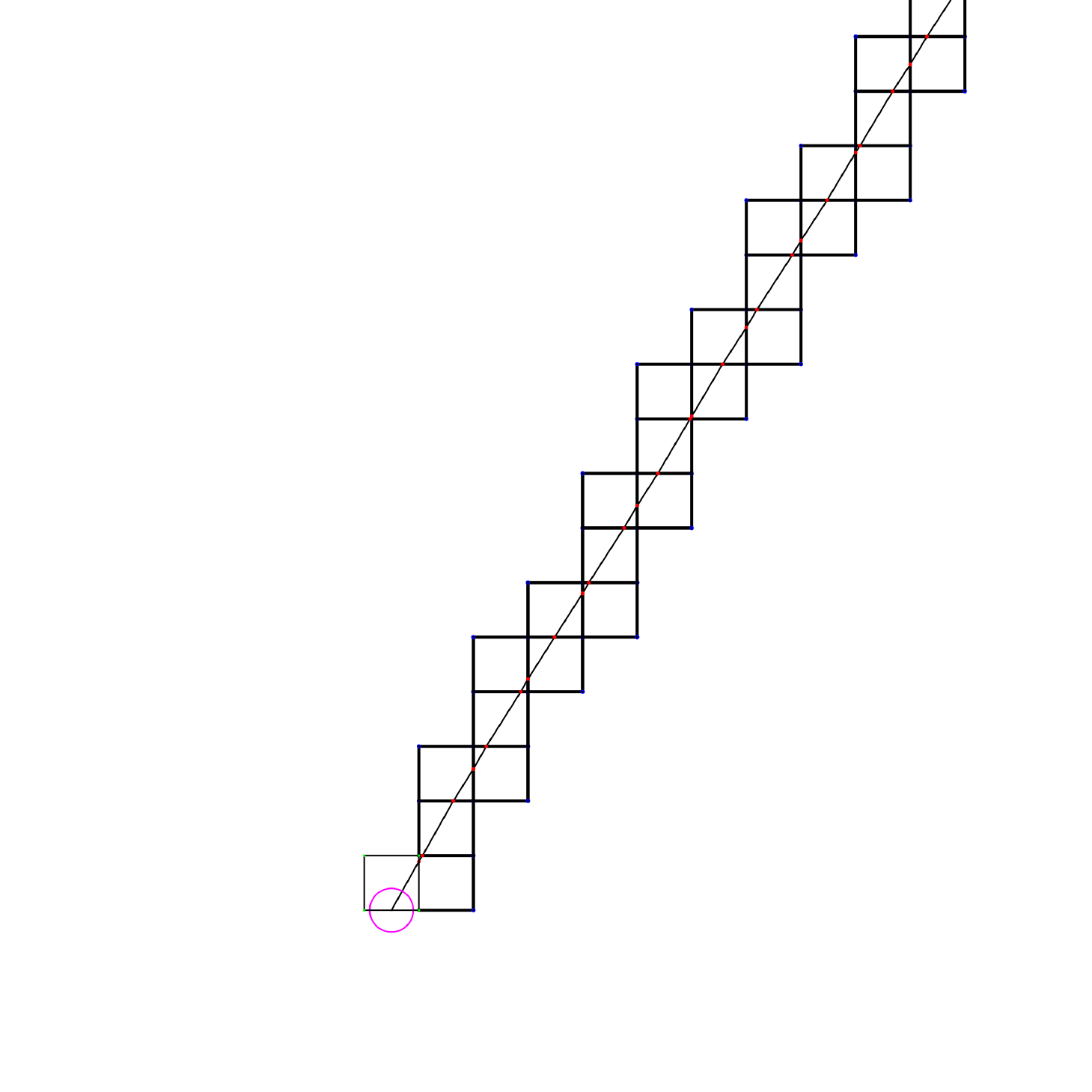}
	\end{minipage}
	\hfill
	\begin{minipage}{0.45\textwidth}
		\includegraphics[width=\textwidth]{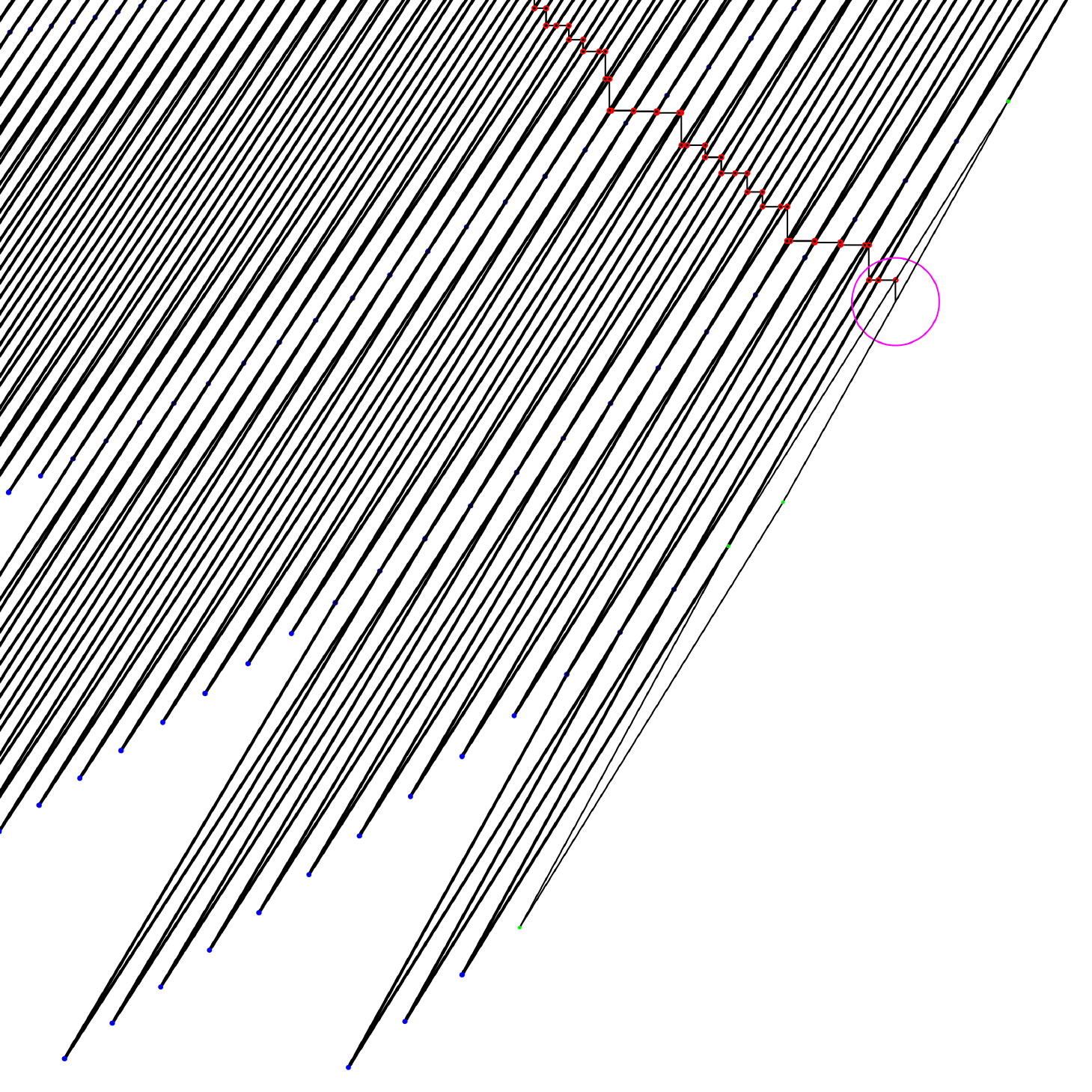}
	\end{minipage}
	\caption{Iterates of the same tiling as in Figure \ref{fig:symtil3} however the initial configuration begins on a positive edge of the $P$-thin tile this time.  Divergence is seen immediately as guaranteed by Theorem \ref{thm:eucdiv}.  
	}\label{fig:symtil4} 
\end{figure}

\begin{figure}
	\begin{minipage}{0.4\textwidth}
		\includegraphics[width=\textwidth]{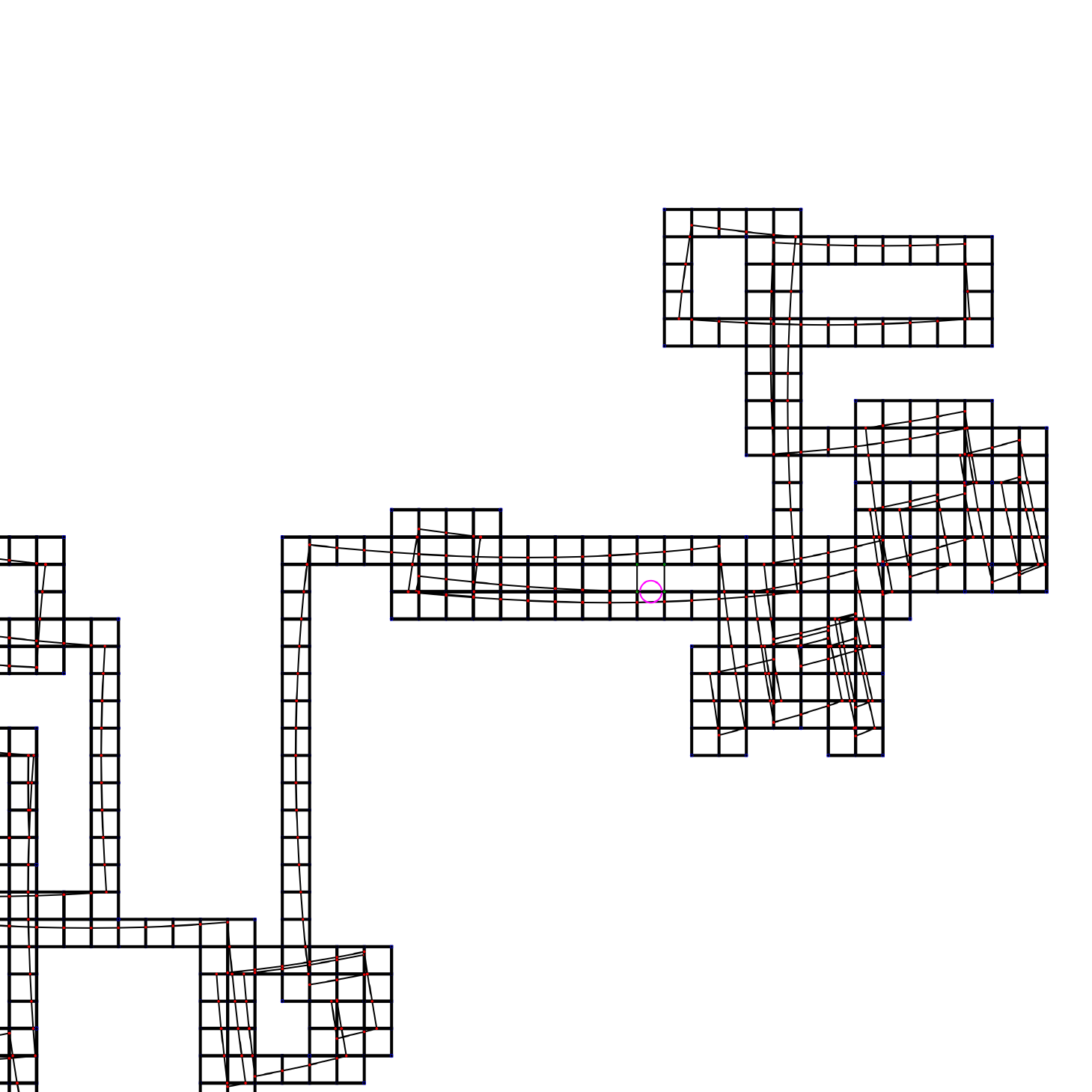}
	\end{minipage}
	\phantom{=}
	\begin{minipage}{0.4\textwidth}
		\includegraphics[width=\textwidth]{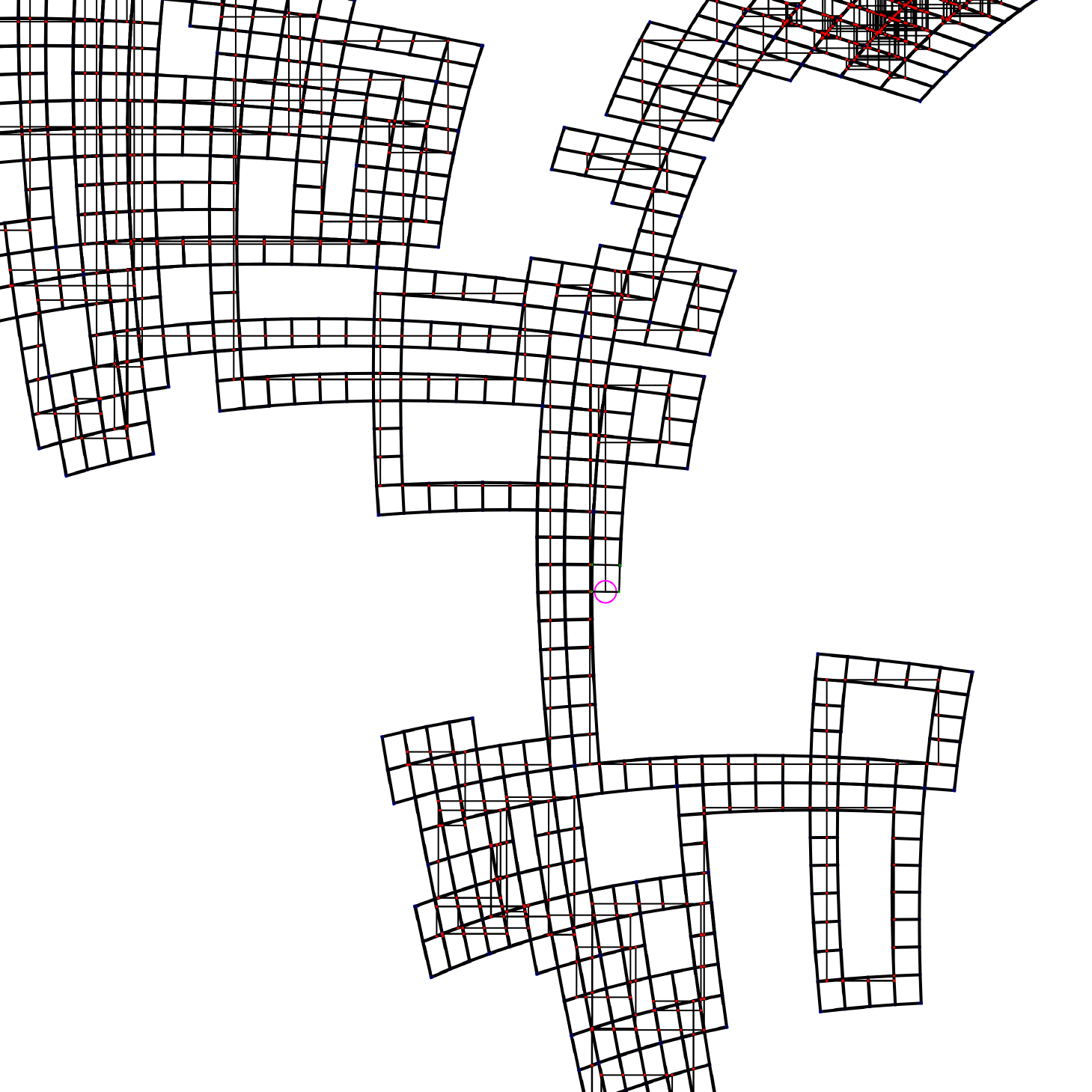}
	\end{minipage}
	\\
	\begin{minipage}{0.4\textwidth}
		\includegraphics[width=\textwidth]{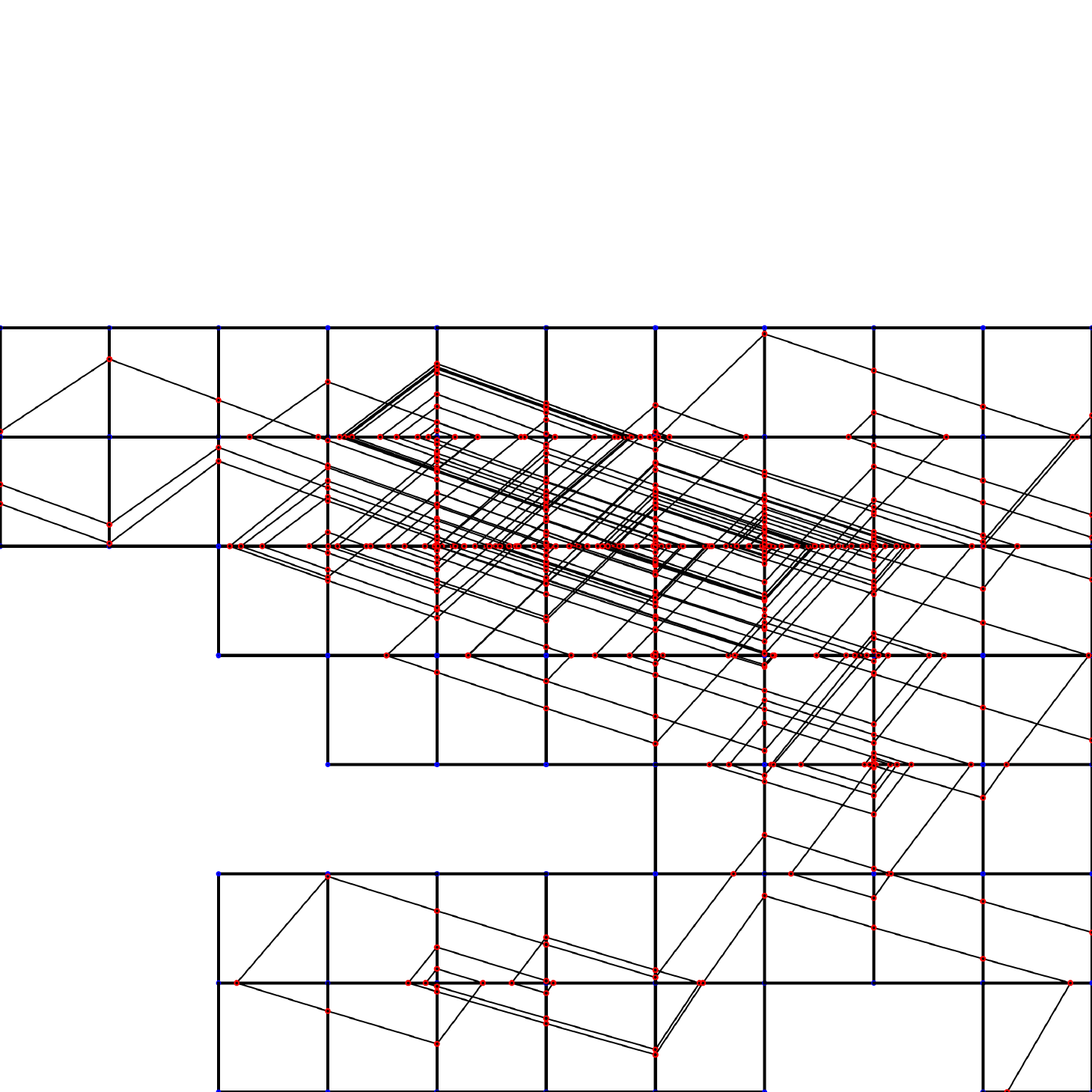}
	\end{minipage}
	\phantom{=}
	\begin{minipage}{0.4\textwidth}
		\includegraphics[width=\textwidth]{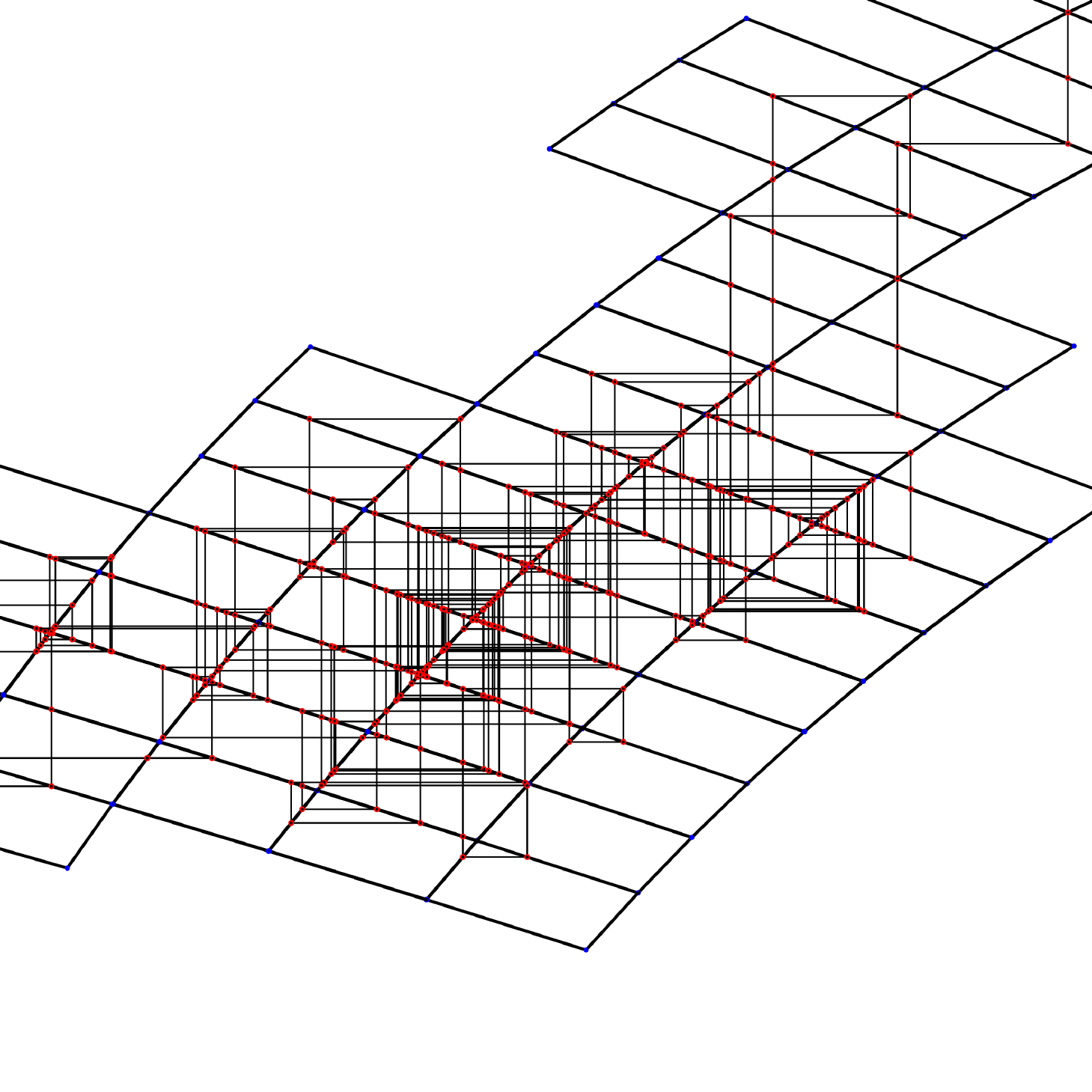}
	\end{minipage}
	\\
	\begin{minipage}{0.4\textwidth}
		\includegraphics[width=\textwidth]{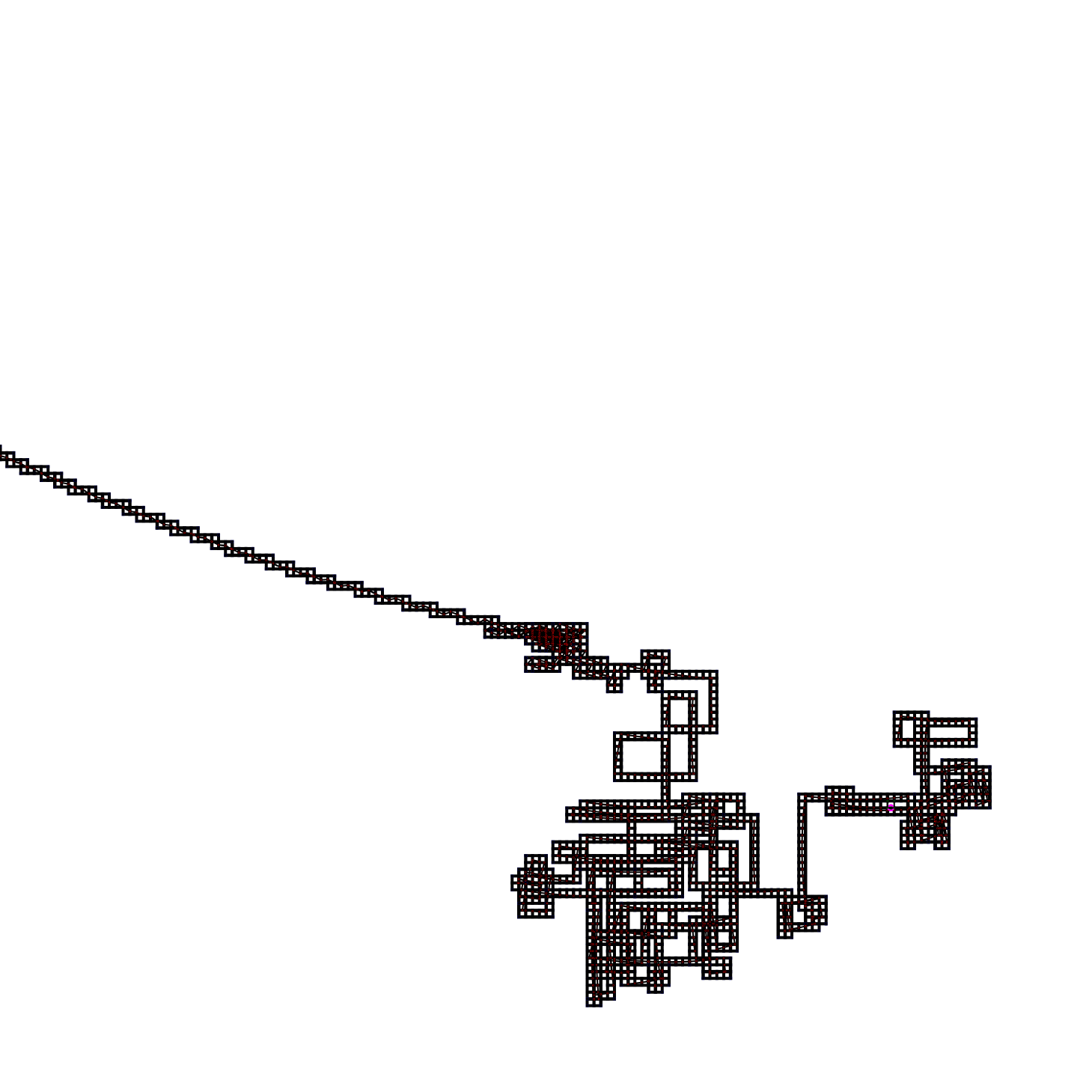}
	\end{minipage}
	\phantom{=}
	\begin{minipage}{0.4\textwidth}
		\includegraphics[width=\textwidth]{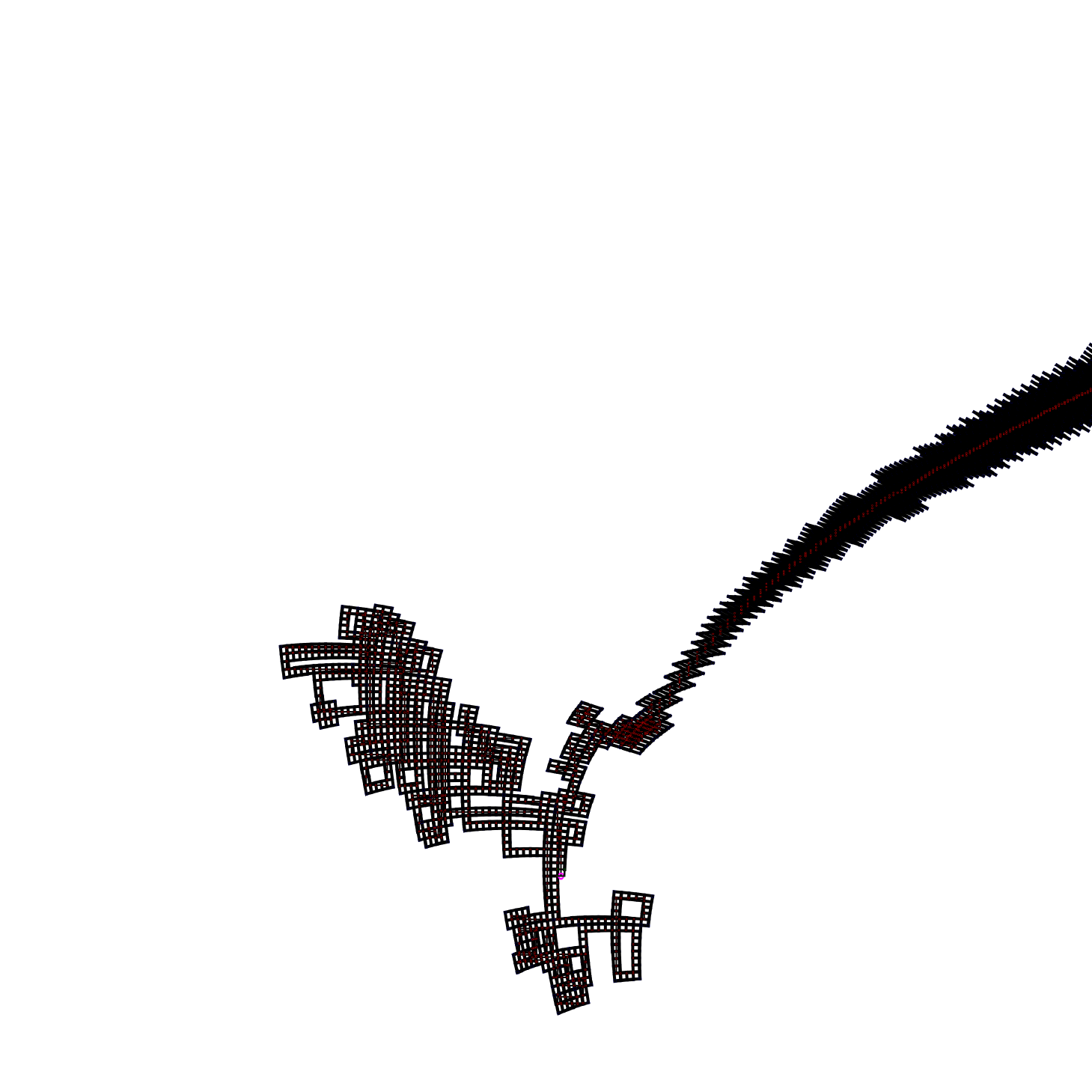}
	\end{minipage}
	\caption{Various different views of the game played on the Euclidean tiling and non-Euclidean tiling $T_{K}$ where $K = (1/4,1/4)$.  In the top row one can see the dynamics are chaotic with trajectories returning and leaving the same tiles multiple times.  In the second row is a zoomed in section of the first row where the trajectories accumulate for many iterations.  This illustrates a failed attempt at constructing a periodic orbit.  The last row is zoomed out sufficiently far to witness eventual divergence.  These images have $5,000$ iterations of the game.  This sort of meandering before divergence is witnessed best with pairs of tilings `close' to one another.  
		}\label{fig:symtil7} 
\end{figure}

While experimentation revealed orbits generically diverged, we were unable to prove anything conclusive.  Moreover the existence of periodic orbits in the presence of a non-Euclidean torus remains mysterious.  The obvious methods of starting with a pair of Euclidean tiles which one knows to admit periodic orbits and then deforming one of the tilings failed.  This is appreciably seen in Figure \ref{fig:symtil7} in the middle row.  The authors did not find the time to investigate the following method in more details, but it is possible that there are parings $T_{k}$ and $T_{K}$ so that the edges around a pair of anchors form a compatible pair of sunbursts as mentioned in Section \ref{sec:intro}.  A bit more precisely, Schwartz introduced definitions for sunbursts which he called \emph{balanced} and \emph{regular}.  A balanced sunburst is one where one can find unit vectors representing the rays whose sum is the zero vector and a regular sunburst is one whose rays are parallel to the $n$-roots of unity.  In our case, necessarily one would have to consider $n = 4$.  It seems possible that there is an anchor in a non-Euclidean tiling that could be balanced, regularity however seems far less likely outside the locus where $k_{1}k_{2} = 0$.  But if one were to find such a pair of compatible sunbursts, then Schwartz's arguments guarantee the existence of a periodic orbit up to a rotation of the tiling [Theorem 1.1] \cite{Schwartz2025Symplectic}.  
\\
\\
Finally, we emphasize that our Theorem applies only to rational tori.  It would be very interesting to say something about the irrational case.  A good starting place would be to investigate when $K \in \Q(\sqrt{d})$ for some non-perfect square natural number $d$.  Exact rational arithmetic can still be readily conducted in this instance as every element will be of the form $a+b\sqrt{d}$ for $a,b\in \Q$ and inversion and multiplication are easily expressible in rational coordinates. 
\\
\\
With the Euclidean and non-Euclidean case written out carefully we move to prove the divergence in the case where the tilings are both non-Euclidean and rational.  The ideas are exactly the same as in Theorem \ref{thm:eucdiv}.  In fact, in some sense this situation is easier as the shapes of the tiles become smaller in both tilings unlike the Euclidean case where the edge directions of the Euclidean tiling determines the same four directions for every tile.

\begin{theorem}\label{thm:noneucdiv}
Let $(T_{k}, T_{K})$ be a pair of transverse non-Euclidean marked complete \emph{rational} affine tilings where the invariant vector fields $X_{T_{k}}$ and $X_{T_{K}}$ are not parallel.  Then there exist marked tiles $P$ and $Q$ and oriented edges of these tiles so that for any configuration along these edges which is defined for all time diverges along their respective one-forms $\omega_{T_{k}}$ and $\omega_{T_{K}}$.  More precisely, if $\{(p_{n}, q_{n})\}_{n=0}^{\infty} \subset \del T_{k}\times \del T_{K}$ denotes the sequence of points of the symplectic tiling billiards trajectories, then we have the following asymptotic behavior. 
\begin{equation}\label{eq:noneucasym}
\lim_{n\to\infty} \int_{p_{n}-p_{0}} \omega_{T_{k}} =  \lim_{n\to\infty} \int_{q_{n}-q_{0}} \omega_{T_{K}} = \infty
\end{equation}
\end{theorem}

\begin{proof}
As observed in the proof of Theorem \ref{thm:eucdiv}, we may apply a simultaneous affine transformation to both $T_{k}$ and $T_{K}$ to obtain two new tilings where the dynamics are equivalent.  To this end, observe $(X_{T_{k}}, X_{T_{K}})$ is a basis for $\R^{2}$ and without loss of generality, let us assume it is oriented so there is an orientation preserving linear map taking $X_{T_{k}}$ to $e_{1}$ and $X_{T_{K}}$ to $e_{2}$.  The tilings obtained by applying this map simultaneously to both are still transverse and have the property that the invariant vector fields are $e_{1}$ and $e_{2}$ respectively.  We may thus reduce to the case where $T_{k}$ has $X_{T_{k}} = e_{1}$ and $T_{K}$ has $X_{T_{K}} = e_{2}$.  
\\
\\
As done in the proof of Theorem \ref{thm:eucdiv}, let $N_{1} := |a_{1}|+|b_{1}|$ and $N_{2} := |a_{2}| + |b_{2}|$ where $a_{1}k_{1} + b_{1}k_{2} = 0$ and $a_{2}K_{1} + b_{2}K_{2} = 0$ are the minimum values over $\Z$ satisfying these equations for $k = (k_{1},k_{2}$ and $K = (K_{1}, K_{2})$ in $\R^{2}$.  Let $M := \text{lcm}\{N_{1}, N_{2}\}$.   \\
\\
Pick small open intervals of $[e_{1}], [e_{2}] \in \fD$ in the circle of directions and close them under their negatives to get small open (disconnected) neighborhoods about both $[\pm e_{1}], [\pm e_{2}] \in \fD$.  Call these open neighborhoods $[\pm e_{1}] \in U_{1}$ and $[\pm e_{2}] \in U_{2}$ respectively.  By picking a sequence of loops so that $\lim_{n\to\infty}\omega_{T_{k}}(T_{T_{k}}(\gamma_{n}))$ tends to infinity, we can find a marked tile $P$ of $T_{k}$ so that its edge directions satisfy $[v_{1}^{+}], [v_{2}^{+}], [v_{1}^{-}], [v_{2}^{-}] \in U_{1}$ where $[v_{1}^{+}], [v_{2}^{+}]$ are very close to $[e_{1}]$ and $[v_{1}^{-}], [v_{2}^{-}]$ are very close to $[-e_{1}]$.  Similarly we may do the same for $T_{K}$ to find a marked tile $Q$ so that its edge directions satisfy $[w_{1}^{+}], [w_{2}^{+}], [w_{1}^{-}], [w_{2}^{-}] \in U_{2}$ with $[w_{1}^{+}], [w_{2}^{+}]$ close to $[e_{2}]$ and $[w_{1}^{-}], [w_{2}^{-}]$ close to $[-e_{2}]$.  See Figure \ref{fig:bothpqthin} below which illustrates this set up.
\\
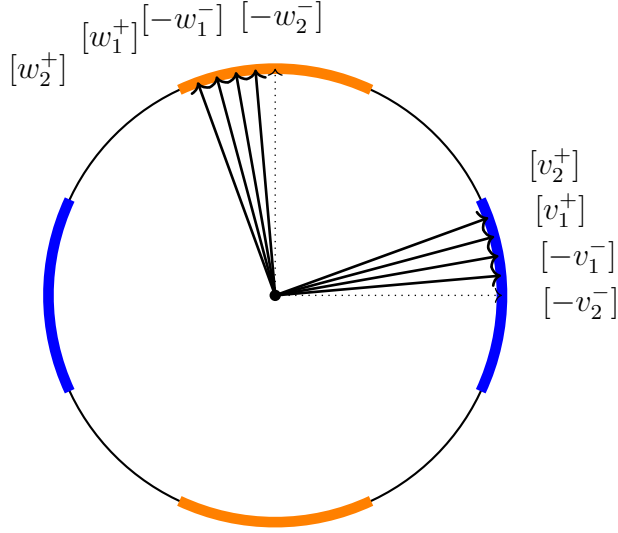
\begin{figure}
\begin{center}
\begin{tikzpicture}[scale=3]

	\def\eps{5}

  	\draw[thick] (0,0) circle (1);
	
  	\draw[line width=4pt, oskyblue]
    		({cos(-5*\eps)},{sin(-5*\eps)})
    		arc[
      			start angle=-5*\eps,
      			end angle=5*\eps,
      			radius=1
    		];
		
  	\draw[line width=4pt, oskyblue]
    		({cos(180-5*\eps)},{sin(180-5*\eps)})
    		arc[
      			start angle=180-5*\eps,
      			end angle=180+5*\eps,
      			radius=1
    		];

  	\draw[line width=4pt, oorange]
    		({cos(90-5*\eps)},{sin(90-5*\eps)})
    		arc[
      			start angle=90-5*\eps,
      			end angle=90+5*\eps,
      			radius=1
    		];
		
  	\draw[line width=4pt, oorange]
    		({cos(270-5*\eps)},{sin(270-5*\eps)})
    		arc[
      			start angle=270-5*\eps,
      			end angle=270+5*\eps,
      			radius=1
    		];
		
  	\draw[dotted, line width=0.5pt, ->]
    		(0,0) -- ({cos(0)},{sin(0)});
		
  	\draw[line width=1pt, ->]
    		(0,0) -- ({cos(0+\eps)},{sin(0+\eps)});
    		\node[xshift=30pt, yshift=-10pt] at ({cos(\eps)},{sin(\eps)}) {$[-v_2^{-}]$};
		
  	\draw[line width=1pt, ->]
    		(0,0) -- ({cos(2*\eps)},{sin(2*\eps)});
    		\node[xshift=30pt, yshift=0pt] at ({cos(2*\eps)},{sin(2*\eps)}) {$[-v_1^{-}]$};
		
  	\draw[line width=1pt, ->]
    		(0,0) -- ({cos(3*\eps)},{sin(3*\eps)});
		\node[xshift=25pt, yshift=10pt] at ({cos(3*\eps)},{sin(3*\eps)}) {$[v_1^{+}]$};

  	\draw[line width=1pt, ->]
    		(0,0) -- ({cos(4*\eps)},{sin(4*\eps)});
    		\node[xshift=25pt, yshift=+20pt] at ({cos(4*\eps)},{sin(4*\eps)}) {$[v_2^{+}]$};

  	\draw[dotted, line width=0.5pt, ->]
    		(0,0) -- ({cos(90)},{sin(90)});
		
  	\draw[line width=1pt, ->]
    		(0,0) -- ({cos(90+\eps)},{sin(90+\eps)});
    		\node[xshift=+10pt, yshift=20pt] at ({cos(90+\eps)},{sin(90+\eps)}) {$[-w_2^{-}]$};
		
  	\draw[line width=1pt, ->]
    		(0,0) -- ({cos(90+2*\eps)},{sin(90+2*\eps)});
    		\node[xshift=-20pt, yshift=20pt] at ({cos(90+2*\eps)},{sin(90+2*\eps)}) {$[-w_1^{-}]$};
		
  	\draw[line width=1pt, ->]
    		(0,0) -- ({cos(90+3*\eps)},{sin(90+3*\eps)});
		\node[xshift=-40pt, yshift=15pt] at ({cos(90+3*\eps)},{sin(90+3*\eps)}) {$[w_1^{+}]$};

  	\draw[line width=1pt, ->]
    		(0,0) -- ({cos(90+4*\eps)},{sin(90+4*\eps)});
    		\node[xshift=-60pt, yshift=5pt] at ({cos(90+4*\eps)},{sin(90+4*\eps)}) {$[w_2^{+}]$};

  	\fill (0,0) circle (0.025);

\end{tikzpicture}

\caption{A schematic of choosing tiles $P$ and $Q$ from $T_{k}$ and $T_{K}$ where $Q$ is $\omega_{T_{K}}$-positively $P$-thin and $P$ is $\omega_{T_{k}}$-positively $Q$-thin.  A direction is $\omega_{T_{k}}$-positive if it lies in the upper-half semi-circle of directions, and a direction is $\omega_{T_{K}}$-positive if it lies in the left-half semi-circle of directions.  This figure illustrates how the $[w_{i}^{+}]$ and $[-w_{i}^{-}]$ are all $\omega_{T_{k}}$-positive, and they all satisfy $\det(v_{i}^{+},\cdot ) > 0$ showing $P$ is $\omega_{T_{k}}$-positively $Q$-thin.  By taking the negatives of the illustrated $v$'s, one can see that $Q$ is $\omega_{T_{K}}$-positively $P$-thin.
}\label{fig:bothpqthin}
\end{center}
\end{figure}

We claim $P$ is $\omega_{T_{k}}$-positively $Q$-thin.  One can choose the directions $[f_{1}], [f_{2}], [f_{3}], [f_{4}]$ representing the edges of $Q$ to be $[w_{1}^{+}], [w_{2}^{+}], [-w_{1}^{-}], [-w_{2}^{-}]$.  By construction all of these directions are $\omega_{T_{k}}$-positive because they lie in the upper semicircle in $\fD$ close to $[e_{2}]$.  Moreover, for the two non-negative edges $v_{i}^{+}$ of $P$ close to $[e_{1}]$, $\det(v_{i}^{+}, f_{j}) > 0$ for $j=1,\hdots,4$ and for the non-positive edges $v_{i}^{-}$, $\det(v_{i}^{-}, f_{j}) < 0$ for $j=1,\hdots,4$ thus $P$ is $Q$-thin.  By similar arguments, there are choices of edge directions $[d_{1}], [d_{2}], [d_{3}], [d_{4}]$ representing the edges of $P$ so that $Q$ is $\omega_{T_{K}}$-positively $P$-thin with respect to these directions.\\
\\ 
Using Lemma \ref{lem:getthin}, we may choose tiles $P'$ and $Q'$ even further away so that for all $\gamma = (n,m)$ satisfying $|n| + |m| < 2M$, we have $\gamma P'$ is $Q$-thin and $\gamma Q'$ is $P$-thin.  In particular this can be done so all the edge directions of $\gamma P'$ lie in $U_{1}$ closer to $[\pm e_{1}]$ than the edges of $P'$, and, all the edge directions of $\gamma Q'$ lie in $U_{2}$ closer to $[\pm e_{2}]$ than the edges of $Q'$, for all such $\gamma = (n,m)$ satisfying $|n| + |m| < 2M$ as in Figure \ref{fig:getsmaller}.\\
\\
Now pick either $\omega_{T_{k}}$-non-negative edge of $P'$ and either $\omega_{T_{K}}$-non-negative edge of $Q'$.  Pick a pair of interior points on these edges $(p_{0},q_{0})$ so that all forward iterations are well-defined and label them $\{(p_{n},q_{n})\}_{n=0}^{n=\infty}$ as was done before.  We proceed as in the proof of Theorem \ref{thm:eucdiv}.  By construction $P'$ is $\omega_{T_{k}}$-positively $Q'$-thin, as the edge directions of $Q'$ are closer to $[\pm e_{2}]$ than $Q$'s.  The tile $Q'$ has four $\omega_{T_{k}}$-positive directions $[f_{1}], [f_{2}], [f_{3}], [f_{4}]$, along which any trajectory starting at $p_{0}$ must intersect $\del P'$ at a non-positive edge by Lemma \ref{lem:intersectedge}.  This new tile $\alpha_{1}P'$ contains this intersection point $p_{1}$, and, is now contained in a non-negative edge.  \\
\\
On $Q'$ we begin at $q_{0}$.  By construction $Q'$ is $\omega_{T_{K}}$-positively $\alpha_{1}P'$-thin, as the edge directions of $\alpha_{1}P'$ are closer to $[\pm e_{1}]$ than $P$'s.  The tile $\alpha_{1}P'$ has four $\omega_{T_{K}}$-positive directions $[d_{1}], [d_{2}], [d_{3}], [d_{4}]$, along which any trajectory starting at $q_{0}$ must intersect $\del Q'$ at a non-positive edge by Lemma \ref{lem:intersectedge}.  This new tile $\beta_{1}Q'$ contains this intersection point $q_{1}$, and, is now contained in a non-negative edge. 
\\
\\
This process gives us our new pair $(p_{1},q_{1})$ in the tiles $\alpha_{1}P'$ and $\beta_{1}Q'$.  By applying the same arguments we may iterate this process to obtain pairs $(p_{n},q_{n})$ while only moving through positive directions relative to each tiling.  As was argued in Theorem \ref{thm:eucdiv}, we may do this in such a way that the $T_{k}$-tiles remain positively $\beta_{i}Q'$-thin, and the $T_{K}$-tiles remain positively $\alpha_{i}P'$-thin for at least $2M$-iterations.  By applying Lemma \ref{lem:positivemoves} to both $T_{k}, T_{K}$, after at most $M < 2M$-iterations, there will be a pair of loops $(\alpha_{i},\beta_{i})$ for which both pairs of edge directions of $\alpha_{i}P'$ and $\beta_{i}Q'$ move closer to $[\pm e_{1}]$ and $[\pm e_{2}]$, or, remain the same.  In either case, this new pair of marked tiles $\alpha_{i}P'$ and $\beta_{i}Q'$ will still be  $\omega_{T_{k}}$-positively $Q'$-thin and $\omega_{T_{K}}$-positively $P'$-thin respectively.  Because the edge directions are closer to their respective limits, they too satisfy this property being stable after $2M$-iterations.  \\
\\
Thus the sequences $p_{n}$, $q_{n}$ only ever move in their respective positive directions and are contained in cones.  By the same arguments in Theorem \ref{thm:eucdiv}, this means the sequences diverge and satisfy the asymptotic divergence as in Equation \ref{eq:noneucasym}.  Because no hypothesis was made on the points on the non-negative edges, this divergence is generic amongst pairs of non-negative edges and tiles satisfying the hypotheses in the beginning, and thus, stable.  
\end{proof}

We conclude this work with remarks and illustrations of Theorem \ref{thm:noneucdiv}.  As noted, the techniques used to prove this are the same ideas as in Theorem \ref{thm:eucdiv}, and thus suffer from the same inability to account for what happens in the irrational case, and, a lack of showing the behavior to be generic.  In some sense both theorems are a stability result about positive $P$-thinness and in essence boil down to formally defining what it means for two polygons to be thin relative to one another in the affine context.  Below in Figure \ref{fig:symtil5} we illustrate an example of two tiles which are thin relative to one another and the divergence of their trajectories in Figure \ref{fig:symtil6} as guaranteed by Theorem \ref{thm:noneucdiv}.  As was the case for the Euclidean and non-Euclidean tiling billiards, we observed that experimentally all orbits eventually diverge, however the trajectories are frequently parabolic as seen in Figure \ref{fig:symtil10}.  This behavior was not observed in the Euclidean and non-Euclidean tile pairing.  
\\
\begin{figure}
	\begin{minipage}{0.45\textwidth}
		\includegraphics[width=\textwidth]{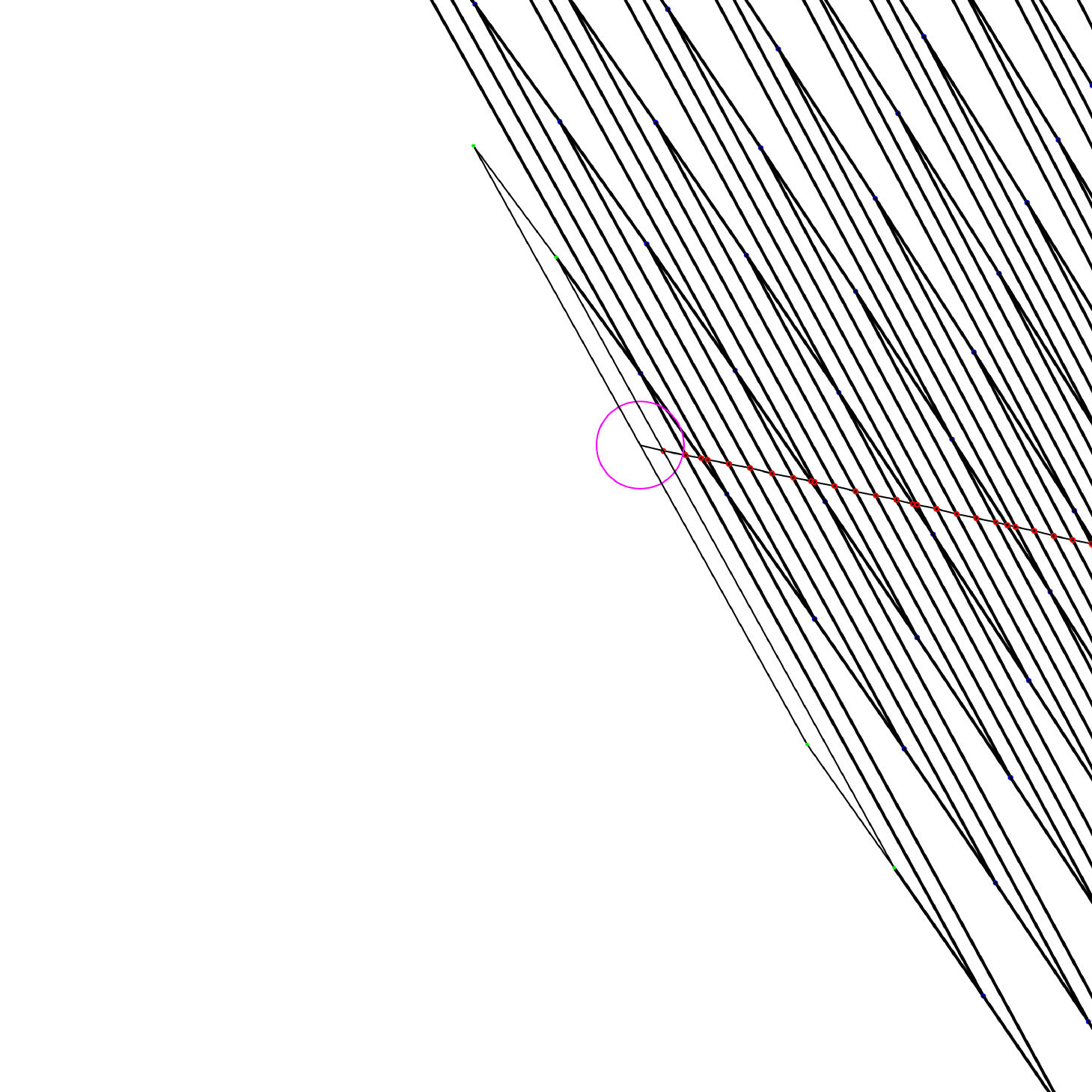}
	\end{minipage}
	\hfill
	\begin{minipage}{0.45\textwidth}
		\includegraphics[width=\textwidth]{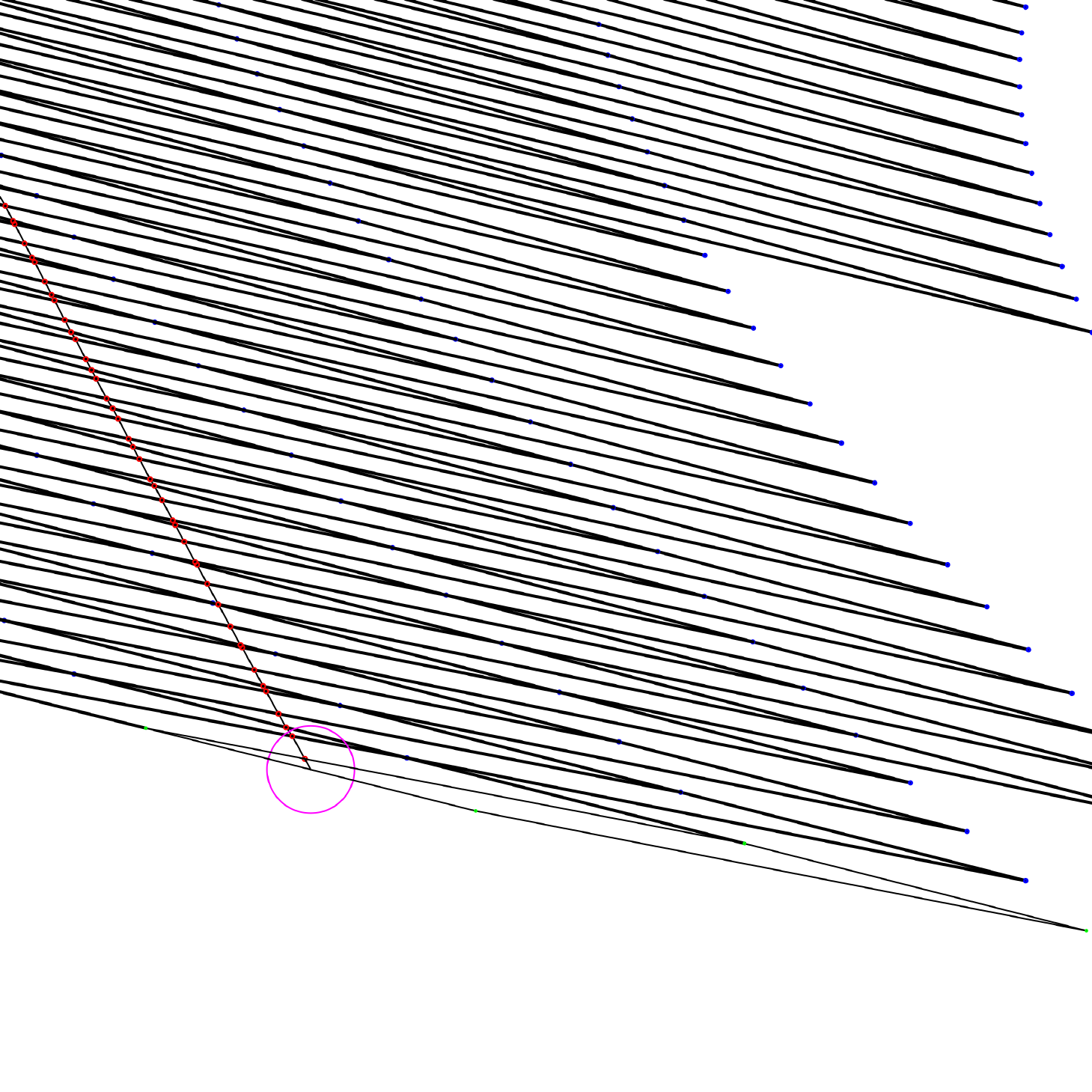}
	\end{minipage}
	\caption{Examples of tilings $T_{k},T_{K}$ where the marked tiles $P$ and $Q$ are positively thin relative to one another as in hypotheses of Theorem \ref{thm:noneucdiv}.  Here the period parameters are given by $k = (2/3,1/4)$ and $K = (1/4,2/3)$.  Both particles are configured on non-negative edges.}  
	\label{fig:symtil5} 
\end{figure}

\begin{figure}
	\begin{minipage}{0.45\textwidth}
		\includegraphics[width=\textwidth]{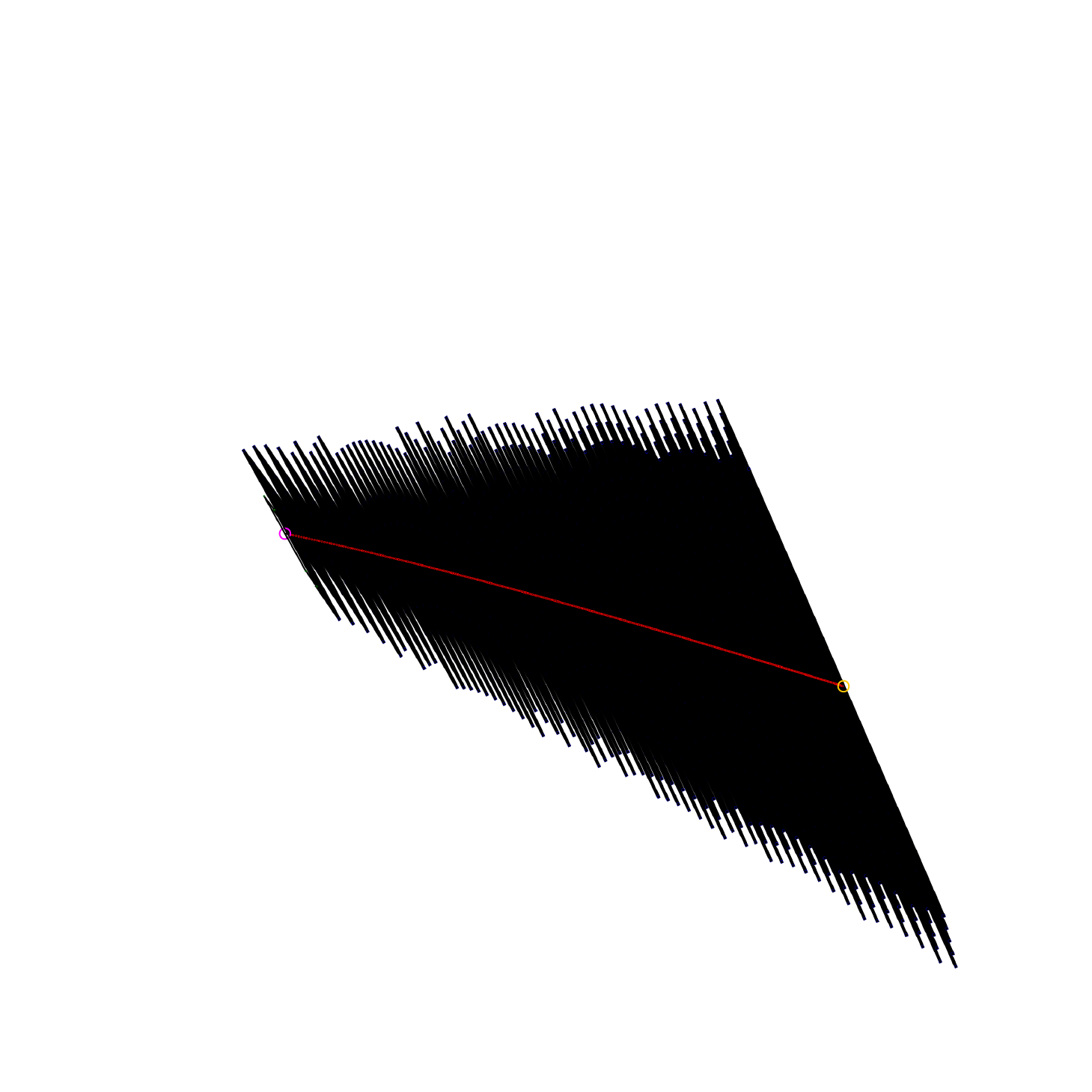}
	\end{minipage}
	\hfill
	\begin{minipage}{0.45\textwidth}
		\includegraphics[width=\textwidth]{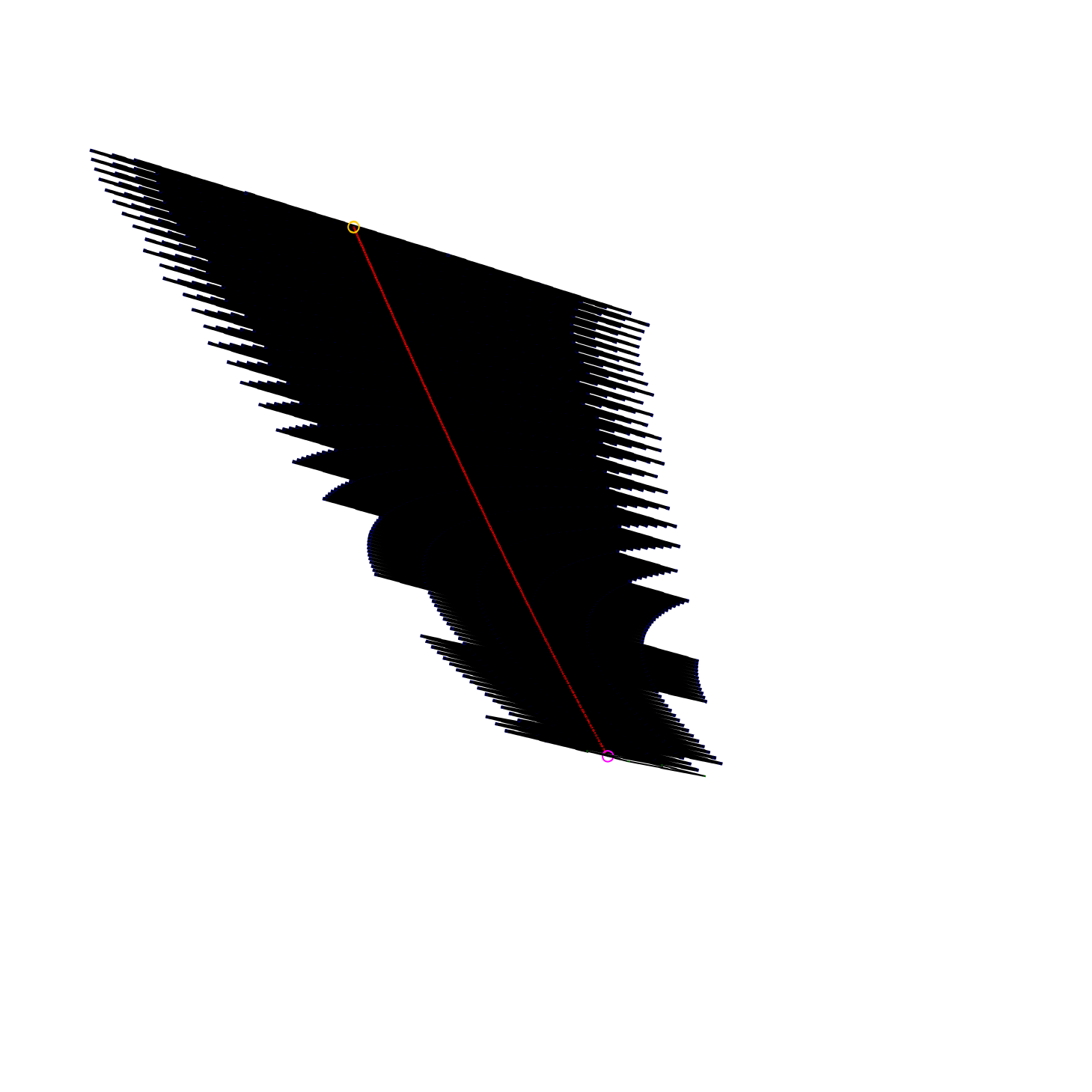}
	\end{minipage}
	\caption{The same setting as in Figure \ref{fig:symtil5} but zoomed out.  The initial points are drawn with pink circles and the terminal points are drawn with orange ones.  One can see the orbits satisfy the conditions of Equation \ref{eq:noneucasym} as the invariant vector fields are directed towards $[X_{k}] = [(1/4,-2/3)]$ and $[X_{K}] = [(2/3,-1/4)]$ respectively.  Both of these trajectories can be seen to lie on the half planes through $p_{0}$ and $q_{0}$ defined by these vectors respectively, and thus asymptotically, diverge along their respective one forms.}  
	\label{fig:symtil6} 
\end{figure}

\begin{figure}
	\begin{minipage}{0.45\textwidth}
		\includegraphics[width=\textwidth]{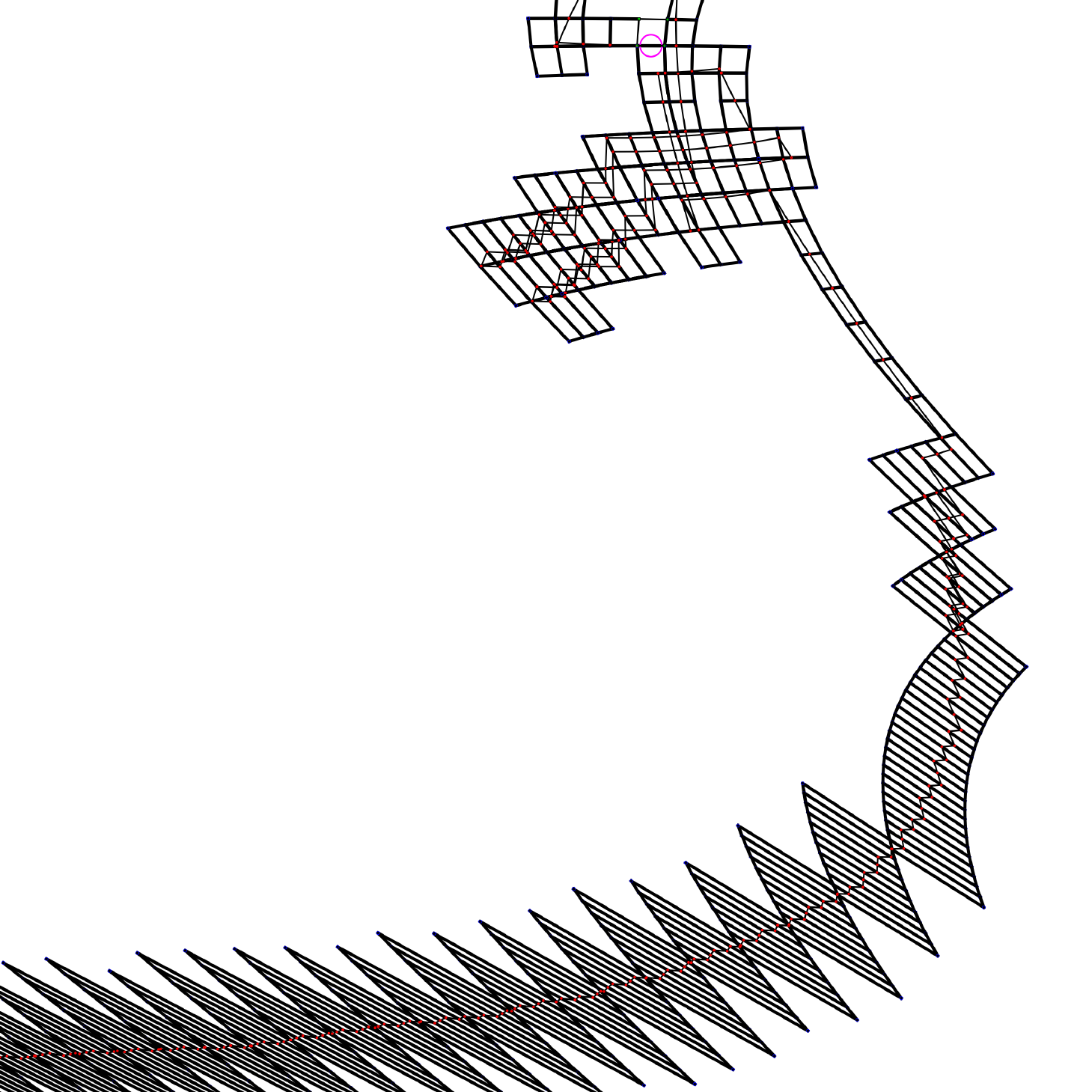}
	\end{minipage}
	\hfill
	\begin{minipage}{0.45\textwidth}
		\includegraphics[width=\textwidth]{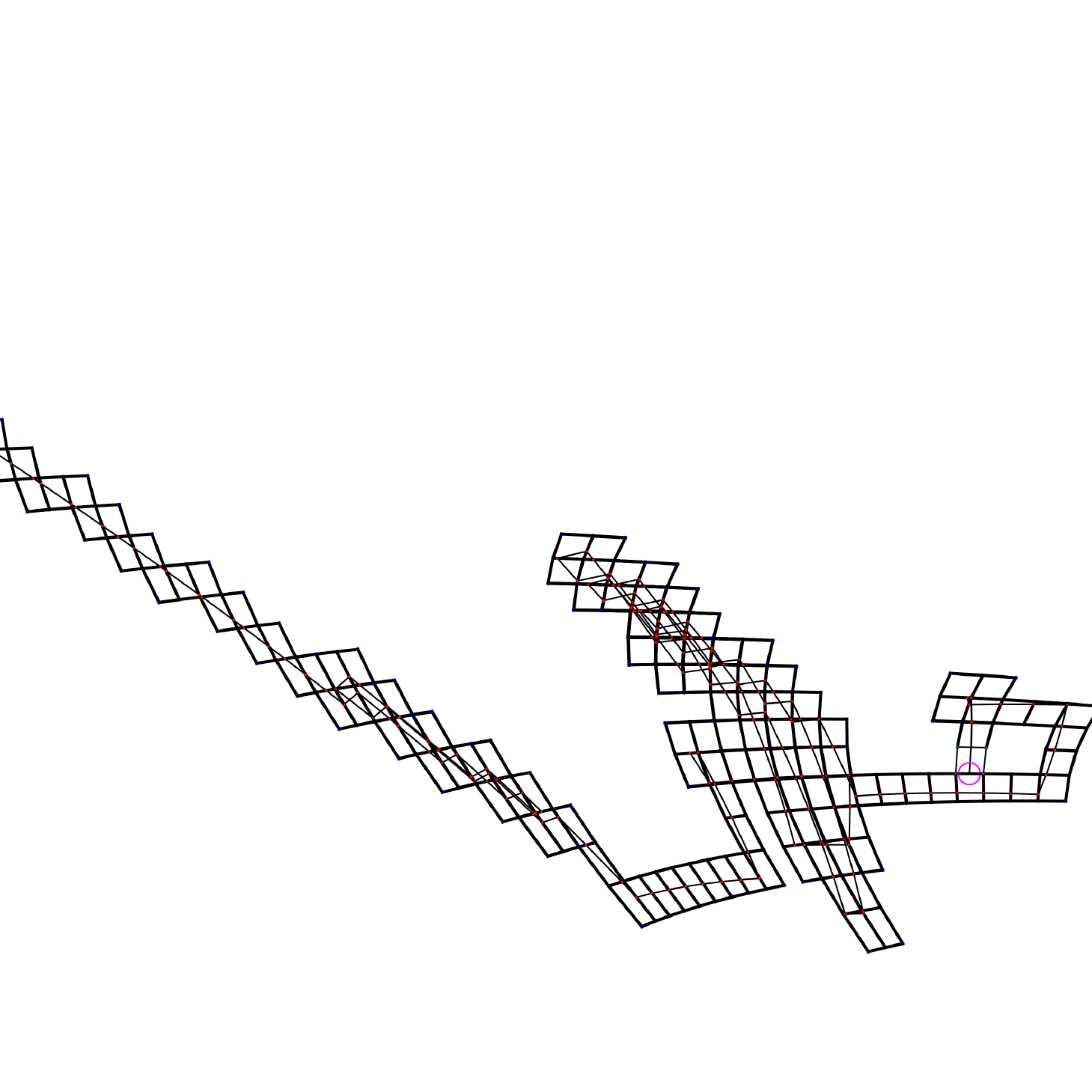}
	\end{minipage}
	\caption{Trajectories for the tiles $T_{k}, T_{K}$ where $k = (1/6,1/2)$ and $K = (1/5,1/2)$.  After about 400 iterations one can appreciate the trajectory of $T_{k}$ seems to resemble a parabola.}  
	\label{fig:symtil10} 
\end{figure}
  
It is unclear to the authors what can be said about the direction of the trajectories in the plane apart from their divergence.  Our numerical analyses do not conclude they limit to any of the obvious invariants such as flows along the invariant parallel or polynomial vector fields in $\R^{2}$.  One can however say on the torus, each individual iteration becomes more parallel to the canonical flow.  This can be seen through the following arguments.  Focus our attention on the tiling $T_{K}$.  Because the trajectories of $T_{K}$ diverge as in Equation \ref{eq:noneucasym}, the tiles flatten out and their edges become more parallel to the canonical directions $[\pm X_{T_{K}}] \in \fD$.  Moreover, the possible edge directions $[v_{i}]$ from the tiles of $T_{k}$ lie in a compact subset of $\fD$ which under a sequence of divergent positive loops must satisfy $L_{T_{K}}(\gamma_{n})[v_{i}] = [\pm X_{T_{K}}]$ by Lemma \ref{lem:getthin}.  Consequently on the torus defined by $T_{K}$, the linear holonomy $L_{T_{K}}$ takes this compact subset of directions closer to $[\pm X_{T_{K}}]$, and thus the trajectories on the torus tend closer to the canonical flow, but, never reach it unless they were parallel to begin with, which by the hypotheses imposed on our tilings as in Theorem \ref{thm:goodtiles}, does not happen.


\bibliographystyle{amsplain}
\bibliography{STBCT.bib}

\end{document}